\documentclass[12pt]{article}
\usepackage[utf8]{inputenc}
\usepackage{mathtools}                                  
\usepackage{booktabs}
\usepackage[english]{babel} 
\usepackage[protrusion=true,expansion=true]{microtype} 
\usepackage{amsmath,amsfonts,amsthm}
\usepackage{ amssymb }
\usepackage{graphicx}
\usepackage{array}
\usepackage{multirow}
\usepackage{hhline}
\usepackage[titletoc]{appendix}

\usepackage[margin=1in,hmarginratio=1:1,top=20mm,columnsep=20pt]{geometry} 
\usepackage{booktabs} 
\usepackage{natbib}
\usepackage{setspace}

\usepackage{microtype} 

\usepackage{graphicx}
\usepackage{subfigure}

\usepackage{tabularx}
\usepackage[font = small,labelfont=bf,textfont=it]{caption} 
\usepackage{footnote}
\usepackage{mathrsfs}
\usepackage{algorithm}
\usepackage[noend]{algpseudocode}

\newtheorem{lemma}{Lemma}[section]
\newtheorem{defn}{Definition}[section]

\newtheorem{prop}{Proposition}[section]
\newtheorem{corollary}{Corollary}[section]
\newtheorem{remark}{Remark}[section]
\newtheorem{theorem}{Theorem}[section]

\newtheorem{setting}{Settings}[section]

\newtheorem{assumption}{Assumption$^*$}[section]

\newtheorem{condition}{Condition}[section]

\newcommand{\RR}{\mathbb{R}}

\newcommand{\argmin}{\operatorname*{argmin}}

\newcommand{\vertiii}[1]{{\left\vert\kern-0.25ex\left\vert\kern-0.25ex\left\vert #1 
    \right\vert\kern-0.25ex\right\vert\kern-0.25ex\right\vert}}

\theoremstyle{definition}

\begin{document}

\title {On the Inference of Applying Gaussian Process Modeling to a Deterministic Function}
\author{Wenjia Wang}
\date{}
\maketitle
\vspace{-5mm}

\begin{abstract}
Gaussian process modeling is a standard tool for building emulators for computer experiments, which are usually used to study deterministic functions, for example, a solution to a given system of partial differential equations. This work investigates applying Gaussian process modeling to a deterministic function from prediction and uncertainty quantification perspectives, where the Gaussian process model is misspecified. Specifically, we consider the case where the underlying function is fixed and from a reproducing kernel Hilbert space generated by some kernel function, and the same kernel function is used in the Gaussian process modeling as the correlation function for prediction and uncertainty quantification. While upper bounds and the optimal convergence rate of prediction in the Gaussian process modeling have been extensively studied in the literature, a comprehensive exploration of convergence rates and theoretical study of uncertainty quantification is lacking. We prove that, if one uses maximum likelihood estimation to estimate the variance in Gaussian process modeling, under different choices of the regularization parameter value, the predictor is not optimal and/or the confidence interval is not reliable. In particular, lower bounds of the prediction error under different choices of the regularization parameter value are obtained. The results indicate that, if one directly applies Gaussian process modeling to a fixed function, the reliability of the confidence interval and the optimality of the predictor cannot be achieved at the same time.
\end{abstract}

\section{Introduction}\label{secintro}
Computer experiments are often used to study a system of interest. For example, \cite{Mak2018} studies a complex simulation model for turbulent flows in swirl injectors. Other examples include \cite{burchell}, who estimates sexual transmissibility of human papillomavirus infection, and \cite{moran}, who uses the Cardiovascular Disease Policy Model to project cost-effectiveness of treating hypertension in the U.S. according to 2014 guidelines. In these examples, the simulators are expensive, and the inputs/responses pairs are often not available for an extensive exploration of the underlying function. One well-established approach for solving this problem is the use of an \textit{emulator}, which is an inexpensive approximation for the simulator. 

In computer experiments, the two major problems are prediction and uncertainty quantification. Gaussian process modeling, which is a widely used method in computer experiments, naturally enables prediction and statistical uncertainty quantification. In the Gaussian process modeling, the underlying function is assumed to be a realization of a Gaussian process. Based on the Gaussian process assumption, the conditional distribution can be constructed at each unobserved point in a region, which provides a natural predictor via conditional expectation and a pointwise confidence interval. The pointwise confidence intervals can be used for statistical uncertainty quantification.

However, in practice, the Gaussian process is usually misspecified. The responses of a computer model usually come from a deterministic function which may not be a sample path of the Gaussian process used in Gaussian process modeling, or may come from a smaller function space that has probability zero \citep{ba2012composite,bower2010parameter,higdon2002space,tan2018gaussian,xu2019modeling}. For example, a function in an infinite-dimensional reproducing kernel Hilbert space is typically smoother than a sample path of the corresponding Gaussian process \cite{steinwart2019convergence}. Here, the word ``\textit{corresponding}'' means that the covariance function of the Gaussian process and the kernel function of the reproducing kernel Hilbert space are the same, up to a constant multiplier. Therefore, if one applies Gaussian process modeling to a function in the corresponding reproducing kernel Hilbert space, a model misspecification issue occurs.

Despite this model misspecification, 
maximum likelihood estimation is commonly used to estimate unknown parameters in the covariance function within the Gaussian process model \citep{santner2013design}. Applying maximum likelihood estimation when a model is misspecified may be problematic because the estimated parameter can diverge as the sample size goes to infinity. For example, \cite{xu2017maximum} shows that if the underlying function is $f(x) = x^\gamma$ on $[0,1]$ and a Gaussian correlation function is used in the Gaussian process modeling, the estimated variance can either go to zero or infinity as the sample size increases to infinity. Another question is: when the Gaussian process model is misspecified, are the confidence intervals with estimated parameters reliable?
In practice, it is often observed that Gaussian process models have poor coverage of their confidence intervals \citep{gramacy2012cases,joseph2011regression,yamamoto2000alternative}. One possible reason is that the Gaussian process model may be misspecified; thus the confidence intervals may be inadequate for quantifying the uncertainty of predictions.

In this work, we investigate the prediction and confidence intervals in misspecified Gaussian process models used to recover deterministic functions from a frequentist view, i.e., assuming the underlying function is fixed but unknown. This is different from the Bayesian perspective, where a Gaussian process prior on the function space is induced. Specifically, we consider the following settings.

\begin{setting}\label{modelsetting}
    The underlying deterministic function $f$ is fixed and lies in a reproducing kernel Hilbert space. A Gaussian process model is applied for prediction and uncertainty quantification. The kernel function of the reproducing kernel Hilbert space and the correlation function in the Gaussian process model are the same.
\end{setting}
In Settings \ref{modelsetting}, we assume that the variance in the Gaussian process model is unknown and needs to be estimated. For more on the formal settings considered in this work and more discussion, see Section \ref{subsec:psetting}. As stated before, the model misspecification occurs under Settings \ref{modelsetting}. This model misspecification does not change the form of the predictor (which is one reason that the Gaussian process model is typically misspecified), but significantly changes the uncertainty quantification results. We consider two cases, one case is that the observations have no noise; the other is that the observations have noise. When the observations have no noise, we show that if an estimated variance obtained by maximum likelihood estimation is used in the confidence intervals, then
the confidence intervals are not reliable. Here, the reliability is used in the sense that is to be introduced later; see Section \ref{subsecofCI}. This suggests that, if the Gaussian process model is misspecified, the confidence interval needs to be carefully constructed to quantify the uncertainties---not merely derived by the corresponding Gaussian process model.

In many cases, computer experiments are stochastic, in the sense that stochastic errors are introduced to simulate the randomness in real systems. A recent overview of stochastic emulators is \cite{baker2020stochastic}. In stochastic computer experiments, a regularization parameter is used to counteract the noise's influence. The value of this regularization parameter is usually a constant \citep{ ankenman,baker2020stochastic,dancik2007mlegp}. It is known that the optimal convergence rate for non-parametric regression is determined by the smoothness of the underlying function, denoted by $\nu$. The optimal convergence rate is $n^{-\frac{\nu}{2\nu +d}}$, where $n$ is the sample size, and $d$ is the dimension of the input space \citep{stone1982optimal}. In this work, we show that if the regularization parameter value is chosen to be a constant, the corresponding predictor is not optimal, in the sense that the convergence rate of the prediction error is of a higher order than the optimal rate. Furthermore, with an estimated variance obtained by maximum likelihood estimation, we show that under different choices of the regularization parameter value (not restricted to be a constant), the corresponding predictor is not optimal, or the confidence interval is not reliable. We also derive some lower bounds on the convergence rates of the prediction error of Gaussian process modeling. These results suggest that we may lose the prediction efficiency or reliability of uncertainty quantification if the Gaussian process model is misspecified and maximum likelihood estimation is used.

The rest of this paper is arranged as follows. In Section \ref{secPre}, we introduce Gaussian process modeling, maximum likelihood estimation, and reproducing kernel Hilbert spaces, as well as our definition of reliability of confidence intervals. The main results of this work are also summarized in Section \ref{secPre}. In Sections \ref{secdeter} and \ref{unreliableCIsec}, we present the main results of this work, under the case where observations have no noise and the case where observations have noise, respectively. Simulation studies are reported in Section \ref{secnum}. Conclusions and discussion are made in Section \ref{secdiscuss}. The technical proofs are given in Appendix.

\section{Preliminaries}\label{secPre}
This section provides a brief introduction to Gaussian process modeling, maximum likelihood estimation, and reproducing kernel Hilbert spaces, which are used in developing the main results. We also provide our definition of the reliability of confidence intervals. Problem settings and a summary of the main results are presented at the end of this section. 

\subsection{Gaussian process modeling}\label{secGPmodel}
In this work, we consider applying Gaussian process modeling to a fixed function $f$, defined on a convex and compact set\footnote{
This condition can be relaxed to a compact set satisfying interior cone condition and with Lipschitz boundary; see \cite{adams2003sobolev,wendland2004scattered} for discussion of these conditions. In fact, the compactness and convexity imply the interior cone condition and Lipschitz boundary; see \cite{hofmann2007geometric,niculescu2006convex}.} $\Omega\subset \mathbb{R}^d$ with a positive Lebesgue measure. Because the domain $\Omega$ is fixed, the corresponding asymptotic framework is called fixed-domain asymptotics \citep{stein1995fixed,stein2012interpolation}. Suppose we have $n$ observed pairs $(x_k,y_k),k=1,\ldots,n$, given by
\begin{eqnarray}
y_k=f(x_k) + \epsilon_k, \label{recovering}
\end{eqnarray}
where $x_k \in \Omega$ are distinct measurement locations (i.e., $x_k\neq x_j$ for $k\neq j$) and $\epsilon_k \sim N(0,\sigma_\epsilon^2)$ are i.i.d. normally distributed random errors with variance $\sigma_\epsilon^2 \geq 0$. If the observations are not corrupted by noise, we have $\sigma_\epsilon^2 = 0$, otherwise $\sigma_\epsilon^2 > 0$. One popular method to recover the function $f$ is stationary Gaussian process modeling. Let $Z$ be a stationary Gaussian process defined on $\RR^d$. For the ease of mathematical treatment, we consider simple kriging. Therefore, we assume $Z$ has mean zero, variance $\sigma^2$ and correlation function $\Psi$, denoted by $Z\sim GP(0,\sigma^2\Psi)$. The correlation function $\Psi$ is stationary, i.e., the function value of $\Psi(x,x')$ only depends on the difference $x-x'$; thus we can write $\Psi(x-x'):=\Psi(x,x')$. We also assume $\Psi$ is strictly positive definite and integrable on $\RR^d$, and $\Psi(0)=1$. By Bochner's theorem (Page 208 of \cite{gihman1974theory}; Theorem 6.6 of \cite{wendland2004scattered}) and Theorem 6.11 of \cite{wendland2004scattered}, there exists a function $f_\Psi$ such that 
\begin{align*}
\Psi(h)=\int_{\mathbb{R}^d} e^{i \omega^T h}  f_\Psi(\omega) d \omega
\end{align*}
for any $h\in \RR^d$. The function $f_\Psi$ is known as the \textit{spectral density} of $Z$ or $\Psi$. In this work, we suppose that $f_\Psi$ decays algebraically, i.e., satisfies the following condition.
\begin{condition}\label{C1}
	There exist constants $c_2 \geq c_1>0$ and $\nu>d/2$ such that, for all $\omega\in\mathbb{R}^d$,
	$$ c_1(1+\|\omega\|_2^2)^{-\nu} \leq  f_\Psi(\omega)\leq c_2(1+\|\omega\|_2^2)^{-\nu}, $$
	where $\|\cdot\|_2$ denotes the Euclidean metric.
\end{condition}
One example of correlation functions satisfying Condition \ref{C1} is the isotropic Mat\'ern correlation function \citep{stein2012interpolation}, given by
\begin{eqnarray}\label{materngai}
\Psi_{M}(h)=\frac{1}{\Gamma(\nu)2^{\tilde \nu-1}}(2\sqrt{\tilde \nu}\phi \|h\|_2)^{\tilde \nu} K_{\tilde \nu}(2\sqrt{\tilde \nu}\phi\|h\|_2),
\end{eqnarray}
with the spectral density \citep{tuo2015theoretical}
\begin{eqnarray*}
f_\Psi(\omega;\tilde \nu,\phi)= \pi^{-d/2}\frac{\Gamma(\tilde \nu+d/2)}{\Gamma(\tilde \nu)}(4\tilde \nu \phi^2)^{\tilde \nu} (4\tilde \nu\phi^2+\|\omega\|_2^2)^{-(\tilde \nu+d/2)},
\end{eqnarray*}
where $\phi,\tilde \nu>0$, and $K_{\tilde \nu}$ is the modified Bessel function of the second kind. By setting $\nu = \tilde \nu + d/2$, we can see $\Psi_M$ satisfies Condition \ref{C1}.

Another example of correlation functions satisfying Condition \ref{C1} is the generalized Wendland correlation function \citep{bevilacqua2019estimation,chernih2014closed, gneitingstationary}, given by
\begin{align*}
    \Psi_{GW}(h) = \left\{
    \begin{array}{lc}
         \frac{1}{B(2\kappa,\mu+1)}\int_{h}^1 u(u^2-h^2)^{\kappa - 1}(1-u)^\mu du, & 0\leq h<1, \\
         0, & h\geq 1,
    \end{array}\right.
\end{align*}
where $\kappa > 0$ and $\mu \geq (d+1)/2 + \kappa$, and $B$ is the beta function. It can be shown that $\Psi_{GW}$ also satisfies Condition \ref{C1}; see Theorem 1 of \cite{bevilacqua2019estimation}.

Suppose we are interested in the value of $Z(x)$ on $\Omega$ and we observe data
\begin{align}
    z_j = Z(x_j)+\epsilon_j', j=1,...,n,
\end{align}
where $x_k \in \Omega$ are distinct measurement locations (i.e., $x_k\neq x_j$ for $k\neq j$) and $\epsilon_k' \sim N(0,\sigma_\epsilon^2)$ are i.i.d. normally distributed random errors with variance $\sigma_\epsilon^2 \geq 0$. Conditional on  $\mathcal{Z}=(z_1,...,z_n)^T$, $Z(x)$ is normally distributed at point $x$. Note that $\mathcal{Z}$ is a random vector, where the randomness is induced by $Z(x_j)$ and $\epsilon_j'$. The conditional expectation of $Z(x)$ is given by
\begin{eqnarray}\label{mean}
\mathbb{E}[Z(x)|\mathcal{Z}]=r(x)^T (R+\mu I_n)^{-1} \mathcal{Z},
\end{eqnarray}
where $r(x)=(\Psi(x-x_1),\ldots,\Psi(x-x_n))^T, R=(\Psi(x_j-x_k))_{j k}$, $I_n$ is an identity matrix, and $\mu = \sigma_\epsilon^2/\sigma^2$. 
The conditional expectation is a nature predictor of $Z(x)$, and it can be shown that the conditional expectation \eqref{mean} is the best linear unbiased predictor \citep{santner2013design,stein2012interpolation}. A predictor given by Gaussian process modeling is then the conditional expectation of $Z(x)$. 

In addition to prediction, uncertainty quantification plays an essential role in statistics. Gaussian process modeling enables statistical uncertainty quantification via confidence intervals \citep{rasmussen2006gaussian,santner2013design}. Conditional on $\mathcal{Z}$, the conditional variance of $Z(x)$ is given by 
\begin{eqnarray*}
    \text{Var}[Z(x)|\mathcal{Z}]=\sigma^2(1-r(x)^T (R+\mu I)^{-1} r(x)),
\end{eqnarray*}
where $R$, $r(x)$ and $\mu$ are as in \eqref{mean}. Let $\Phi$ denote the cumulative distribution function of the standard normal distribution $N(0,1)$ and let $q_\beta = \Phi^{-1}(1 - \beta/2)$ denote the $(1 - \beta/2)$th quantile, where $\beta \in (0,1)$. A level $(1 - \beta)100\%$ pointwise confidence interval on point $x\in \Omega$ can be constructed by
\begin{align*}
CI_{n,\beta}(x) = [\mathbb{E}[Z(x)|\mathcal{Z}] - c_{n,\beta}(x), \mathbb{E}[Z(x)|\mathcal{Z}] + c_{n,\beta}(x)],
\end{align*}
where
\begin{align}\label{cnxbetainCI}
c_{n,\beta}(x) = & q_{1-\beta/2} \sqrt{\text{Var}[Z(x)|\mathcal{Z}]}
 =  q_{1-\beta/2}  \sqrt{\sigma^2(1 - r(x)^T (R+\mu I_n)^{-1} r(x))},
\end{align}
and $r(x)$ and $R$ are as in \eqref{mean}. The confidence interval is often used in the numerical simulations to show the uncertainty quantification results in Gaussian process modeling. A few examples are \cite{ba2012composite,gramacy2015local, hung2015analysis,williams2006combining}.

Recall that we apply Gaussian process modeling to the deterministic function $f$. Therefore, we \textit{treat} the observations $Y=(y_1,...,y_n)^T$ as the observations from the Gaussian process $Z$. The predictor of $f(x)$ becomes
\begin{eqnarray}
f_n(x) := \mathbb{E}[Z(x)|y_1,\ldots,y_n]=r(x)^T (R+\mu I_n)^{-1} Y, x\in \Omega,\label{predictor}
\end{eqnarray}
and the confidence interval becomes 
\begin{align}\label{CIZx}
CI_{n,\beta}(x) = [f_n(x) - c_{n,\beta}(x), f_n(x) + c_{n,\beta}(x)],
\end{align}
where $c_{n,\beta}(x)$ is as in \eqref{cnxbetainCI}.

\subsection{Maximum likelihood estimation}
Recall that we consider the underlying function is deterministic and lies in some reproducing kernel Hilbert space (see Settings \ref{modelsetting}). The Gaussian process modeling is used for prediction and uncertainty quantification, where the correlation function is the same as the kernel function of the reproducing kernel Hilbert space. Despite the model misspecification issue, i.e., the underlying function is typically smoother than the sample path in the corresponding Gaussian process, maximum likelihood estimation \citep{rasmussen2006gaussian,santner2013design,stein2012interpolation} is often used to specify the parameters value that are not pre-determined in the covariance function $\sigma^2\Psi(\cdot - \cdot)$. Let $\Psi_\theta$ be a family of correlation functions indexed by $\theta=(\theta_1,\ldots,\theta_q)^T\in \Theta \subset \RR^q$, and $R(\theta) = (\Psi_\theta(x_j - x_k))_{jk}$. By direct calculation and reparametrization, it can be shown that, up to an additive constant, the log-likelihood function is (Page 169 of \cite{stein2012interpolation}; Page 66 of \cite{santner2013design})
\begin{align}\label{loglike1}
    \ell(\theta,\sigma^2,\mu;X,Y) = -\frac{n}{2}\log\sigma^2-\frac{1}{2}\log \det (R(\theta)+\mu I_n) -\frac{Y^T (R(\theta) +\mu I_n)^{-1}Y}{2\sigma^2}.
\end{align}
The maximum likelihood estimate of the unknown parameters $\theta,\sigma^2,\mu$ can be found by maximizing the log-likelihood function. In practice, it is often assumed that the scale parameter, which is $\phi$ in \eqref{materngai} if a Mat\'ern correlation function is used, and the variance $\sigma^2$ are unknown and need to be estimated \citep{bachoc2013parametric,bostanabad2018uncertainty,hung2011penalized,joseph2006limit,lee2018single,li2005analysis,rasmussen2006gaussian,santner2013design,stein2012interpolation}. The smoothness parameter $\tilde\nu$ in the Mat\'ern correlation function \eqref{materngai} is usually pre-determined. We will discuss the parameter $\mu$ later. The consistency of parameter estimation has been studied in literature under the assumption that the underlying truth is a realization of a Gaussian process \citep{anderes2010consistent,bevilacqua2019estimation,kaufman2013role,mardia1984maximum,wang2011fixed,ying1991asymptotic,ying1993maximum,zhang2004inconsistent}. In particular, \cite{bevilacqua2019estimation,kaufman2013role,wang2011fixed,ying1991asymptotic,zhang2004inconsistent} show that the parameter estimation can be inconsistent and only the microergodic parameter can be estimated. For details of microergodic parameters, see \cite{matheron2012estimating}; also see pages 163-165 of \cite{stein2012interpolation}.

If the underlying truth is a fixed function that is not a sample path of the Gaussian process used in the Gaussian process modeling, to the best of our knowledge, the only related work is \cite{xu2017maximum}. In \cite{xu2017maximum}, it is shown that if the underlying function is $f(x) = x^\gamma$ with $\gamma\geq 0$ defined on $[0,1]$ and a Gaussian correlation function is used in the Gaussian process modeling, the estimated variance can either go to infinity or converge to zero as the sample size increases. Other works related to the parameter estimation under model misspecification include \cite{bachoc2013cross,bachoc2018asymptotic}.

In this work, we do not consider the consistency of parameter estimation, but we investigate the influence of parameter estimation on the prediction and uncertainty quantification of Gaussian process modeling under model misspecification.
For the ease of mathematical treatment, we assume $\Psi$ is known. It follows the standard arguments that the maximizer of \eqref{loglike1} with respect to $\sigma^2$ is
\begin{eqnarray}
\hat{\sigma}^2=Y^T (R+\mu I_n)^{-1}Y/n.\label{sigmahat}
\end{eqnarray}
Similar settings have also been considered by \cite{karvonen2020maximum}, where they call $\sigma$ a scale parameter and consider only estimating $\sigma$. In practice, $\mu$ is usually imposed as a constant \citep{dancik2007mlegp}, estimated by the sample average if there are replicates on each measurement location \citep{ankenman}, or estimated via maximum likelihood estimation \citep{wang2018controlling}. We mainly focus on the first approach (using an imposed $\hat \mu_n$) and note that the results can be easily generalized to the case of using the second approach. The third approach is much more complicated, and we do not consider it in the present work. Here we use $\hat \mu_n$ to stress the difference, because it may not be the true value of $\mu$. Recall that the underlying truth is a deterministic function from a frequentist point of view. Thus there is no definition of the ``true'' value of $\sigma^2$. It is also not clear what is the ``true'' $\mu = \sigma_\epsilon^2/\sigma^2$. In this work, we call $\hat \mu_n$ a \textit{regularization parameter}, because this parameter is \textit{imposed}. Nevertheless, we consider that the value of the regularization parameter $\hat \mu_n$ can increase or decrease with the sample size and is not restricted to be a constant.

By plugging in estimated (and imposed) parameters, we can obtain the corresponding predictor and confidence interval. We use $\hat f_n(x)$ to denote the predictor with imposed $\hat \mu_n$ on an unobserved point $x$, i.e.,
\begin{eqnarray}
\hat f_n(x) = r(x)^T (R+\hat \mu_n I_n)^{-1} Y,\label{predictorhat}
\end{eqnarray}
where $r(x)$ and $R$ are as in \eqref{mean}. By plugging $\hat \sigma^2$ as in \eqref{sigmahat} and $\hat \mu_n$ into \eqref{CIZx} and \eqref{cnxbetainCI}, we obtain the estimated pointwise confidence interval on point $x\in \Omega$
\begin{align}\label{CIZxest}
\widehat{CI}_{n,\beta}(x) = [\hat f_n(x) - \hat c_{n,\beta}(x), \hat f_n(x) + \hat c_{n,\beta}(x)],
\end{align}
where
\begin{align}\label{cnxbetainCIest}
\hat c_{n,\beta}(x) = & q_{1-\beta/2}  \sqrt{\hat \sigma^2(1 - r(x)^T (R+\hat \mu_n I_n)^{-1} r(x))},
\end{align}
and $\hat f_n(x)$ is as in \eqref{predictorhat}. In \eqref{cnxbetainCIest}, we impose a regularization parameter $\hat \mu_n > 0$ if $\sigma_\epsilon^2 > 0$, and set $\hat \mu_n = 0$ if $\sigma_\epsilon^2 = 0$. Note that the estimated variance $\hat \sigma^2$ is not present in \eqref{predictorhat}; thus it does not influence the predictor. However, $\hat \sigma^2$ appears in \eqref{CIZxest} and as we will see later, it influences the reliability of the confidence interval; thus the uncertainty quantification results of Gaussian process modeling.

\subsection{Reliability of confidence intervals}\label{subsecofCI}
In practice, the Gaussian process model is often misspecified. The underlying fixed function may lie in a subspace of the support of the corresponding Gaussian process and the subspace may have probability zero, or may not even be in the support.
This model misspecification may influence the reliability of the confidence interval thus the quality of uncertainty quantification. However, it is not possible to quantify the reliability of confidence intervals without having a clear definition of the term ``reliability''. In this section, we first review some possible ways to define the reliability, and propose our definition of the reliability of confidence intervals. 

Recall that in this work, we assume that the underlying truth is a deterministic function. Therefore, we mainly consider the reliability of confidence intervals for a fixed function. Let $g\in \mathcal{G}$ be a fixed function, where $\mathcal{G}$ is a Hilbert space of functions equipped with norm $\|\cdot\|_{\mathcal{G}}$. Let $I_Xg$ be a linear predictor for a function $g\in \mathcal{G}$, where $X=\{x_1,...,x_n\} \subset \Omega$ is the set of measurement locations. The predictor $I_Xg$ depends on $X$ and observations. Suppose the observations are $y^{(g)}_j$ for $j=1,...,n$, given by
\begin{align}\label{ygdef}
     y^{(g)}_j = g(x_j) + \epsilon_j,
\end{align}
where $\epsilon_j$'s are i.i.d. noise realizations of a random variable with mean zero and variance $\sigma_\epsilon^2\in [0,\infty)$. Typically, $I_Xg$ is a linear combination of the observations, i.e., has the form
\begin{align}\label{IXgdef}
    I_Xg(x) = \sum_{j=1}^n b_j(x) y^{(g)}_j
\end{align}
for point $x\in \Omega$, where $b_j$'s are functions not depending on $g$ but can depend on $X$. Let $CI_{X,\beta}(x) = [I_Xg(x) - a_\beta(x), I_Xg(x) +a_\beta(x)]$ be an \textit{imposed} level $(1 - \beta)100\%$ pointwise confidence interval on point $x\in \Omega$ (determined by the uncertainty quantification method that a user applies), where $\beta \in (0,1)$ and $a_\beta$ is a non-negative function. Clearly, this imposed confidence interval may not have confidence level $(1 - \beta)100\%$. We want to define the reliability of this imposed confidence interval. Note that $CI_{X,\beta}$ and $a_\beta$ can depend on $X$, but we suppress the dependency for notational simplicity. Also note the confidence interval on point $x$ is centered at $I_Xg(x)$. 

Probably the most natural way to define the reliability is by the definition of confidence intervals. This approach is considered by \cite{sniekers2015credible}. Consider the probability $\mathbb{P}_g(g(x) \in CI_{X,\beta}(x))$ for point $x\in \Omega$, where $\mathbb{P}_g$ refers to the distribution of $ y^{(g)}_1,..., y^{(g)}_n$ as in \eqref{ygdef}, where the ``true'' $g$ is given. If confidence intervals are reliable, the probability $\mathbb{P}_g(g(x) \in CI_{X,\beta}(x))$ should be close to the nominal level $(1 - \beta)100\%$, or at least larger than $(1 - \beta)100\%$ (conservative). In \cite{sniekers2015credible}, a function defined on $[0,1]$ and a Brownian motion prior are considered. The measurement locations are equally spaced. Under these settings, \cite{sniekers2015credible} shows that $CI_{X,\beta}$ can be conservative or unreliable, depending on the smoothness of $g$. However, as stated in \cite{sniekers2015credible}, the exact formulas strongly depend on the equally spaced measurement locations and cannot be easily extended to a more general choice of measurement locations.

Another probability-based definition of reliability of confidence intervals is by the average coverage probability (ACP) \citep{nychka1988bayesian}. In \cite{nychka1988bayesian}, confidence intervals are considered to be reliable if the ACP for the function $g$ and the confidence interval $CI_{X,\beta}$
\begin{align*}
    \frac{1}{n}\sum_{j=1}^n \mathbb{P}(g(x_j) \in CI_{X,\beta}(x_j))
\end{align*}
is close to the nominal level $(1-\beta)100\%$, where $x_j$ and $g(x_j)$ are as in \eqref{ygdef}. This definition only quantifies the reliability of the confidence interval on the measurement locations and does not count the confidence interval $CI_{X,\beta}(x)$ at any \textit{unobserved} point $x\in \Omega$. Therefore, the ACP is not suitable to be used for quantifying the uncertainties because if the observations are noiseless and an interpolant is used, the ACP is always equal to one.

Coverage rates are often used to assess the reliability of the confidence interval in the field of computer experiments \citep{joseph2011regression,lee2018single,sung2019multiresolution}. The coverage rate is defined by 
\begin{align}\label{coverageratedef}
    \frac{{\rm Vol}(\{x| g(x)\in  CI_{X,\beta}(x)\}) }{{\rm Vol}(\Omega)},
\end{align}
where Vol$(A)$ denotes the volume of a set $A\subset \Omega$ with respect to the Lebesgue measure.
A practical way to compute the coverage rate is by random sampling. Suppose $x_1',...,x_N'$ are $N$ uniformly distributed points in $\Omega$. Then the coverage rate can be approximated by
\begin{align*}
    \frac{{\rm card}(\{x_j'| g(x_j')\in  CI_{X,\beta}(x_j')\}) }{N},
\end{align*}
where card$(B)$ denotes the cardinality of a set $B$. However, we find it is hard to theoretically investigate the quantity \eqref{coverageratedef}, because $\{x| g(x)\in  CI_{X,\beta}(x)\}$ can be irregular and hard to characterize. 

In this work, we consider the ratio of the prediction error and the width of the confidence interval, given by $(g-I_X g)/|CI_{X,\beta}|$, where $|CI_{X,\beta}| = 2 a_\beta$ denotes the width of $CI_{X,\beta}$. We use the convention $0/0 = 0$ if $|CI_{X,\beta}(x)| = 0$ for some $x\in \Omega$. If the confidence interval $CI_{X,\beta}$ is reliable, the width of the confidence interval should be large enough to cover the difference between the predictor $I_Xg$ and the true function $g$ with high probability such that the ratio $|g(x)-I_X g(x)|/|CI_{X,\beta}(x)|$ is small for $x\in \Omega$. In particular, we consider the expectation $\left(\mathbb{E}\|(g-I_X g)/|CI_{X,\beta}|\|_{L_{p}(\Omega)}^p\right)^{1/p}$ for $2\leq p \leq \infty$ (we assume it exists; if it does not exist, then the confidence interval is thought to be not reliable), where the expectation is taken with respect to the noise and the set of measurement locations $X$, and $\|f\|_{L_{p}(\Omega)}$ is the $L_p$-norm of $f\in L_{p}(\Omega)$, defined by
\begin{align*}
    \|f\|_{L_{p}(\Omega)}^p = \int_{\Omega} |f(x)|^p dx.
\end{align*}
The expectation $\left(\mathbb{E}\|(g-I_X g)/|CI_{X,\beta}|\|_{L_{p}(\Omega)}^p\right)^{1/p}$ is the $L_p$-norm on the probability space $(\mathcal{A}, \mathcal{B}, P)$, where $\mathcal{A}$ is the sample space, $\mathcal{B}$ is the Borel algebra, and $P$ is the probability measure induced by the noise $\epsilon$ and $X$. Note that the randomness in $(g-I_X g)/|CI_{X,\beta}|$ does not come from the function $g$, because $g$ is fixed from a frequentist perspective. Because we are interested in the scenario when the number of measurement locations increases, we consider an infinite sequence of the set of measurement locations, denoted by $\mathcal{X} = \{X_1,X_2,...,X_n,...\}$. Without loss of generality, we assume that ${\rm card}(X_n)=n$, where $n$ takes its value in an infinite subset of $\mathbb{N}_+$. We call $\mathcal{X}$ a \textit{sampling scheme}, as in \cite{tuo2019kriging}. In the rest of this work, we suppress the dependency of $X$ on $n$ for notational simplicity. If the confidence interval $CI_{X,\beta}$ is reliable, $\left(\mathbb{E}\|(g-I_X g)/|CI_{X,\beta}|\|_{L_{p}(\Omega)}^p\right)^{1/p}$ should be small, at least should be less than a constant that does not depend on the sample size. From a standard frequentist perspective, we consider the \textit{minimax} setting, i.e., we consider the \textit{worst} case.

According to the prior knowledge on the function $g$, we consider two subcases. Recall that $g\in \mathcal{G}$, where $\mathcal{G}$ is a Hilbert space of functions equipped with norm $\|\cdot\|_{\mathcal{G}}$. The first subcase is that $\|g\|_{\mathcal{G}}$ is upper bounded by some known constant. Without loss of generality, assume this known constant is one, i.e., the function $g$ lies in the unit ball of $\mathcal{G}$. We say the confidence interval $CI_{X,\beta}$ is $L_p$-\textit{weakly-reliable}, if
\begin{align}\label{eq:weakdf}
    \sup_{g\in \mathcal{G}, \|g\|_{\mathcal{G}} \leq 1} \left(\mathbb{E}\|(g-I_X g)/|CI_{X,\beta}|\|_{L_{p}(\Omega)}^p\right)^{1/p}\leq C
\end{align}
holds for all $X\in \mathcal{X}$ and all $n$, where $C$ is a constant not depending on $n$. In other words, the confidence interval is weakly-reliable if it is reliable in a ball of $\mathcal{G}$ with certain radius. However, in practice we cannot always expect $\|g\|_{\mathcal{G}}$ to be bounded by a \textit{known} constant. Since $g$ is a fixed function in $\mathcal{G}$, we know $\|g\|_{\mathcal{G}}$ is finite. Therefore, for any increasing sequence $\{a_n\}_{n\geq 1}$ not depending on $g$ and $\lim_{n \rightarrow \infty} a_n = \infty$, there exists an $N$ such that for all $n\geq N$, $\|g\|_{\mathcal{G}}\leq a_n$. We say the confidence interval $CI_{X,\beta}$ is $L_p$-\textit{strongly-reliable}, if there exists an increasing sequence $\{a_n\}_{n\geq 1}$ not depending on $g$ such that $\lim_{n \rightarrow \infty} a_n = \infty$ and
\begin{align}\label{eq:strongdf}
    \sup_{g\in \mathcal{G}, \|g\|_{\mathcal{G}} \leq a_n}  \left(\mathbb{E}\|(g-I_X g)/|CI_{X,\beta}|\|_{L_{p}(\Omega)}^p\right)^{1/p}\leq C'
\end{align}
holds for all $X\in \mathcal{X}$ and all $n$, where $C'$ is a constant not depending on $n$. Here we note that the constants $C$ and $C'$ can depend on $\mathcal{X}$.
Roughly speaking, a confidence interval is strongly-reliable if it is eventually reliable in the entire space $\mathcal{G}$ as the sample size increases to infinity. We summarize the above arguments in the following definition. Note in Definition \ref{defreci}, we suppress the dependency of $X$ on $n$ for notational simplicity.

\begin{defn}\label{defreci}
Let $\beta\in (0,1)$ be fixed, and $\mathcal{X}$ be a sampling scheme. Let $I_X g$ be a linear predictor as in \eqref{IXgdef} and $n = {\rm card}(X)$ be the sample size. Let $CI_{X,\beta}$ be an imposed level $(1 - \beta)100\%$ confidence interval centered at $I_X g$ with $\lim_{n\rightarrow \infty}\sup_{x\in \Omega}|CI_{X,\beta}(x)|=0$, where $|CI_{X,\beta}(x)|$ is the width of $CI_{X,\beta}(x)$. 

For $ 2 \leq p \leq \infty$, $CI_{X,\beta}$ is said to be $L_p$-weakly-reliable if \eqref{eq:weakdf} holds for all $n$, and is said to be $L_p$-strongly-reliable if there exists an increasing sequence $\{a_n\}_{n\geq 1}$ not depending on $g$,  and $\lim_{n \rightarrow \infty} a_n = \infty$ such that for all $n$, \eqref{eq:strongdf} holds, where $C$ and $C'$ are constants not depending on $n$ but possibly depending on $p$, $\beta$, and $\mathcal{X}$. The expectation is taken with respect to noise and $X$. 
\end{defn}

\begin{remark}
In Definition \ref{defreci}, the reason we require the width of the confidence interval satisfies $\lim_{n\rightarrow \infty}\sup_{x\in \Omega}|CI_{X,\beta}(x)|=0$ because we want the confidence interval to provide some information, otherwise we can select a wide confidence interval (for example, $CI_{X,\beta}(x)=[I_Xg(x)-n,I_Xg(x)+n]$ for all $x\in \Omega$) which can cover $g(x)$ and does not provide any useful information.
\end{remark}

Definition \ref{defreci} is motivated by the properties of confidence intervals of Gaussian process. Let $Z\sim GP(0,\sigma^2\Psi)$ be a Gaussian process defined on $\Omega$. On point $x\in \Omega$, let $I_X^{(1)}Z(x) = \mathbb{E}[Z(x)|\mathcal{Z}]$, where $\mathbb{E}[Z(x)|\mathcal{Z}]$ is as in \eqref{mean}. It can be seen that $I_X^{(1)}Z$ is a linear predictor and has the form as in \eqref{IXgdef}. Let $CI_{n,\beta}(x)$ be the confidence interval as in \eqref{CIZx}. Furthermore, assume the observations are not corrupted by noise, which implies $\sigma_\epsilon^2 = 0$ and $\mu = 0$. Consider $\left(\mathbb{E}\|(Z-I_X^{(1)} Z)/|CI_{n,\beta}|\|_{L_{p}(\Omega)}^p\right)^{1/p}.$ We have the following proposition.
\begin{prop}\label{propCIforGP}
Let $Z,I_X^{(1)} Z,CI_{n,\beta}$ be described above, and $\beta\in(0,1)$. Then we have 
\begin{align}\label{propCIforGPeq}
\left(\mathbb{E}\|(Z-I_X^{(1)} Z)/|CI_{n,\beta}|\|_{L_{p}(\Omega)}^p\right)^{1/p} = C 
\end{align}
holds for all $n$ and any $2\leq p < \infty$, where $C$ is a constant only depending on $p$, $\beta$ and $\Omega$. 
\end{prop}
It can be seen that our definition of reliability stated in Definition \ref{defreci} is analogous to \eqref{propCIforGPeq}. According to Definition \ref{defreci}, if a confidence interval $CI_{X,\beta}$ is reliable, then for any fixed constant $c>0$, $cCI_{X,\beta}:=[I_Xg(x) - ca_\beta(x), I_Xg(x) +ca_\beta(x)]$ is also reliable. 
Furthermore, the ``less than or equal to'' relationship in Definition \ref{defreci} encourages a wider confidence interval. Therefore, our definition of the reliability is more like a \textit{necessary} condition rather than a sufficient condition. One way to specify the constant in Definition \ref{defreci} is by using the constant $C$ in Proposition \ref{propCIforGP}. However, one can argue that this constant may not be appropriate because unlike the unbiased predictor $I_X^{(1)} Z$, $I_X g$ is usually a biased predictor, and the constant in Proposition \ref{propCIforGP} may not be large enough to cover the bias. Practitioners may also consider other constants to counteract the model misspecification. How to choose an appropriate constant is out of the scope of this work, and we do not make any further discussion.

\subsection{Reproducing kernel Hilbert spaces and power functions}\label{subsecrkhs}

In this subsection, we review reproducing kernel Hilbert spaces and power functions, which are closely related to the Gaussian process model. Under the settings of computer experiments, if $\mu = 0$ in the Gaussian process model, the right-hand side of \eqref{predictor} is called a kriging interpolant \citep{wang2019prediction}, denoted by 
\begin{eqnarray}\label{interpolantFunction}
\mathcal{I}_{\Psi,X}f(x)=r(x)^T R^{-1} Y,
\end{eqnarray}
where $X = \{x_1,...,x_n\}$ denotes the set of measurement locations. Note that $x_k \in \Omega$ are distinct measurement locations and $\Psi$ is strictly positive definite, thus $R$ is invertible. In the area of scattered data approximation, the interpolation using operator $\mathcal{I}_{\Psi,X}$ is also called radial basis function approximation. A standard theory of radial basis function approximation works by employing \textit{reproducing kernel Hilbert spaces}. One way to define the reproducing kernel Hilbert space generated by a stationary correlation function is via the Fourier transform, defined by
$$\mathcal{F}(f)(\omega)=(2\pi)^{-d/2}\int_{\mathbb{R}^d} f(x) e^{-ix^T\omega}d x$$ for $f\in L_1(\mathbb{R}^d)$. The definition of the reproducing kernel Hilbert space can be generalized to $f\in L_2(\RR^d)\cap C(\RR^d)$. See \cite{girosi1995regularization} and Theorem 10.12 of \cite{wendland2004scattered}.

\begin{defn}\label{Def:NativeSpace}
Let $\Psi$ be a stationary correlation function that is integrable on $\RR^d$. Define the reproducing kernel Hilbert space $\mathcal{N}_\Psi(\RR^d)$ generated by $\Psi$ as
	$$\mathcal{N}_\Psi(\RR^d):=\{f\in L_2(\RR^d)\cap C(\RR^d):\mathcal{F}(f)/\sqrt{\mathcal{F}(\Psi)}\in L_2(\RR^d)\},$$
	with the inner product
	$$\langle f,g\rangle_{\mathcal{N}_\Psi(\RR^d)}=(2\pi)^{-d/2}\int_{\RR^d}\frac{\mathcal{F}(f)(\omega)\overline{\mathcal{F}(g)(\omega)}}{\mathcal{F}(\Psi)(\omega)}d \omega.$$
\end{defn}

For a positive number $\nu > d/2$, the Sobolev space on $\RR^d$ with smoothness $\nu$ can be defined as
\begin{align*}
H^\nu(\mathbb{R}^d) = \{f\in L_2(\mathbb{R}^d): |\mathcal{F}(f)(\cdot)| (1+\|\cdot\|_2^2)^{\nu/2}\in L_2(\mathbb{R}^d)\},
\end{align*}
equipped with an inner product
$$\langle f,g\rangle_{H^\nu(\mathbb{R}^d)}=(2\pi)^{-d/2}\int_{\RR^d}\mathcal{F}(f)(\omega)\overline{\mathcal{F}(g)(\omega)}(1+\|\omega\|_2^2)^{\nu}d \omega.$$ It can be shown that $H^\nu(\mathbb{R}^d)$ coincides with the reproducing kernel Hilbert space $\mathcal{N}_\Psi(\RR^d)$, if $\Psi$ satisfies Condition \ref{C1} (\cite{wendland2004scattered}, Corollary 10.13, also see Lemma \ref{coro1013}). This equivalence allows us to evaluate whether a predictor in a reproducing kernel Hilbert space is optimal; see Section \ref{unreliableCIsec} for more details.

Reproducing kernel Hilbert spaces can also be defined on a suitable subset (for example, convex and compact) $\Omega\subset \RR^d$, denoted by $\mathcal{N}_\Psi(\Omega)$, with norm
\begin{eqnarray*}
\|f\|_{\mathcal{N}_\Psi(\Omega)}=\inf\{\|f_E\|_{\mathcal{N}_\Psi(\RR^d)}:f_E\in\mathcal{N}_\Psi(\RR^d),f_E|_\Omega=f\},
\end{eqnarray*}
where $f_E|_\Omega$ denotes the restriction of $f_E$ to $\Omega$. Sobolev spaces on $\Omega$ can be defined in a similar way.

If $f\in \mathcal{N}_\Psi(\Omega)$, there is a simple error bound (\cite{wendland2004scattered}, Theorem 11.4):
\begin{eqnarray}\label{firstestimate}
|f(x)-\mathcal{I}_{\Psi,X}f(x)|\leq P_{\Psi,X}(x)\|f\|_{\mathcal{N}_\Psi(\Omega)},
\end{eqnarray}
for each $x\in\Omega$, where $P_{\Psi,X}$ is a function independent of $f$. 
The square of $P_{\Psi,X}$ is called the \textit{power function}, given by
\begin{eqnarray*}
P^2_{\Psi,X}(x)=1-r(x)^T R^{-1}r(x)
\end{eqnarray*}
for each $x\in \Omega$, where $r(x)$ and $R$ are as in \eqref{mean}. In addition, we define
\begin{eqnarray}\label{power}
\mathcal{P}_{\Psi,X}:=\sup_{x\in\Omega}P_{\Psi,X}(x).
\end{eqnarray} 
Note that the power function $P_{\Psi,X}$ and its supremum $\mathcal{P}_{\Psi,X}$ only depend on $X$, $\Omega$ and $\Psi$, and does not depend on the observations.

\subsection{Problem settings and summary of results}\label{subsec:psetting}
In this work, we consider the inference of misspecified Gaussian process models. Specifically, we consider prediction and uncertainty quantification when applying Gaussian process modeling to a fixed function $f\in \mathcal{N}_\Psi(\Omega)$, under the following misspecified model assumption.

\begin{assumption}[Misspecified model assumption]\label{assummis}
The function $f$ is a realization of a Gaussian process with mean zero and covariance function $\sigma^2\Psi$ with a finite $\sigma>0$.
\end{assumption}
We use an asterisk ``$^*$'' to denote that Assumption$^*$ \ref{assummis} is a \textit{misspecified} assumption, and is not true. After the earliest version of this work was submitted, Assumption$^*$ \ref{assummis} was also considered by \cite{karvonen2020maximum}. Under Assumption$^*$ \ref{assummis}, we incorrectly assume $f$ is a realization of $Z\sim GP(0,\sigma^2\Psi)$ for $f\in \mathcal{N}_\Psi(\Omega)$. Assumption$^*$ \ref{assummis} is a misspecified model assumption because if $Z\sim GP(0,\sigma^2\Psi)$, then $\mathbb{P}(Z\in \mathcal{N}_\Psi(\Omega))=0$ since $\Psi$ satisfies Condition \ref{C1} and $\Omega$ is convex and compact with positive Lebesgue measure \cite{driscoll1973reproducing}. In fact, the smoothness of the sample paths are at least $d/2$ different from the smoothness of the correlation function, if $\nu$ in Condition \ref{C1} is larger than $d$ \cite{steinwart2019convergence}. The different assumptions of $f\in \mathcal{N}_\Psi(\Omega)$ and $f$ is a realization of $Z\sim GP(0,\sigma^2\Psi)$ yield the same predictor, but the prediction error analysis methodologies are completely different. For discussion of these two different assumptions, see \cite{scheuerer2013interpolation}. 

Under the misspecified model assumption Assumption$^*$ \ref{assummis}, one can use maximum likelihood estimation to ``estimate'' the unknown parameters and impose confidence intervals. Of course, this is questionable, but it is widely used in practice as stated in Section \ref{secintro}, and also in numerical examples showing in research papers. In these synthetic numerical examples, the test function is typically chosen to be a fixed function with closed form (and usually infinitely differentiable), which naturally satisfies the condition $f\in \mathcal{N}_{\Psi}(\Omega)$. 
Under Assumption$^*$ \ref{assummis}, we show the following results:
\begin{enumerate}
    \item[(i)] If the observations are not corrupted by noise, then the confidence interval is not $L_p$-weakly-reliable for $p\in (2,\infty]$, and is not $L_2$-strongly-reliable.
    \item[(ii)] If the observations are corrupted by noise, then the confidence interval is not $L_2$-strongly-reliable, or the predictor is not optimal, in the sense that the predictor does not achieve the optimal convergence rate under $L_2$ metric. 
\end{enumerate}
In the rest of this work, we will use the following definitions. For two positive sequences $a_n$ and $b_n$, we write $a_n\asymp b_n$ if, for some constants $C,C'>0$, $C\leq a_n/b_n \leq C'$. Similarly, we write $a_n\gtrsim b_n$ and $b_n\lesssim a_n$ if $a_n\geq Cb_n$ for some constant $C>0$. For notational simplicity, we will use $C,C',C_1,C_2,...$ and $\eta,\eta_0,\eta_1,...$ to denote the constants, of which the values can change from line to line.

\section{When the observations are noiseless}\label{secdeter}
In this section, we consider the case that the observations have no noise. We call this case \textit{deterministic case}, because several measurements at the same location will always lead to the same response. 

\subsection{The unreliability of the confidence interval} \label{subsecdetunre}
We focus on the Mat\'ern correlation function, defined in \eqref{materngai}. Since $\phi$ and $\tilde \nu$ are known, we can let $\phi = 1/(2\sqrt{\tilde \nu})$, because otherwise we can stretch the region $\Omega$ to adjust the scale parameter $\phi$. After a proper reparametrization, we can rewrite \eqref{materngai} as
\begin{align}\label{matern}
	\Psi_{M}(h)=\frac{1}{\Gamma(\nu - d/2)2^{\nu - d/2 -1}}\|h\|_2^{\nu - d/2} K_{\nu - d/2}(\|h\|_2)
\end{align}
for $h\in \RR^d$, where $\nu > d/2$. We set $\Psi = \Psi_M$ in this section. 

Recall that in the deterministic case, 
$\sigma_\epsilon^2 = 0$, thus $\epsilon_k=0$, $k=1,...,n$, $\mu = 0$ and $\hat \mu_n = 0$. The predictor $\hat f_n(x)$ in \eqref{predictorhat} becomes a kriging interpolant \eqref{interpolantFunction}, i.e.,
\begin{eqnarray}\label{krigingidefn}
\hat f_n(x)=\mathcal{I}_{\Psi,X}f(x) = r(x)^T R^{-1} Y
\end{eqnarray}
for any point $x\in \Omega$, where $r(x)$ and $R$ are as in \eqref{mean}, and $Y=(y_1,...,y_n)^T$. Because the observations are not corrupted by noise, we have $y_k=f(x_k)$, for $k=1,...,n$. Note that in \eqref{krigingidefn}, the variance is not present and there is no estimated or imposed parameter.

As stated in Section \ref{subsecofCI}, for $\beta\in (0,1)$, an imposed confidence interval with estimated variance at point $x\in\Omega$ can be constructed by plugging $\hat \mu_n = 0$ in \eqref{CIZxest}. The confidence interval is given by 
\begin{align}\label{CIZxestdeter}
\widehat{CI}_{n,\beta}(x) = [\hat f_n(x) - \hat c_{n,\beta}(x), \hat f_n(x) + \hat c_{n,\beta}(x)],
\end{align}
where
\begin{align}
\hat c_{n,\beta}(x) = & q_{1-\beta/2}  \sqrt{\hat \sigma^2(1 - r(x)^T R^{-1}r(x))},\label{cnxbetainCIestdeter}\\
\text{ and }\qquad\hat \sigma^2  = &  Y^TR^{-1}Y/n.\nonumber
\end{align}
Since the underlying function $f$ is fixed and $f\in \mathcal{N}_\Psi(\Omega)$, we can apply \eqref{firstestimate} to derive an upper bound on the prediction error $|f(x)-\hat f_n(x)|$ for $x\in \Omega$. By \eqref{CIZxestdeter}, at point $x$, if $f(x)\in \widehat{CI}_{n,\beta}(x)$, we have $|f(x)-\hat f_n(x)|\leq \hat c_{n,\beta}(x)$. Comparing this inequality with \eqref{firstestimate}, and noting that $\hat c_{n,\beta}(x) \asymp \hat \sigma P_{\Psi,X}(x)$, if the confidence interval is reliable, intuitively, it can be expected that $\hat \sigma^2$ should be close to $\|f\|_{\mathcal{N}_\Psi(\Omega)}^2$. However, this is not true. From the identity \citep{wendland2004scattered}
\begin{align}\label{anindentity}
\|f-\mathcal{I}_{\Psi,X}f\|_{\mathcal{N}_{\Psi}(\Omega)}^2 + \|\mathcal{I}_{\Psi,X}f\|_{\mathcal{N}_{\Psi}(\Omega)}^2 = \|f\|_{\mathcal{N}_{\Psi}(\Omega)}^2,
\end{align}
it can be seen that $\hat{\sigma}^2 = \|\mathcal{I}_{\Psi,X}f\|_{\mathcal{N}_{\Psi}(\Omega)}^2/n \leq \|f\|_{\mathcal{N}_\Psi(\Omega)}^2/n = O(n^{-1})$, which is not close to $\|f\|_{\mathcal{N}_\Psi(\Omega)}^2$ as $n$ becomes larger. This indicates that $\hat c_{n,\beta}$ is too small to be used in constructing confidence intervals. Following this intuition, we show that the confidence interval is not reliable, as stated in Theorem \ref{PROPCIDETER}. We need the following condition. Recall that we suppress the dependency of $X$ on $n$ for notational simplification.

\begin{condition}\label{condoffill}
Let $X=\{x_1,...,x_n\}$, thus $n = {\rm card}(X)$. The fill distance of $X$, defined as
$$h_{X,\Omega}:= \sup_{x\in\Omega}\inf_{x_j\in X}\|x-x_j\|_2,$$
satisfies $h_{X,\Omega}\asymp n^{-1/d}$, for all $X\in \mathcal{X}$, where $\mathcal{X}$ is a sampling scheme.
\end{condition}
Condition \ref{condoffill} can be easily fulfilled. For example, sampling schemes with grid points satisfy Condition \ref{condoffill}. In fact, any quasi-uniform sampling scheme satisfies Condition \ref{condoffill}, as shown in the following proposition. 
\begin{prop}[Proposition 14.1 of \cite{wendland2004scattered}]
Let $\mathcal{X}$ be a sampling scheme. Suppose there exists a constant $C>0$ such that for all $X\in \mathcal{X}$, $h_{X,\Omega}\leq Cq_n$, where $$q_n:=\min_{x_j,x_k\in X,1\leq j\neq k\leq n}\|x_j-x_k\|_2/2.$$ Then we have $h_{X,\Omega}\asymp n^{-1/d}$. Such sequence $\mathcal{X}$ is said quasi-uniform. 
\end{prop}
By the definition of fill distance, it can be seen that $$\Omega\subset \bigcup_{k=1}^n \mathbf{B}(x_k,h_{X,\Omega}),$$ where $\mathbf{B}(x_k,h_{X,\Omega})$ denotes the Euclidean ball centered at $x_k$ with radius $h_{X,\Omega}$. Therefore, a comparison of volumes yields $$\text{Vol}(\Omega) \leq n\text{Vol}(\mathbf{B}(0,h_{X,\Omega})) = nh_{X,\Omega}^d\frac{\pi^{d/2}}{\Gamma(d/2+1)}.$$ Hence, for any set of measurement locations $X$ with card$(X)=n$, $h_{X,\Omega} \gtrsim n^{-1/d}$. By \eqref{firstestimate}, \eqref{power} and Lemma \ref{Th:Matern} in Appendix \ref{apppfthPROPCIDETER}, a set of measurement locations with small fill distance is desired, because we want the measurement locations to be spread in $\Omega$ as much as possible. Because quasi-uniform sampling schemes achieve the optimal rate of fill distance, they are widely used in computer experiments. Thus we believe Condition \ref{condoffill} is satisfied in many practical situations.

The following proposition provides an upper bound on $|\widehat{CI}_{n,\beta}(x)|$, which implies $\lim_{n\rightarrow \infty}\sup_{x\in \Omega}|\widehat{CI}_{n,\beta}(x)|=0$. Proposition \ref{propdetCIb} is a direct result of Lemma \ref{Th:Matern} and the relationship $\hat{\sigma}^2 \leq \|f\|_{\mathcal{N}_\Psi(\Omega)}^2/n$, thus the proof is omitted.

\begin{prop}\label{propdetCIb}
For any fixed sampling scheme $\mathcal{X}$ satisfying Condition \ref{condoffill}, we have that $|\widehat{CI}_{n,\beta}(x)|\leq Cn^{-\frac{\nu}{d}}$.
\end{prop}

Under Condition \ref{condoffill}, we have the following theorem. Note that $f\in \mathcal{N}_\Psi(\Omega)$ and $\Psi$ is a Mat\'ern correlation function defined in \eqref{matern} with $\nu > d/2$ imply $f\in H^{\nu}(\Omega)$; thus $f\in L_\infty(\Omega)$. 
\begin{theorem}\label{PROPCIDETER}
Suppose $2 < p \leq \infty$, $\beta \in (0,1)$ are fixed, and $\sigma_\epsilon = 0$. For any fixed sampling scheme $\mathcal{X}$ satisfying Condition \ref{condoffill}, we have that
\begin{align}\label{mainpropdetereq}
\sup_{\|f\|_{\mathcal{N}_\Psi(\Omega)} \leq 1} \|(f-\hat f_n)/|\widehat{CI}_{n,\beta}|\|_{L_{p}(\Omega)} \geq C n^{1/2 - 1/p}
\end{align}
holds for all $n$. For any increasing sequence $\{a_n\}_{n\geq 0}$ satisfying $\lim_{n \rightarrow \infty} a_n = \infty$, we have that
\begin{align}\label{mainpropdetereq2}
\sup_{\|f\|_{\mathcal{N}_\Psi(\Omega)} \leq a_n} \|(f-\hat f_n)/|\widehat{CI}_{n,\beta}|\|_{L_{2}(\Omega)} \geq C' a_n
\end{align}
holds for all $n$.
In \eqref{mainpropdetereq} and \eqref{mainpropdetereq2}, $\hat f_n$ is as in \eqref{krigingidefn}, $\widehat{CI}_{n,\beta}$ is as in \eqref{CIZxestdeter}, $C$ and $C'$ are positive constants depending on $p$, $\beta$, $\Omega$, $\Psi$ and the constants in Condition \ref{condoffill}, and do not depend on $n$.
\end{theorem}

\begin{remark}
After the earliest version of this work was submitted, a related result has appeared as Theorem 3.2 of \cite{karvonen2020maximum}, which showed that for any function $f\in \mathcal{N}_\Psi(\Omega)$, $\|(f-\hat f_n)/|\widehat{CI}_{n,\beta}|\|_{L_{\infty}(\Omega)} \leq C n^{1/2}.$ 
\end{remark}

Theorem \ref{PROPCIDETER} states that if one uses the estimated variance $\hat\sigma^2$ derived by maximum likelihood estimation to construct a pointwise confidence interval, the confidence interval can be unreliable. The confidence interval is not $L_p$-weakly-reliable for $2<p\leq \infty$ as in \eqref{mainpropdetereq}, i.e., for any $M>0$ and sufficient large $n$, there exists a function $f$ in the unit ball of $\mathcal{N}_\Psi(\Omega)$ such that $\|(f-\hat f_n)/|\widehat{CI}_{n,\beta}|\|_{L_{p}(\Omega)}\geq M$. Furthermore, the confidence interval is not $L_2$-strongly-reliable as in \eqref{mainpropdetereq2}, i.e., for $M>0$ and sufficient large $n$, there exists a function $f\in {N}_\Psi(\Omega)$ such that $\|(f-\hat f_n)/|\widehat{CI}_{n,\beta}|\|_{L_{2}(\Omega)}\geq M$. Therefore, it may not be appropriate to quantify the uncertainties by using the confidence interval derived by Gaussian process modeling for a deterministic function lying in the corresponding reproducing kernel Hilbert space if there is no noise.

\subsection{Some reliable confidence intervals under Assumption$^*$ \ref{assummis}}\label{subsecdet2other}
We adopt a reviewer's suggestion and consider two other approaches to imposing $\hat\sigma^2$: (1) setting it equal to a constant; and (2) removing the $1/n$ factor from the maximum likelihood estimate. Note that in both cases, Lemma \ref{Th:Matern} and Condition \ref{condoffill} imply that $|\widehat{CI}_{n,\beta}(x)|\leq Cn^{-\nu/d+1/2}$; thus $\lim_{n\rightarrow \infty}\sup_{x\in \Omega}|\widehat{CI}_{n,\beta}(x)|=0$.
If we set $\hat \sigma^2$ to be a positive constant, the corresponding confidence interval is $L_\infty$-weakly-reliable (thus is $L_p$-weakly-reliable for $2\leq p <\infty$) but not $L_\infty$-strongly-reliable, as stated in the following theorem. 
\begin{theorem}\label{propsec32relicon}
Suppose $\beta \in (0,1)$ is fixed, and $\sigma_\epsilon=0$. Let $\Psi = \Psi_M$, where $\Psi_M$ is as in \eqref{matern}. Let $\widehat{CI}_{n,\beta}$ be as in \eqref{CIZxestdeter} but with $\hat \sigma^2 = 1$. For any fixed sampling scheme $\mathcal{X}$ satisfying Condition \ref{condoffill}, we have that
\begin{align}\label{relidetereq1}
\sup_{\|f\|_{\mathcal{N}_\Psi(\Omega)} \leq 1} \|(f-\hat f_n)/|\widehat{CI}_{n,\beta}|\|_{L_{\infty}(\Omega)} \leq 1/(2q_{1-\beta/2})
\end{align}
holds for all $n$. For any increasing sequence $\{a_n\}_{n\geq 0}$ satisfying $\lim_{n \rightarrow \infty} a_n = \infty$, we have that
\begin{align}\label{relidetereq2}
\sup_{\|f\|_{\mathcal{N}_\Psi(\Omega)} \leq a_n} \|(f-\hat f_n)/|\widehat{CI}_{n,\beta}|\|_{L_{\infty}(\Omega)} \geq C a_n
\end{align}
holds for all $n$. The constant $C$ is positive and depends on $p$, $\beta$, $\Omega$, $\Psi$ and the constants in Condition \ref{condoffill}, but does not depend on $n$. 
\end{theorem}
Next we discuss the second approach, removing the $1/n$ factor from the maximum likelihood estimate. By this approach, the constructed confidence interval $\widehat{CI}_{n,\beta}$ is as in \eqref{CIZxestdeter} with $\hat \sigma^2 = Y^T R^{-1}Y = \|\mathcal{I}_{\Psi,X}f\|_{\mathcal{N}_{\Psi}(\Omega)}^2$. From the identity \eqref{anindentity}, it can be seen that if $\|f-\mathcal{I}_{\Psi,X}f\|_{\mathcal{N}_{\Psi}(\Omega)}^2$ converges to zero, $Y^T R^{-1}Y$ converges to $\|f\|_{\mathcal{N}_{\Psi}(\Omega)}^2$. However, in general $\|f-\mathcal{I}_{\Psi,X}f\|_{\mathcal{N}_{\Psi}(\Omega)}^2$ is not $o(1)$ \citep{edmunds2008function}. Therefore, we need to impose a stronger condition on $f$ such that $\|f-\mathcal{I}_{\Psi,X}f\|_{\mathcal{N}_{\Psi}(\Omega)}^2$ converges to zero and the corresponding confidence interval is reliable. Define an integral operator $T: L_2(\Omega) \rightarrow L_2(\Omega)$ by
\begin{align*}
    Tv(x) = \int_\Omega \Psi(x-y)v(y)dy, \quad v\in L_2(\Omega), \quad x\in \Omega,
\end{align*}
and
\begin{align*}
     T(L_2(\Omega)) = \{f| f = T v, v\in L_2(\Omega)  \}.
\end{align*}
If $f\in T(L_2(\Omega))$, the following lemma states that $\|f-\mathcal{I}_{\Psi,X}f\|_{\mathcal{N}_{\Psi}(\Omega)}\lesssim \mathcal{P}_{\Psi,X}$. Note that by Lemma \ref{Th:Matern}, $\mathcal{P}_{\Psi,X}=o(1)$; thus $\|f-\mathcal{I}_{\Psi,X}f\|_{\mathcal{N}_{\Psi}(\Omega)}^2= o(1)$.
\begin{lemma}\label{propsec32reli}
Suppose $f\in T(L_2(\Omega))$. Then we have $$\|f-\mathcal{I}_{\Psi,X}f\|_{\mathcal{N}_{\Psi}(\Omega)} \leq C \mathcal{P}_{\Psi,X} \|T^{-1} f\|_{L_2(\Omega)},$$
where $\mathcal{P}_{\Psi,X}$ is as in \eqref{power}, and $C$ only depends on $\Omega$.
\end{lemma}
Lemma \ref{propsec32reli} can be derived directly by the proof of Theorem 11.23 in \cite{wendland2004scattered} and the fact $\mathcal{P}_{\Psi,X}\leq 1$; thus the proof is omitted here. 
We have the following theorem, which states that the confidence interval constructed by the second approach is asymptotically reliable for a fixed function $f$.
\begin{theorem}\label{propCIdeterrely1}
Suppose $\beta \in (0,1)$ and $f\in T(L_2(\Omega))$ are fixed, and $\sigma_\epsilon=0$. Let $\Psi = \Psi_M$, where $\Psi_M$ is as in \eqref{matern}. Let $\widehat{CI}_{n,\beta}$ be as in \eqref{CIZxestdeter} but with $\hat \sigma^2 = Y^TR^{-1}Y$. For any fixed sampling scheme $\mathcal{X}$ satisfying Condition \ref{condoffill}, there exists $N>0$ depending on $\Psi$, $\Omega$, $f$ and the constants in Condition \ref{condoffill}, such that for all $n\geq N$, 
\begin{align*}
\|(f-\hat f_n)/|\widehat{CI}_{n,\beta}|\|_{L_{\infty}(\Omega)} \leq C,
\end{align*}
where $C$ is a positive constant only depending on $f$, $\Psi$, $\Omega$, and $\beta$.
\end{theorem}
Although Theorem \ref{propCIdeterrely1} does not imply that the confidence interval is $L_\infty$-strongly-reliable because the sample size $N$ depends on $f$, it can provide a guideline for practitioners to construct confidence intervals for deterministic functions. Whether the confidence interval with $\hat \sigma^2 = Y^TR^{-1}Y$ is $L_\infty$-strongly-reliable and the confidence interval with constant $\hat \sigma^2$ is $L_p$-strongly-reliable ($p<\infty$) will be pursued in future works.

\section{When the observations are noisy}\label{unreliableCIsec}

In this section, we consider the case that the observations are corrupted by noise. We call it \textit{stochastic case}, because multiple evaluations of the function on the same measurement location may have different observations. The observations $y_k$'s are given by \eqref{recovering}. In the stochastic case, the variance of noise $\sigma_\epsilon^2 > 0$. In this section, we still assume $f\in \mathcal{N}_\Psi(\Omega)$ in \eqref{recovering} is a fixed function. Under the misspecified assumption Assumption$^*$ \ref{assummis},
we use $\hat f_n(x)$ defined by
\begin{eqnarray}
\hat f_n(x) = r(x)^T (R+\hat \mu_n I_n)^{-1} Y,\label{prehatsto}
\end{eqnarray}
to predict $f(x)$ on a point $x\in\Omega$, where $r(x)$ and $R$ are as in \eqref{mean}, and $Y=(y_1,...,y_n)^T$. 
Through this section, we assume that the measurement locations $x_1,...,x_n$ are drawn uniformly from the input space $\Omega$, and $\hat\mu_n \asymp n^{\alpha}$ with $\alpha\in \RR$. It is obvious that $\alpha$ should be less than one in order to make meaningful predictions. In particular, if $\alpha = 0$, then $\hat \mu_n$ is at a constant rate, which is widely used in computer experiments \citep{baker2020stochastic,dancik2007mlegp}. If replicates on the same measurement location are available, then \cite{ankenman} sets $\hat \mu_n$ to be the sample variance of these replicates, which also converges to a constant as the number of replicates on each measurement location goes to infinity. It is well-known that if $\alpha = d/(2\nu+d)$, $\hat f_n$ achieves the optimal convergence rate $n^{-\frac{\nu}{2\nu+d}}$ under $L_2$ metric  \citep{stone1982optimal,geer2000empirical}. In the following theorem, we show that if $\alpha\neq d/(2\nu+d)$, the optimal convergence rate is not achieved. Recall that we use $C,C',C_1,C_2,...$ and $\eta, \eta_0,\eta_1,...$ to denote the constants, of which the values can change from line to line, and $x_k$'s are drawn uniformly from $\Omega$.
\begin{theorem}\label{NONOPTFHAT}
Suppose $\hat\mu_n \asymp n^{\alpha}$ with $\alpha<1$, and the correlation function $\Psi$ satisfies Condition \ref{C1}. Let $\hat f_n$ be given by \eqref{prehatsto}. Let $X=\{x_1,...,x_n\}$, where $x_1,...,x_n$ are uniformly distributed on $\Omega\subset\RR^d$. Under the stochastic case $(\sigma_\epsilon > 0)$, the following statements are true for all $n$.
\newline
\noindent(i) Suppose $1 > \alpha > \frac{d}{2\nu + d}$. With probability at least $1 - C_1\exp( - C_2 n^{\eta_1})$,
\begin{align}\label{mainthmlbstoineq}
\sup_{f\in \mathcal{N}_\Psi(\Omega), \|f\|_{\mathcal{N}_\Psi(\Omega)} \leq 1} \|f-\hat f_n\|_{L_{2}(\Omega)}^2 \geq C_3 n^{-(1 - 2\eta)\frac{2\nu}{2\nu + d}},
\end{align}
where $\eta = \left(\frac{(\alpha - 1)(2\nu + d)}{2\nu} + 1\right)/4 \in (0,1/4)$.
In particular, with probability at least $1 - C_4\exp( - C_5 n^{\eta_2})$,
\begin{align}\label{mainthmlbstoineqexp}
\sup_{f\in \mathcal{N}_\Psi(\Omega), \|f\|_{\mathcal{N}_\Psi(\Omega)} \leq 1}  \mathbb{E}_\epsilon\|f-\hat f_n\|_{L_{2}(\Omega)}^2 \geq C_6 n^{-(1 - 2\eta)\frac{2\nu}{2\nu + d}}.
\end{align}
(ii) Suppose $\alpha < \frac{d}{2\nu + d}$. With probability at least $1 - C_7\exp(-C_8n^{\eta_3})$,
\begin{align}\label{mainthmlbstoineqfixmu}
\sup_{f\in \mathcal{N}_\Psi(\Omega), \|f\|_{\mathcal{N}_\Psi(\Omega)} \leq 1} \mathbb{E}_\epsilon\|f-\hat f_n\|_{L_{2}(\Omega)}^2 \geq C_9 n^{\eta_4},
\end{align}
where $\eta_4 = \min\left(-\frac{1}{2}, (1 - \alpha)\frac{d}{2\nu} - 1\right) > -\frac{2\nu}{2\nu + d}$. In (i) and (ii), the constants $C_i$'s and $\eta_1,\eta_2,\eta_3$ are positive and depending on $\Psi$, $\Omega$ and $\sigma_\epsilon^2$ but not depending on $n$, and the expectation is taken with respect to $\epsilon_k$'s. The probability of \eqref{mainthmlbstoineq} is with respect to $X$ and $\epsilon_k$'s, and the probabilities of \eqref{mainthmlbstoineqexp} and \eqref{mainthmlbstoineqfixmu} are with respect to $X$.
\end{theorem}
Theorem \ref{NONOPTFHAT} provides non-asymptotic lower bounds on the mean squared prediction error under different choices of the regularization parameter value. In particular, it shows that if $\alpha \neq d/(2\nu + d)$, the optimal convergence rate cannot be achieved with high probability, as summarized in the following corollary.
\begin{corollary}\label{CORONONOPTFHAT}
Suppose $\hat\mu_n \asymp n^{\alpha}$, the correlation function $\Psi$ satisfies Condition \ref{C1}, and $\alpha\in (-\infty, d/(2\nu + d)) \cup (d/(2\nu + d),1)$. Let $\hat f_n$ be given by \eqref{prehatsto}. Under the stochastic case $(\sigma_\epsilon > 0)$, we have that
\begin{align*}
\sup_{f\in \mathcal{N}_\Psi(\Omega), \|f\|_{\mathcal{N}_\Psi(\Omega)} \leq 1}  \mathbb{E}\left(\|f-\hat f_n\|_{L_{2}(\Omega)}^2\right) \geq C n^{\eta}
\end{align*}
holds for all $n$, where $\eta > -\frac{2\nu}{2\nu + d}$ and $C>0$ are constants depending on $\Psi$, $\Omega$, and $\sigma^2_\epsilon$.
\end{corollary}

Corollary \ref{CORONONOPTFHAT} states that with the uniformly distributed measurement locations, the value of $\alpha$ other than $d/(2\nu + d)$ cannot lead to the optimal predictor. This is intuitively true because the regularization parameter $\hat \mu_n$ determines the trade-off between the bias and variance. It is well-known in the literature that by choosing $\alpha = d/(2\nu + d)$, we can achieve the best trade-off between the bias and variance. If $\alpha > d/(2\nu + d)$, the bias becomes large, and the variance is small. On the other hand, if $\alpha < d/(2\nu + d)$, the variance is large, and the bias is small. In both cases, we cannot achieve the best trade-off between the bias and variance, and have a slower convergence rate of the prediction error. To the best of our knowledge, it has not been presented in the literature that by choosing the regularization parameter value other than the rate $n^{\frac{d}{2\nu+d}}$, the optimal rate cannot be achieved. 

Next, we consider the uncertainty quantification for $\hat f_n$. Recall that for $\beta\in(0,1)$, at point $x\in \Omega$, the confidence interval constructed by Gaussian process modeling is given by 
\begin{align*}
\widehat{CI}_{n,\beta}(x) = [\hat f_n(x) - \tilde c_{n,\beta}(x;\hat \mu_n), \hat f_n(x) + \tilde c_{n,\beta}(x;\hat \mu_n)],
\end{align*}
where
\begin{align}
\tilde c_{n,\beta}(x ;\hat \mu_n) = & q_{1-\beta/2}  \sqrt{\hat \sigma^2(1 - r(x)^T (R+\hat \mu_n I_n)^{-1} r(x))},\label{cnxbetainCIestilde}\\
\text{ and }\qquad \hat{\sigma}^2= & Y^T (R+\hat \mu_n I_n)^{-1}Y/n.\nonumber
\end{align}
The following proposition states that $|\widehat{CI}_{n,\beta}(x)| = 2\tilde c_{n,\beta}(x;\hat \mu_n)$ converges to zero with $0\leq \alpha<1$.
\begin{prop}\label{propstocib}
Under the conditions of Theorem \ref{NONOPTFHAT} and $0\leq \alpha < 1$, it holds that with probability at least $1-C_1\exp(-C_2n^\eta)$,
$c_{n,\beta}(x ;\hat \mu_n)^2\leq C_3n^{-\frac{2\nu+(\alpha-1)d}{2\nu}}$ for all $n$, where $C_1,C_2,C_3$ and $\eta$ are positive constants depending on $\Psi,\Omega$, and $\sigma_\epsilon^2$.
\end{prop}
Proposition \ref{propstocib} is a direct result of Lemmas \ref{lemmaofnonpara} and \ref{lemmaerrorwnug} in Appendix \ref{secpfoftheomnonoptfhat}, and implies $\lim_{n\rightarrow\infty}\sup_{x\in \Omega}|\widehat{CI}_{n,\beta}(x)|=0$ if $0 \leq \alpha <1$. Intuitively, when $\alpha$ is large, the bias is large. Therefore, the confidence interval, which is a reflection of the variance, is not wide enough to capture the bias. As a consequence, the confidence interval $\widehat{CI}_{n,\beta}$ is not reliable. On the other hand, a smaller regularization parameter value lets the variance dominate. Therefore, the variance dominates the confidence interval; thus the confidence interval is reliable. The results related to the reliability of the confidence interval $\widehat{CI}_{n,\beta}$ are presented in the following theorem. In Theorem \ref{COROCISTO}, recall that $x_1,...,x_n$ are drawn uniformly from $\Omega$, and $|\widehat{CI}_{n,\beta}(x)|=2\tilde c_{n,\beta}(x;\hat \mu_n)$.
\begin{theorem}\label{COROCISTO}
Suppose $\hat\mu_n \asymp n^{\alpha}$ with $\alpha<1$, and the correlation function $\Psi$ satisfies Condition \ref{C1}. Fix $\beta \in (0,1)$. Under the stochastic case $(\sigma_\epsilon > 0)$, the following statements are true for all $n$.

\noindent(i) Suppose $\alpha > \frac{d}{2\nu + d}$. With probability at least $1 - C_1\exp( - C_2 n^{\eta_1})$,
\begin{align}\label{mainpropstoeqexp}
\sup_{f\in \mathcal{N}_\Psi(\Omega), \|f\|_{\mathcal{N}_\Psi(\Omega)} \leq 1}  \mathbb{E}_\epsilon\|(f-\hat f_n)/\tilde c_{n,\beta}(\cdot;\hat \mu_n)\|_{L_{2}(\Omega)}^2 \geq C n^{\eta_2}.
\end{align}

\noindent(ii) Suppose $0\leq \alpha < \frac{d}{2\nu + d}$. With probability at least $1 - C_3\exp( - C_4 n^{\eta_3})$,
\begin{align}\label{mainpropstoeqexp2}
\sup_{f\in \mathcal{N}_\Psi(\Omega), \|f\|_{\mathcal{N}_\Psi(\Omega)} \leq \sqrt{\log n}}  \mathbb{E}_\epsilon\|f-\hat f_n\|_{L_{2}(\Omega)}^2 \leq C' \mathbb{E}_\epsilon\|\tilde c_{n,\beta}(\cdot;\hat \mu_n)\|_{L_{2}(\Omega)}^2.
\end{align}
Suppose $\alpha = \frac{d}{2\nu + d}$. Then for any increasing positive sequence $\{a_n\}_{n\geq 0}$ satisfying $\lim_{n \rightarrow \infty} a_n = \infty$,
\begin{align}\label{mainpropstoeqexp23}
\sup_{f\in \mathcal{N}_\Psi(\Omega), \|f\|_{\mathcal{N}_\Psi(\Omega)}^2 \leq a_n}  \mathbb{E}_\epsilon\|(f-\hat f_n)/\tilde c_{n,\beta}(\cdot;\hat \mu_n)\|_{L_{2}(\Omega)}^2 \geq C'' a_n.
\end{align}

\noindent(iii) Suppose $\alpha < 0$. With probability at least $1 - C_5\exp(-C_6 n^{\eta_5})$,
\begin{align}\label{mainpropstoeqexpsmall}
\sup_{f\in \mathcal{N}_\Psi(\Omega), \|f\|_{\mathcal{N}_\Psi(\Omega)} \leq \log n}  \mathbb{E}_\epsilon\|(f-\hat f_n)/\tilde c_{n,\beta}(\cdot;\hat \mu_n)\|_{L_{2}(\Omega)}^2 \leq C'''.
\end{align}
In \eqref{mainpropstoeqexp}-\eqref{mainpropstoeqexpsmall}, $\hat f_n$ is as in 
\eqref{prehatsto} and $\tilde c_{n,\beta}(\cdot;\hat \mu_n)$ is as in \eqref{cnxbetainCIestilde}. In all statements, the expectation is taken with respect to $\epsilon = (\epsilon_1,...,\epsilon_n)^T$, and the constants $C',C'',C'''$, $C_i$'s and $\eta_j$'s are positive, and depend on $\Psi$, $\Omega$, $\beta$ and $\sigma_\epsilon^2$ but not depending on $n$. The probabilities are with respect to $X$.
\end{theorem}
As direct results of Theorem \ref{COROCISTO}, we have the following corollary, which states the results related to the $L_2$-reliability of the confidence interval $\widehat{CI}_{n,\beta}(x)$.
\begin{corollary}\label{COROCOROCISTO}
Suppose $\hat\mu_n \asymp n^{\alpha}$ with $\alpha<1$, and the correlation function $\Psi$ satisfies Condition \ref{C1}. Fix $\beta \in (0,1)$. Under the stochastic case $(\sigma_\epsilon > 0)$, the following statements are true for all $n$.

\noindent(i) Suppose $\alpha > \frac{d}{2\nu + d}$. We have
\begin{align*}
\sup_{f\in \mathcal{N}_\Psi(\Omega), \|f\|_{\mathcal{N}_\Psi(\Omega)} \leq 1}  \mathbb{E} \left(\|(f-\hat f_n)/\tilde c_{n,\beta}(\cdot;\hat \mu_n)\|_{L_{2}(\Omega)}^2\right) \geq C n^{\eta_1}.
\end{align*}

\noindent(ii) Suppose $0\leq \alpha < \frac{d}{2\nu + d}$.  We have
\begin{align*}
\sup_{f\in \mathcal{N}_\Psi(\Omega), \|f\|_{\mathcal{N}_\Psi(\Omega)} \leq \sqrt{\log n}}  \mathbb{E}\|f-\hat f_n\|_{L_{2}(\Omega)}^2 \leq C' \mathbb{E}\|\tilde c_{n,\beta}(\cdot;\hat \mu_n)\|_{L_{2}(\Omega)}^2.
\end{align*}
Suppose $\alpha = \frac{d}{2\nu + d}$. Then for any increasing positive sequence $\{a_n\}_{n\geq 0}$ satisfying $\lim_{n \rightarrow \infty} a_n = \infty$,
\begin{align*}
\sup_{f\in \mathcal{N}_\Psi(\Omega), \|f\|_{\mathcal{N}_\Psi(\Omega)}^2 \leq a_n}  \mathbb{E} \left(\|(f-\hat f_n)/\tilde c_{n,\beta}(\cdot;\hat \mu_n)\|_{L_{2}(\Omega)}^2\right) \geq C'' a_n.
\end{align*}

\noindent(iii) Suppose $\alpha < 0$. We have
\begin{align*}
\sup_{f\in \mathcal{N}_\Psi(\Omega), \|f\|_{\mathcal{N}_\Psi(\Omega)} \leq \log n}  \mathbb{E} \left(\|(f-\hat f_n)/\tilde c_{n,\beta}(\cdot;\hat \mu_n)\|_{L_{2}(\Omega)}^2\right) \leq C'''.
\end{align*}
In all statements, the constants $C',C'',C'''$, $C_i$'s and $\eta_1$ are positive, and only depend on $\Psi$, $\Omega$, $\beta$ and $\sigma_\epsilon^2$ but not depending on $n$.
\end{corollary}

Note that Theorem \ref{COROCISTO} and Corollary \ref{COROCOROCISTO} do not make any theoretical assertion about $L_2$-weak-reliability under the case $0\leq \alpha \leq \frac{d}{2\nu + d}$, and $L_2$-strong-reliability under the case $0\leq \alpha < \frac{d}{2\nu + d}$. As \eqref{mainpropstoeqexp2} indicates, we conjecture that the constructed confidence interval under the stochastic case is $L_2$-weakly-reliable if $ \alpha = \frac{d}{2\nu + d}$, and is $L_2$-strongly-reliable if $0\leq \alpha < \frac{d}{2\nu + d}$. We also note that \textit{(iii)} in Corollary \ref{COROCOROCISTO} does not imply the $L_2$-strong-reliability since we do not confirm the width of the confidence interval converges to zero when $\alpha < 0$.

Combining Corollary \ref{CORONONOPTFHAT} and Corollary \ref{COROCOROCISTO} suggests that if one applies the prediction and uncertainty quantification procedure from Gaussian process modeling to a deterministic function with noisy observations, the optimality of the predictor and the $L_2$-strong-reliability of the confidence interval cannot be achieved at the same time. 

As a by-product of Theorem \ref{NONOPTFHAT}, we show that if the observations are not corrupted by noise, and a regularization parameter is used as a counteract of the potential numerical instability \citep{peng2014choice}, then with uniformly distributed measurement locations, the prediction error can be controlled.
\begin{theorem}\label{thmnuggeterror}
Suppose $\sigma_\epsilon = 0$, $\hat \mu_n \asymp n^{\alpha}$ with $\alpha < 1$, and the correlation function $\Psi$ satisfies Condition \ref{C1}. Then for all $n$, with probability at least $1 - C_1\exp( - C_2 n^{\eta_1})$, we have
\begin{align*}
    \|f-\hat f_n\|_{L_\infty(\Omega)}^2 \leq C_3\max\left(n^{(\alpha - 1)(1 - \frac{d}{2\nu})}, n^{-(1 - \frac{d}{2\nu})}\right),
\end{align*}
where $C_1,C_2,C_3$ and $\eta_1$ are positive constants depending on $\Psi$, $\Omega$, and $f$.
\end{theorem}
Theorem \ref{thmnuggeterror} is a direct result of Lemma \ref{lemmaerrorwnug} in Appendix \ref{secpfoftheomnonoptfhat}. Theorem \ref{thmnuggeterror} states that $\hat f_n$ defined in \eqref{prehatsto} converges to the true underlying function $f$, even if the observations are noiseless. Note that in this theorem, it is allowed that the regularization parameter value increases as the sample size increases.

\section{A numerical example}\label{secnum}
We present a numerical example to illustrate the results in Section \ref{secdeter}, where we show that the confidence interval is not reliable in the deterministic case.

Recall that Theorem \ref{PROPCIDETER} states that there exists a function with smoothness (at least) $\nu$ such that the confidence interval is not $L_p$-reliable. However, such a function is generally not straightforward to find. In this section, we numerically illustrate that there exists a function such that the confidence interval is not $L_p$-reliable. 

Consider the following function \citep{gramacy2012cases},
\begin{align}\label{test1dfun}
    f(x) = \sin(4x) - 0.02 t_1(x;1.57,0.05)
\end{align}
for $x\in [0,1]$, where $t_1(\cdot;J_\mu,J_\sigma)$ is a Cauchy density with mean $J_\mu$ and spread $J_\sigma$. In \cite{gramacy2012cases}, it is shown that using a Gaussian correlation function yields a poor coverage rate. In this section we use a Mat\'ern correlation function. The numerical results suggest that with a Mat\'ern correlation function, the Gaussian process model does not provide an $L_p$-reliable confidence interval. As in Section \ref{secdeter}, we compute
\begin{align*}
    \|(f-\hat f_n)/|\widehat{CI}_{n,\beta}|\|_{L_{p}(\Omega)},
\end{align*}
where $f$ is as in \eqref{test1dfun}, $\hat f_n$ is as in \eqref{krigingidefn}, $\widehat{CI}_{n,\beta}$ is as in \eqref{CIZxestdeter}, and $p=4$. We use a Mat\'ern correlation function as in \eqref{matern} with $\nu = 3.5$. It can be checked that $f\in \mathcal{N}_{\Psi_M}([0,1])$. We consider the 95\% confidence interval constructed by the Gaussian process modeling. Thus, $\beta = 0.05$. We set the number of measurement locations as $n=20k$, $k=2,...,20$, and the measurement locations are grid points. We use  
\begin{align}\label{deternumapp}
    \mathcal{E} = \frac{1}{500}\sum_{j=1}^{500} \frac{|f(x_j) - \hat f_n(x_j)|^4}{|\widehat{CI}_{n,\beta}(x_j)|^4}
\end{align}
to approximate $\|(f-\hat f_n)/|\widehat{CI}_{n,\beta}|\|_{L_{4}(\Omega)}^4$, where $x_j$'s are the first 500 points in the Halton sequence \citep{niederreiter1992random}. This should give a good approximation since the points are dense enough. We add a jitter $10^{-8}$ to stabilize the computation of the matrix inverse in \eqref{krigingidefn} and \eqref{CIZxestdeter}. The results are shown in Panel 1 of Figure \ref{deter1dfig}.

We use the following approach to numerically show the rate of divergence of the ratio of the prediction error divided by the width of the confidence interval. By Theorem \ref{PROPCIDETER}, we have
\begin{align}\label{deterfunn1}
    \sup_{\|g\|_{\mathcal{N}_\Psi(\Omega)} \leq 1} \|(g-\hat g_n)/|\widehat{CI}_{n,\beta}|\|_{L_{4}(\Omega)}^4 \geq C n.
\end{align}
Taking logarithm on both sides of \eqref{deterfunn1}, we have 
\begin{align}\label{deterfunn2}
    \log \left(\sup_{\|g\|_{\mathcal{N}_\Psi(\Omega)} \leq 1} \|(g-\hat g_n)/|\widehat{CI}_{n,\beta}|\|_{L_{4}(\Omega)}^4 \right) \geq \log n + \log C.
\end{align}
We regress 
\begin{align*}
    \log\left(\frac{1}{500}\sum_{j=1}^{500} \frac{|f(x_j) - \hat f_n(x_j)|^4}{|\widehat{CI}_{n,\beta}(x_j)|^4}\right)
\end{align*}
on $\log n$. If the regression coefficient is larger than one, it indicates that the results in Theorem \ref{PROPCIDETER} hold. The results are shown in Panel 2 of Figure \ref{deter1dfig}.

Next, we consider two approaches for constructing confidence intervals described in Section \ref{subsecdet2other}. Similarly, we compute
\begin{align}\label{deternumapp2}
    \frac{1}{500}\sum_{j=1}^{500} \frac{|f(x_j) - \hat f_n(x_j)|^4}{|\widehat{CI}_{n,\beta}(x_j)|^4},
\end{align}
with $\hat \sigma^2 = C$ and $\hat \sigma^2 = Y^TR^{-1}Y$ in $\widehat{CI}_{n,\beta}$, respectively, where $C$ is a constant satisfying $\|f\|_{\mathcal{N}_\Psi([0,1])}^2\leq C$. In this example, we set $C = 2\tilde Y^T\tilde R^{-1}\tilde Y$, where $\tilde Y = (f(\tilde x_1),...,f(\tilde x_{1000}))^T$, $\tilde R = (\Psi(\tilde x_j-\tilde x_k))_{jk}$, and $\tilde x_j$'s are 1000 grid points. Since $\tilde Y^T\tilde R^{-1}\tilde Y$ should be a good approximation of $\|f\|_{\mathcal{N}_\Psi([0,1])}^2$, $C$ should be a valid upper bound of  $\|f\|_{\mathcal{N}_\Psi([0,1])}^2$. The results are shown in Panels 3 and 4 of Figure \ref{deter1dfig}.

\begin{figure}[h!]
    \centering
    \begin{subfigure}
        \centering
        \includegraphics[height=2.3in]{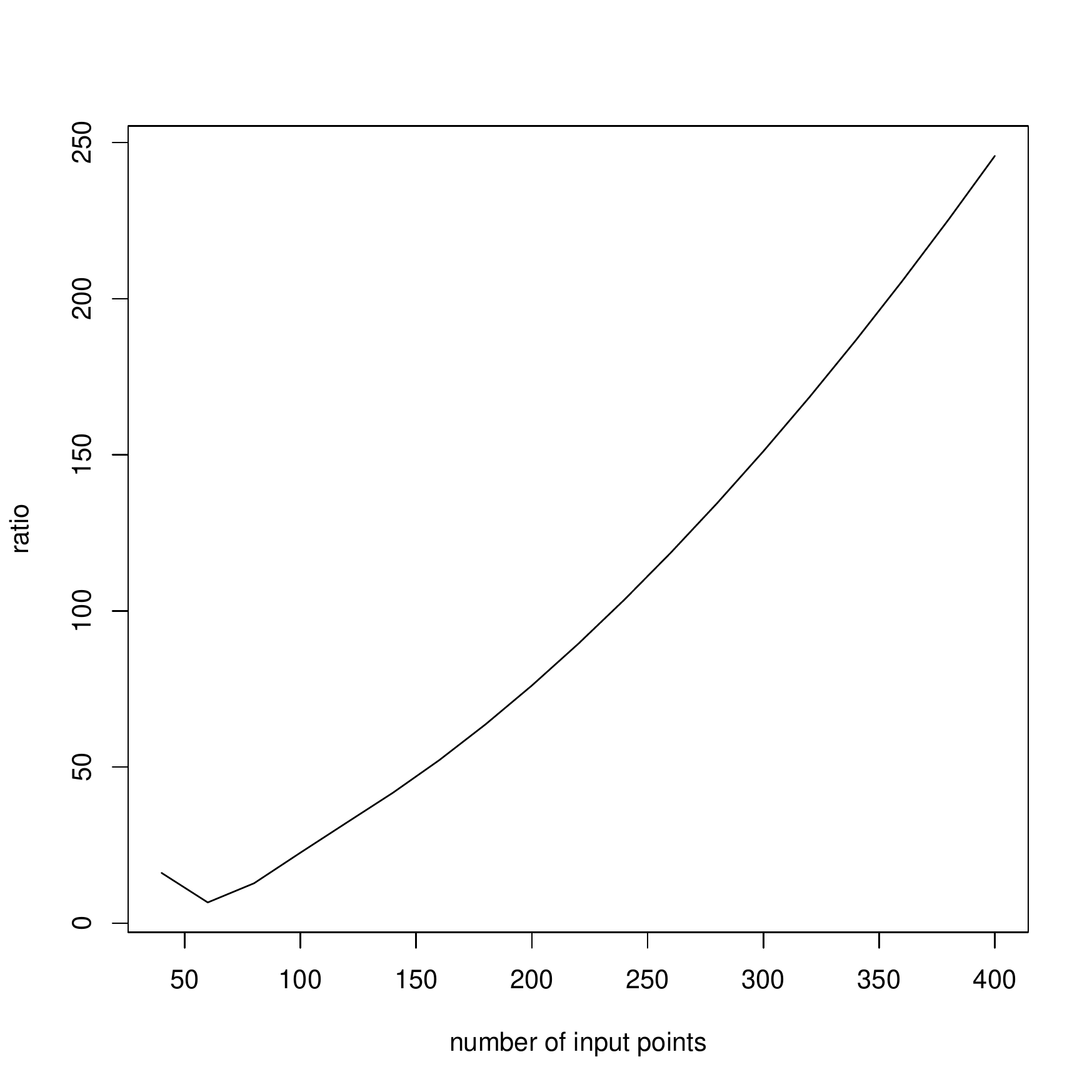}
    \end{subfigure}
    \begin{subfigure}
        \centering
        \includegraphics[height=2.3in]{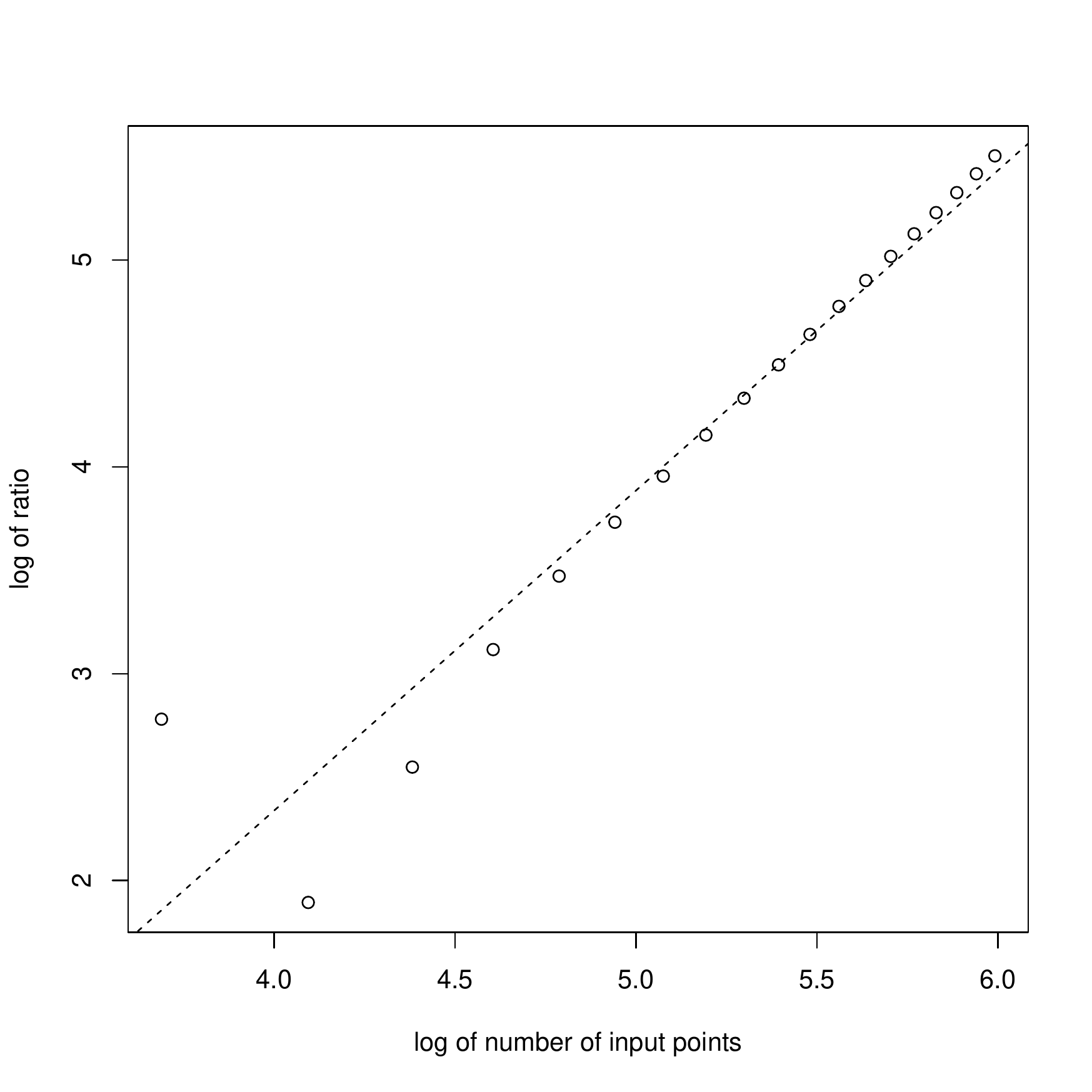}
    \end{subfigure}
    \begin{subfigure}
        \centering
        \includegraphics[height=2.3in]{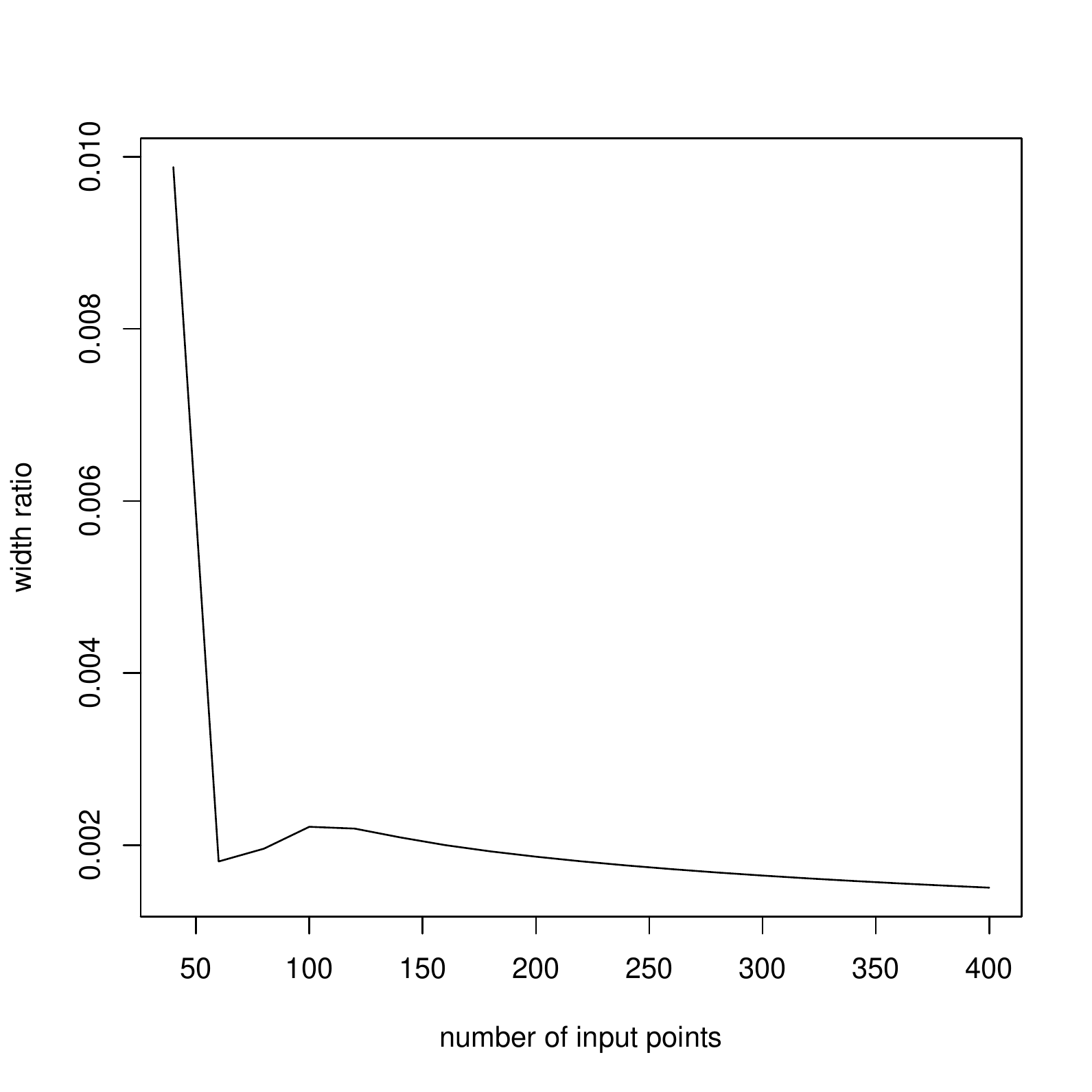}
    \end{subfigure}
    \begin{subfigure}
        \centering
        \includegraphics[height=2.3in]{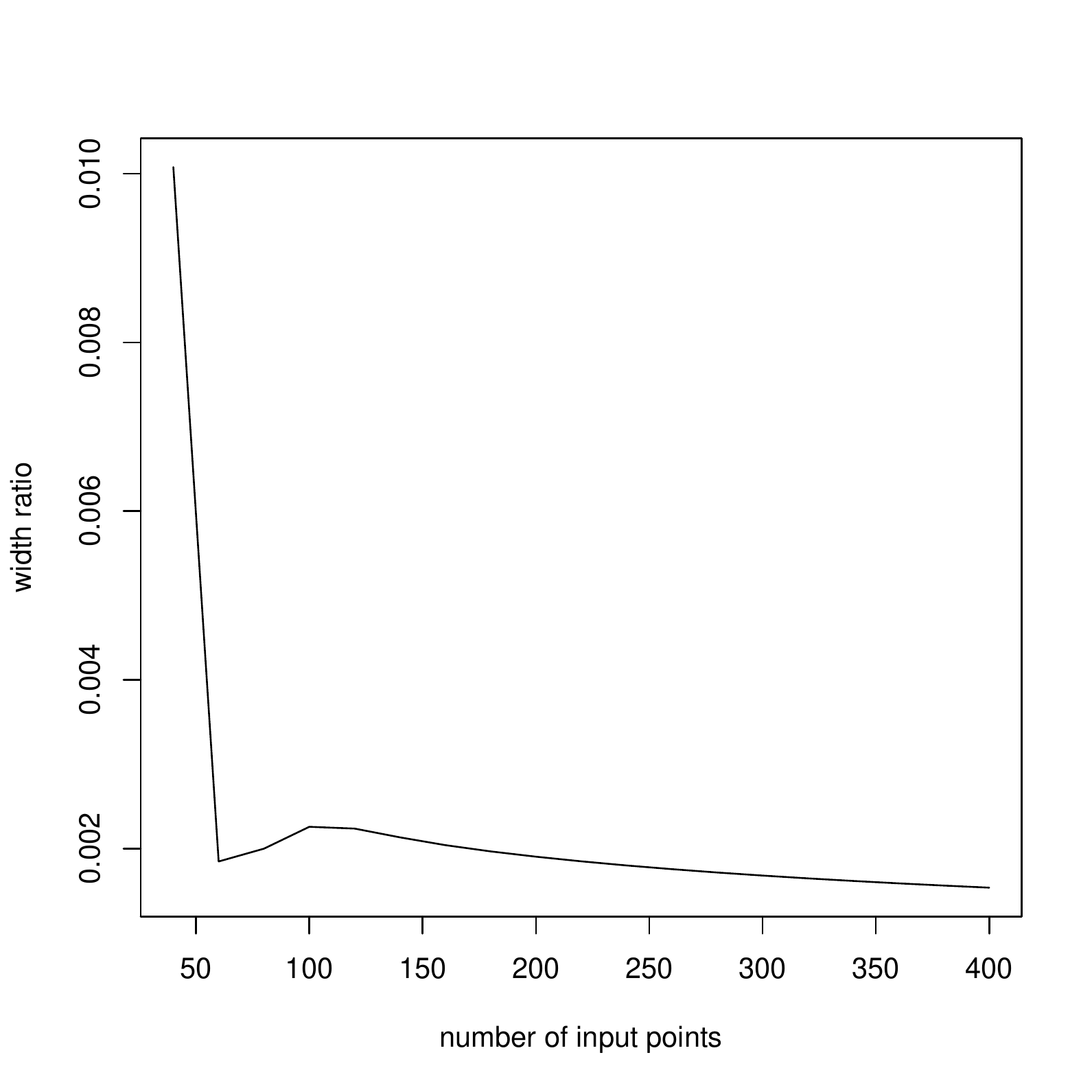}
    \end{subfigure}
    \caption{{\bf Panel 1: }Plot of $\mathcal{E}$ in \eqref{deternumapp} with $\hat \sigma^2 = Y^TR^{-1}Y/n$. {\bf Panel 2: } The regression line of $\log\mathcal{E}$ on $\log n$. The dashed line shows the regression line. Each point denotes $\log\mathcal{E}$ for each number of measurement locations. The regression coefficient is 1.548. {\bf Panel 3: }Plot of $\log\mathcal{E}$ with $\hat \sigma^2 = C$. {\bf Panel 4: }Plot of $\log\mathcal{E}$ with $\hat \sigma^2 = Y^TR^{-1}Y$.}
    \label{deter1dfig}
\end{figure}

It can be seen from Panel 1 of Figure \ref{deter1dfig} that $\mathcal{E}$ in \eqref{deternumapp} increases as the number of measurement locations increases. From Panel 2 of Figure \ref{deter1dfig}, we can see that the regression line fits the data relatively well. The regression coefficient is 1.548, which is larger than one. 
This gives us an empirical confirmation that our results in Theorem \ref{PROPCIDETER} are valid, and there exists a function such that the confidence interval is not $L_p$-reliable. As indicated by Figure \ref{deter1dfig}, we believe the results in Theorem \ref{PROPCIDETER} can be improved. In Panel 3, although the ratio increases, the largest value of the ratio is only about $10^{-4}$. Panels 3 and 4 indicate that the two approaches in Section \ref{subsecdet2other} can provide reliable confidence intervals. It can be seen that the confidence interval derived by setting $\hat \sigma^2=C$ may be conservative (the ratio is very small). Therefore, other uncertainty quantification methods may be considered if the underlying function is known to be in some reproducing kernel Hilbert space.

\section{Conclusions and discussion}\label{secdiscuss}
In this work, we consider the prediction and uncertainty quantification of the Gaussian process model applied to a fixed function in the corresponding reproducing kernel Hilbert space from a frequentist perspective. The model is misspecified under Assumption$^*$ \ref{assummis}. We consider two cases, the deterministic case, in which the observations are noiseless, and the stochastic case, where the observations are corrupted by noise. In both cases, we assume that the variance is estimated by maximum likelihood estimation, according to Assumption$^*$ \ref{assummis}. 
In the deterministic case, we show that the constructed confidence interval in the Gaussian process model is not $L_p$-weakly-reliable for $p>2$, and is not $L_2$-strongly-reliable. We also present two reliable confidence intervals under some scenarios. In the stochastic case, the regularization parameter value is assumed to be at a certain rate. We show that the predictor derived by the Gaussian process modeling is not optimal and/or the constructed confidence interval is not $L_2$-strongly-reliable. These results indicate that the optimal predictor and $L_2$-strong-reliability cannot be achieved at the same time if the Gaussian process model is misspecified. As by-products, we obtain several lower bounds on the mean squared prediction error under different choices of the regularization parameter value.

In the Gaussian process model, it is often assumed that there are some unknown parameters, and maximum likelihood estimation or Bayesian methods, are used to estimate these parameters, even if the Gaussian process model is misspecified. Prediction performance of the Gaussian process model with maximum likelihood estimation under misspecification has been studied by \cite{stein1993spline,bachoc2018asymptotic}.
In \cite{stein1993spline}, they show for periodic functions, under some situations, the misspecified Gaussian process model with maximum likelihood estimation can still work well in terms of prediction. If the underlying truth is indeed a Gaussian process, using a misspecified correlation function and maximum likelihood estimation may not have desired prediction performance, as suggested in \cite{bachoc2018asymptotic}.

In addition to prediction, uncertainty quantification is another important problem in computer experiments. Because the Gaussian process model has a probabilistic structure, models based on Gaussian process modeling are usually validated via confidence intervals. In order to test the performance of these models, typically, several simulations are conducted. In some literature, the test function is selected to be a deterministic function with a closed form. In these cases, the Gaussian process model is usually misspecified, i.e., the function may not be a sample path of the corresponding Gaussian process. However, an imposed pointwise confidence interval is still constructed and used to quantify the uncertainty. It has been observed that Gaussian process models often have poor coverage of their confidence intervals \citep{gramacy2012cases,joseph2011regression,yamamoto2000alternative}. There is no theoretical result explaining this phenomenon from a frequentist perspective to the best of our knowledge. Our results provide some insights on explaining the poor coverage of the confidence interval, and a better understanding of the model misspecification in Gaussian process modeling.

Several statistical inference methods have been studied if the confidence interval is not derived directly from the Gaussian process modeling. Most of them are from a Bayesian perspective. For example, in \cite{sniekers2015credible,yoo2016supremum}, credible intervals are constructed and analyzed for Gaussian process models (or particularly Brownian motion). Additionally, \cite{yang2017frequentist,yano2018frequentist} derive finite sample bounds on frequentist coverage errors of Bayesian credible intervals for Gaussian process models. Therefore, one can also use other inference methods besides the confidence interval in the Gaussian process modeling.

Several problems are not considered in this work. First, we only consider uniformly distributed measurement locations in the stochastic case, where fixed designs are not considered. Second, we only consider the case that the regularization parameter value is predetermined with a certain rate. We could not confirm similar results if we use maximum likelihood estimation to estimate the regularization parameter value, or select parameter values using other criteria as in \cite{kou2003efficiency}. Also, we do not consider using maximum likelihood estimation to estimate the parameters of the correlation function $\Psi$. Therefore, a thorough investigation of applying maximum likelihood estimation under misspecification is needed. Third, as discussed in Section \ref{subsecofCI}, Definition \ref{defreci} is a necessary condition. One possible way to improve this definition is to restrict the $L_p$-norm of the ratio such that it is bounded away from zero. 

\section*{Acknowledgments}
The author is grateful to the AE and three reviewers for their constructive comments. The author especially appreciates one reviewer's valuable suggestions that improve this paper a lot.

\begin{appendix}

\section*{Appendix}
\setcounter{section}{0}
\numberwithin{equation}{section}
\numberwithin{lemma}{section}
\section{Notation}
We use $\langle \cdot, \cdot \rangle_n$ to denote the empirical inner product, which is defined by $$\langle f, g \rangle_n = \frac{1}{n}\sum_{k=1}^n f(x_k)g(x_k)$$ for two functions $f$ and $g$, and let $\|g\|_n^2 = \langle g, g \rangle_n$ be the empirical norm of function $g$. In particular, let $$\langle \epsilon, f \rangle_n = \frac{1}{n}\sum_{k=1}^n\epsilon_k f(x_k)$$ for a function $f$, where $\epsilon = (\epsilon_1,\ldots,\epsilon_n)^T$. Let $a\vee b =\max(a,b)$ for two real numbers $a,b$. We use $H(\cdot, \mathcal{F}, \|\cdot\|)$ and $H_B(\cdot,\mathcal{F},\|\cdot\|)$ to denote the entropy number and the bracket entropy number of class $\mathcal{F}$ with the (empirical) norm $\|\cdot\|$, respectively. Through the proof, we assume Vol$(\Omega)=1$ for the ease of notational simplicity.

\section{Proof of Proposition \ref{propCIforGP}}
By \eqref{CIZx} and \eqref{cnxbetainCI}, it can be seen that $|CI_{n,\beta}(x)| = 2q_{1-\beta/2}  \sqrt{\text{Var}[Z(x)|\mathcal{Z}]}$ for a point $x$. 
By Fubini's theorem, 
\begin{align*}
\mathbb{E}\|(Z-I_X^{(1)} Z)/|CI_{n,\beta}|\|_{L_{p}(\Omega)}^p & = \int_{x\in \Omega}\frac{\mathbb{E}(Z(x)-I_X^{(1)} Z(x))^p}{C_1  (\sqrt{\text{Var}[Z(x)|\mathcal{Z}]})^{p}}dx\\
& = \int_{x\in \Omega}\frac{2^{p/2}\Gamma(\frac{p+1}{2})}{\sqrt{\pi}}\frac{(\sqrt{\text{Var}[Z(x)|\mathcal{Z}]})^p}{C_1  (\sqrt{\text{Var}[Z(x)|\mathcal{Z}]})^{p}}dx\\
& = C_2,
\end{align*}
where the second equality can be derived by the calculation of $p$th moment of normal distribution. See \cite{walck1996hand}.

\section{Proof of Theorem \ref{PROPCIDETER}}\label{apppfthPROPCIDETER}
The proof of Theorem \ref{PROPCIDETER} relies on approximation numbers. The $n$th approximation number of the embedding $id: H^{\nu}(\Omega) \rightarrow L_{p}(\Omega)$, denoted by $b_n$, is defined by
\begin{align}\label{defappnum}
    b_n = \inf \{\|id - L\|, L \in \mathcal{L}(H^{\nu}(\Omega), L_{p}(\Omega)), \text{rank}(L)<n\},
\end{align}
where $\mathcal{L}(H^{\nu}(\Omega), L_{p}(\Omega))$ is the family of all bounded linear mappings $H^{\nu}(\Omega) \rightarrow L_{p}(\Omega)$, $\|\cdot\|$ is the operator norm, and $\text{rank}(L)$ is the dimension of the range of $L$. In \eqref{defappnum}, we define rank$(L) = 0$ if $L(f) = 0$ for all $f\in H^{\nu}(\Omega)$. Therefore, it can be seen that $\|id\| = b_1 \geq b_2 \geq ... > 0$. Lemma \ref{appnumtoL2} states a property of approximation numbers \citep{edmunds2008function}.
\begin{lemma}\label{appnumtoL2}
Suppose $p\geq 2$. The approximation number $b_n$ defined in \eqref{defappnum} satisfies that for all $n\in \mathbb{N}$, 
\begin{align}\label{approxNum}
c_1 n^{-\frac{\nu}{d}+\frac{1}{2}-\frac{1}{p }} \leq b_n \leq c_2 n^{-\frac{\nu}{d}+\frac{1}{2}-\frac{1}{p }},
\end{align}
where $c_1$ and $c_2$ are two positive constants depending on $\Omega$, $\nu$ and $p$.
\end{lemma}
Lemma \ref{Th:Matern} is a direct result of Theorem 5.14 of \cite{wu1993local}, which provides an upper bound on $\mathcal{P}_{\Psi_M,X}$ defined in \eqref{power}.

\begin{lemma}\label{Th:Matern}
	Let $\Omega$ be compact and convex with a positive Lebesgue measure; $\Psi_M$ be a Mat\'ern correlation function given by (\ref{matern}). Suppose Condition \ref{condoffill} holds for a sampling scheme $\mathcal{X}$. Then there exist constants $c,c_1$ depending only on $\Omega$, $\mathcal{X}$, and $\nu$ in (\ref{matern}), such that
	$\mathcal{P}_{\Psi_M,X}\leq c n^{-\frac{\nu}{d} + \frac{1}{2}}$ provided that $n^{-\frac{\nu}{d} + \frac{1}{2}}\leq c_1$ and $X\in \mathcal{X}$.
\end{lemma}

Lemma \ref{coro1013} is a direct result of Corollary 10.13 in \cite{wendland2004scattered} and the extension theorem \citep{devore1993besov}. Lemma \ref{coro1013} states that the reproducing kernel Hilbert space $\mathcal{N}_{\Psi}(\Omega)$ coincides with the Sobolev space with smoothness $\nu$ $H^{\nu}(\Omega)$, for correlation functions satisfying Condition \ref{C1}.  
\begin{lemma}\label{coro1013}
	Suppose Condition \ref{C1} is satisfied. We have the following.
	\begin{enumerate}
	    \item The reproducing kernel Hilbert space $\mathcal{N}_{\Psi}(\RR^d)$ coincides with the Sobolev space with smoothness $\nu$ $H^{\nu}(\RR^d)$, and the norms $\|\cdot\|_{\mathcal{N}_{\Psi}(\RR^d)}$ and $\|\cdot\|_{H^\nu(\RR^d)}$ are equivalent.
	    \item Suppose $\Omega$ is compact and convex. Then the reproducing kernel Hilbert space $\mathcal{N}_{\Psi}(\Omega)$ coincides with the Sobolev space with smoothness $\nu$ $H^{\nu}(\Omega)$, and the norms $\|\cdot\|_{\mathcal{N}_{\Psi}(\Omega)}$ and $\|\cdot\|_{H^\nu(\Omega)}$ are equivalent.
	\end{enumerate} 
\end{lemma}

Now we are ready to prove Theorem \ref{PROPCIDETER}.

By Lemma \ref{appnumtoL2}, there exists a function $\phi_n\in H^\nu(\Omega)$ satisfying $\|\phi_n\|_{H^\nu(\Omega)} = 1$ such that
\begin{align*}
c_1n^{-\frac{2\nu}{d}+1-\frac{2}{p}} & \leq \|\phi_n - \mathcal{I}_{\Psi,X}\phi_n\|_{L_p(\Omega)}^2,
\end{align*}
since $\mathcal{I}_{\Psi,X}$ is a rank $n$ linear operator. By Lemma \ref{Th:Matern}, \eqref{anindentity}, and Lemma \ref{coro1013}, we have for sufficiently large $N$ such that $N^{-\frac{\nu}{d} + \frac{1}{2}} \leq h_0$ and for all $n\geq N$,
\begin{align*}
\hat c_{n,\beta}(x)^2 & = \frac{C_1}{n}Y^T R^{-1}Y(1 - r(x)^TR^{-1}r(x)) \leq \frac{C_1}{n}Y^T R^{-1}Y \mathcal{P}_{\Psi_M,X}^2\\
& \leq \frac{C_2}{n} \|\phi_n\|_{\mathcal{N}_{\Psi}(\Omega)}^2 n^{-\frac{2\nu}{d}+1} \leq \frac{C_3}{n} \|\phi_n\|_{H^\nu(\Omega)}^2 n^{-\frac{2\nu}{d} + 1} = C_3n^{-\frac{2\nu}{d}},
\end{align*}
for any $x\in \Omega$, where $Y=(\phi_n(x_1),...,\phi_n(x_n))^T$. Let $f$ in \eqref{mainpropdetereq} be equal to $\phi_n$. Therefore, we have 
\begin{align*}
    \|(f - \hat f_n)/\hat c_{n,\beta}\|_{L_p(\Omega)}^2 & \geq C_4 \|f - \hat f_n\|_{L_p(\Omega)}^2n^{\frac{2\nu}{d}} \geq C_5 n^{1-\frac{2}{p}},
\end{align*}
which finishes the proof of \eqref{mainpropdetereq}.

The case $p=2$ can be proved similarly. The only difference is that we let $f = a_n\phi_n$ such that $\|a_n\phi_n\|_{H^\nu(\Omega)} = a_n$.

\section{Proof of Theorem \ref{propsec32relicon}}
We first show that \eqref{relidetereq1} holds. Plugging $\hat c_{n,\beta}(x) = q_{1-\beta/2}P_{\Psi,X}(x)$, it suffices to show 
\begin{align*}
    \left|\frac{f(x) - \hat f_n(x)}{P_{\Psi,X}(x)}\right| \leq 1
\end{align*}
holds for all $f$ with $\|f\|_{\mathcal{N}_\Psi(\Omega)}\leq 1$ and $x\in \Omega$. This is a direct result of \eqref{firstestimate}. 

The second inequality \eqref{relidetereq2} can be shown by a similar approach as in the proof of Theorem \ref{PROPCIDETER}. By Lemma \ref{appnumtoL2}, there exists a function $\phi_n\in H^\nu(\Omega)$ satisfying $\|\phi_n\|_{H^\nu(\Omega)} = 1$ such that
\begin{align*}
c_1n^{-\frac{2\nu}{d}+1} & \leq \|\phi_n - \mathcal{I}_{\Psi,X}\phi_n\|_{L_\infty(\Omega)}^2.
\end{align*}
By Lemma \ref{Th:Matern}, we have 
\begin{align*}
\hat c_{n,\beta}(x)^2 & \leq C_1 \|\phi_n\|_{\mathcal{N}_{\Psi}(\Omega)}^2 n^{-\frac{2\nu}{d}+1} \leq C_2 \|\phi_n\|_{H^\nu(\Omega)}^2 n^{-\frac{2\nu}{d} + 1} = C_2n^{-\frac{2\nu}{d}+1},
\end{align*}
for any $x\in \Omega$ and sufficiently large $n$ such that $n^{-\frac{\nu}{d} + \frac{1}{2}} \leq h_0$. Letting $f=a_n\phi_n$, we have
\begin{align*}
    \|(f - \hat f_n)/\hat c_{n,\beta}\|_{L_\infty(\Omega)}^2 & \geq C_3 a_n,
\end{align*}
which finishes the proof.

\section{Proof of Theorem \ref{propCIdeterrely1}}
By plugging $\hat c_{n,\beta}(x) = q_{1-\beta/2}\sqrt{Y^TR^{-1}Y}P_{\Psi,X}(x)$, it suffices to show that there exists $N>0$ such that for all $n>N$,
\begin{align*}
    \left|\frac{f(x) - \hat f_n(x)}{\sqrt{Y^TR^{-1}Y}P_{\Psi,X}(x)}\right| \leq C.
\end{align*}
By \eqref{firstestimate} and Lemma \ref{propsec32reli}, for sufficiently large $n$, it can be seen that 
\begin{align*}
    \left|\frac{f(x) - \hat f_n(x)}{\sqrt{Y^TR^{-1}Y}P_{\Psi,X}(x)}\right|^2 \leq &  \frac{\|f\|_{\mathcal{N}_\Psi(\Omega)}^2}{Y^TR^{-1}Y}=\frac{\|f\|_{\mathcal{N}_\Psi(\Omega)}^2}{\|f\|_{\mathcal{N}_\Psi(\Omega)}^2 - \|f-\mathcal{I}_{\Psi,X}f\|_{\mathcal{N}_{\Psi}(\Omega)}^2}\\
    \leq & \frac{\|f\|_{\mathcal{N}_\Psi(\Omega)}^2}{\|f\|_{\mathcal{N}_\Psi(\Omega)}^2 - C^2 P_{\Psi,X}^2 \|T^{-1} f\|_{L_2(\Omega)}^2}\\
    \leq & \frac{\|f\|_{\mathcal{N}_\Psi(\Omega)}^2}{\|f\|_{\mathcal{N}_\Psi(\Omega)}^2 - C_1 n^{-\frac{2\nu}{d}+1} \|T^{-1} f\|_{L_2(\Omega)}^2}\leq \frac{1}{2},
\end{align*}
where the first inequality is by \eqref{anindentity}, and the last inequality follows from $n^{-\frac{2\nu}{d}+1}$ converges to $0$. Then we finish the proof.

\section{Proof of Theorem \ref{NONOPTFHAT}}\label{secpfoftheomnonoptfhat}
Recall that in the stochastic case, we assume $x_1,...,x_n$ are drawn uniformly from $\Omega$. Before we show the proof of Theorem \ref{NONOPTFHAT}, we first present some lemmas used in this section. 
Note that the proof of Lemma \ref{lemmaefsmall} is based on Lemma 8.4 of \cite{geer2000empirical}; thus it is omitted here. Lemmas \ref{lemmavandegeerf} and \ref{thm:thm21inGeer2014} are Theorem 10.2 of \cite{geer2000empirical} and Theorem 2.1 of \cite{van2014uniform}, respectively.

\begin{lemma}\label{lemmaefsmall}
Suppose $\epsilon_1,...,\epsilon_n$ are independent and identically normally distributed variables. Then for all $t > C$, with probability at least $1 - C_1\exp(-C_2 t^2)$,
\begin{align*}
\sup_{g\in \mathcal{N}_{\Psi}(\Omega)}\frac{|\langle \epsilon, g \rangle_n|}{\|g\|_n^{1 - \frac{d}{2\nu}}\|g\|_{\mathcal{N}_{\Psi}(\Omega)}^{\frac{d}{2\nu}}} \leq tn^{-\frac{1}{2}}.
\end{align*}
\end{lemma}

\begin{lemma}\label{lemmavandegeerf}
Suppose $f\in H^\nu(\Omega)$ and $\hat \mu_n^{-1} = O_P(n^{-\frac{d}{2\nu +d }})$. Then we have
\begin{align*}
\|f - \hat f_n\|_n & = O_P(\hat \mu_n^{\frac{1}{2}}n^{-\frac{1}{2}} \vee n^{\frac{d-2\nu}{4\nu}}\hat \mu_n^{-\frac{d}{4\nu}}),\nonumber\\
\|\hat f_n\|_{\mathcal{N}_{\Psi}(\Omega)} & = O_P(1 \vee n^{\frac{d}{4\nu}}\hat \mu_n^{-\frac{2\nu + d}{4\nu}}),
\end{align*}
where $\hat f_n$ is defined as in \eqref{prehatsto}.
\end{lemma}

\begin{lemma}\label{thm:thm21inGeer2014}
Let $\mathcal{R}:=\sup_{f\in\mathcal{F}}\|f\|_{L_2(\Omega)}, K:=\sup_{f\in\mathcal{F}}\|f\|_{L_\infty(\Omega)},$ where $\mathcal{F}$ is a function class. Then for all $t>0$, with probability at least $1-\exp(-t)$,
\begin{align*}
\sup_{f\in\mathcal{F}}\bigg|\|f\|^2_n-\|f\|^2_{L_2(\Omega)}\bigg|\leq C_1\bigg(\frac{2\mathcal{R} J_\infty(K,\mathcal{F})+\mathcal{R}K\sqrt{t}}{\sqrt{n}}+\frac{4 J_\infty^2(K,\mathcal{F})+K^2t}{n}\bigg),
\end{align*}
where $C_1$ is a constant, and
\begin{align}\label{defJinfty}
J_\infty^2(z,\mathcal{F})=C_2^2\inf_{\delta>0}\mathbb{E}\bigg[z\int_\delta^1\sqrt{H(uz/2,\mathcal{F},\|\cdot\|_{L_\infty(\Omega)})}du+\sqrt{n}z\delta \bigg]^2,
\end{align}
with $C_2$ another constant.
\end{lemma}

The following lemma is a Bernstein-type inequality for a single function $g$. See, for example, \cite{massart2007concentration}.
\begin{lemma}\label{Berforsingleg}
Suppose $X_i$, $i=1, \ldots, n$ are uniformly distributed on $\Omega$. Let $Z_i=(\|g\|_{L_2(\Omega)}^2/\text{Vol}(\Omega)-g(X_i)^2)/\|g\|^2_{\mathcal{N}_{\Psi}(\Omega)}$ for a function $g\in \mathcal{N}_{\Psi}(\Omega)$. Suppose $|Z_i|\leq b$ for some constant $b>0$. For all $t>0$, we have
\begin{align*}
P\bigg(\frac{1}{n}\sum_{i=1}^n Z_i \geq t\bigg)\leq \exp\bigg[-\frac{nt^2/2}{E(Z_1^2)+bt/3}\bigg],
\end{align*}
which is the same as
\begin{align*}
P\bigg(\frac{\|g\|_{L_2(\Omega)}^2/\text{Vol}(\Omega)}{\|g\|^2_{\mathcal{N}_{\Psi}(\Omega)}}-\frac{\|g\|_n^2}{\|g\|^2_{\mathcal{N}_{\Psi}(\Omega)}} \geq t\bigg)\leq \exp\bigg[-\frac{nt^2/2}{E(Z_1^2)+bt/3}\bigg].
\end{align*}
\end{lemma}

\begin{lemma}\label{lemmaratio1}
Assume for class $\mathcal{G}$, $\sup_{g\in \mathcal{G}}\|g\|_{L_\infty(\Omega)}\leq K < 1$, and the bracket entropy $H_B(\delta_n/\text{Vol}(\Omega),\mathcal{G},\|\cdot\|_{L_\infty(\Omega)})\leq \frac{n\delta_n^2}{1200\text{Vol}(\Omega)K^2}$, and $n\delta_n^2\rightarrow \infty$, where $\text{Vol}(\Omega)$ denotes the volume of $\Omega$ and $0 < \delta_n < 1$. Then we have
\begin{align*}
P\bigg(\inf_{\|g\|_{L_2(\Omega)} \geq 2\delta_n, g\in \mathcal{G}} \frac{\|g\|^2_n}{\|g\|_{L_2(\Omega)}^2}<\eta_1\bigg)\leq C_1\exp(-C_2n\delta_n^2/K^2),
\end{align*}
and
\begin{align*}
P\bigg(\sup_{\|g\|_{L_2(\Omega)} \geq 2\delta_n, g\in \mathcal{G}} \frac{\|g\|^2_n}{\|g\|_{L_2(\Omega)}^2}>\eta_2\bigg)\leq C_3\exp(-C_4n\delta_n^2/K^2),
\end{align*}
for some constants $\eta_1, \eta_2 > 0$ and $C_i$'s only depending on $\Omega$.
\end{lemma}

\begin{lemma}\label{lemmaofnonpara}
For any $\mu \asymp n^{\alpha}$ with $0 \leq \alpha < 1$, with probability at least $1-C\exp(-n^{\eta})$,
\begin{align*}
\frac{\mu}{n}Y^T (R+\mu I_n)^{-1}Y \asymp \sigma_\epsilon^2,
\end{align*}
where $Y=(y_1,...,y_n)^T$ with $y_k$ defined in \eqref{recovering}, and $R$ is as in \eqref{mean}. Furthermore, if $\mu^{-1} = O_P(n^{-\frac{d}{2\nu +d }})$, then
\begin{align*}
\frac{\mu}{n}Y^T (R+\mu I_n)^{-1}Y \rightarrow \sigma_\epsilon^2
\end{align*}
in probability.
\end{lemma}

\begin{lemma}\label{tracelemma}
Suppose $A,B$ and $C\in \RR^{n\times n}$ are positive definite matrices. We have
\begin{align*}
\text{tr} ((A + B)(A + B + C)^{-1}) \geq \text{tr} (A(A + C)^{-1}),
\end{align*}
and
\begin{align*}
\text{tr} ((A + B)^2(A + B + C)^{-2}) \geq \text{tr} (A^2(A + C)^{-2}).
\end{align*}
\end{lemma}

\begin{lemma}\label{lemmaerrorwnug}
Suppose $f \in \mathcal{N}_{\Psi}(\Omega)$ and $0 \leq\alpha < 1$. With probability at least $1 - C_1\exp( - C_2 n^{\eta_1})$, we have
\begin{align*}
    (f(x) - r(x)^T(R+ \hat \mu_n I_n)^{-1}f(X))^2 \leq (1 -  r(x)^T(R+ \hat \mu_n I_n)^{-1}r(x))\|f\|_{\mathcal{N}_{\Psi}(\Omega)}^2,
\end{align*}
and
\begin{align*}
    1- r(x)^T(R + \hat \mu_n I_n)^{-1}r(x) \lesssim n^{(\alpha - 1)(1 - \frac{d}{2\nu})},
\end{align*}
where $r(x)$ and $R$ are as in \eqref{mean}, $\hat \mu_n  \asymp n^{\alpha}$, and $f(X) = (f(x_1),...,f(x_n))^T$.
\end{lemma}

We first state the intuition behind the proof. Direct calculation shows that
\begin{align*}
    & \mathbb{E}_\epsilon\|f-\hat f_n\|_{L_{2}(\Omega)}^2\nonumber\\ = & \mathbb{E}_\epsilon\int_\Omega (f(x) - r(x)^T(K+ \hat \mu_n I_n)^{-1}f(X) - r(x)^T(K+ \hat \mu_n I_n)^{-1}\epsilon)^2 dx\nonumber\\
= & \underbrace{\sigma_\epsilon^2 \int_\Omega r(x)^T(K+ \hat \mu_n I_n)^{-2}r(x)dx}_{variance} + \underbrace{\int_\Omega (f(x) - r(x)^T(K+ \hat \mu_n I_n)^{-1}f(X))^2 dx}_{bias}.
\end{align*}
If $\hat \mu_n$ is large ($\alpha > \frac{d}{2\nu+d}$), the bias dominates. Therefore, to obtain a lower bound of the mean-squared prediction error for the case $\alpha > \frac{d}{2\nu+d}$, we only need to obtain a lower bound of the bias term. On the other hand, if $\hat \mu_n$ is small ($\alpha < \frac{d}{2\nu+d}$), we only need to obtain a lower bound for the variance term.

Now we present the proof of Theorem \ref{NONOPTFHAT}. We first consider the case $\alpha > \frac{d}{2\nu+d}$.

By the proof of Lemma \ref{lemmaofnonpara}, it can be seen that 
\begin{align}\label{lemma32identity}
\hat f_n = \argmin_{g\in \mathcal{N}_{\Psi}(\Omega)}\bigg( \frac{1}{n}\sum_{i=1}^n(y_i-g(x_i))^2 + \frac{\hat \mu_n}{n} \|g\|^2_{\mathcal{N}_{\Psi}(\Omega)}\bigg).
\end{align}
In the rest of proof we will write $\hat f_n$ as $\hat f$ for simplification. Plugging \eqref{recovering} into the objective function of \eqref{lemma32identity}, we have
\begin{align}\label{lemma32rela}
 &\frac{1}{n}\sum_{i=1}^n(y_i-\hat f(x_i))^2 + \frac{\hat \mu_n}{n} \|\hat f\|^2_{\mathcal{N}_{\Psi}(\Omega)}\nonumber\\
= & \|f - \hat f\|_n^2 + 2\langle \epsilon, f - \hat f \rangle_n + \frac{1}{n}\sum_{i=1}^n \epsilon_i^2 + \frac{\hat \mu_n}{n} \|\hat f\|^2_{\mathcal{N}_{\Psi}(\Omega)}.
\end{align}
By Lemma \ref{lemmaefsmall}, 
\begin{align*}
|\langle \epsilon, f - \hat f \rangle_n| & \leq tn^{-\frac{1}{2}} \|f - \hat f\|_n^{1 - \frac{d}{2\nu}}\|f - \hat f\|_{\mathcal{N}_{\Psi}(\Omega)}^{\frac{d}{2\nu}},
\end{align*}
with probability at least $1 - C_1\exp(-C_2 t^2)$. By Lemma \ref{lemmavandegeerf} and the triangle inequality, $$\|f - \hat f\|_{\mathcal{N}_{\Psi}(\Omega)} \leq \|f\|_{\mathcal{N}_{\Psi}(\Omega)} +\|\hat f\|_{\mathcal{N}_{\Psi}(\Omega)} \leq C_3.$$
Therefore, \eqref{lemma32rela} can be lower bounded by
\begin{align}\label{lemma32relazai}
&\frac{1}{n}\sum_{i=1}^n(y_i-\hat f(x_i))^2 + \frac{\hat \mu_n}{n} \|\hat f\|^2_{\mathcal{N}_{\Psi}(\Omega)}\nonumber\\
\geq & \|f - \hat f\|_n^2 + \frac{1}{n}\sum_{i=1}^n \epsilon_i^2 + \frac{\hat \mu_n}{n} \|\hat f\|^2_{\mathcal{N}_{\Psi}(\Omega)} - 2C_3 tn^{-\frac{1}{2}} \|f - \hat f\|_n^{1 - \frac{d}{2\nu}}.
\end{align}
By Lemma \ref{lemmavandegeerf} and the interpolation inequality, we have $\|\hat f\|_{\mathcal{N}_{\Psi}(\Omega)} \leq C_4$, and $\|\hat f\|_{L_\infty(\Omega)} \leq c\|\hat f\|^{1-\frac{d}{2\nu}}_{L_2(\Omega)}\|\hat f\|^{\frac{d}{2\nu}}_{H^{\nu}(\Omega)}\leq c_1\|\hat f\|_{\mathcal{N}_{\Psi}(\Omega)}\leq (c_1\vee 1)C_4\leq  C_5$. 

Let $\mathcal{F} = H^{\nu}(C_5)$, where $H^{\nu}(C_5)$ denotes the ball in the Sobolev space $H^{\nu}(\Omega)$ with radius $C_5$. Thus, the bracket entropy number can be bounded by \citep{adams2003sobolev}
\begin{align*}
H_B(\delta_n,\mathcal{F},\|\cdot\|_{L_\infty(\Omega)})\leq C_6\bigg(\frac{1}{\delta_n}\bigg)^{d/\nu},
\end{align*}
and $\hat f \in \mathcal{F}$. Hence, $J_\infty^2(C_5,\mathcal{F})$ in \eqref{defJinfty} can be bounded by
\begin{align}\label{Jinf}
J_\infty^2(C_5,\mathcal{F}) & = C_7^2\inf_{\delta>0}E\bigg[C_5\int_\delta^1\sqrt{H(uC_5/2,\mathcal{F},\|\cdot\|_{L_\infty(\Omega)})}du+\sqrt{n}z\delta \bigg]^2\nonumber\\
& \leq C_8^2 \bigg[C_5\int_0^1\bigg(\frac{1}{uC_5}\bigg)^{d/(2\nu)}du\bigg]^2\nonumber\\
& = C_9^2 C_5^2\bigg(\frac{1}{C_5}\bigg)^{d/\nu}\leq C_{10}.
\end{align}
By Lemma \ref{thm:thm21inGeer2014}, for all $t>0$, with probability at least $1-\exp(-t)$,
\begin{align*}
\sup_{f\in\mathcal{F}}\bigg|\|f\|^2_n-\|f\|^2_{L_2(\Omega)}\bigg|\leq C_{11}\bigg(\frac{2\mathcal{R} \sqrt{C_{10}}+\mathcal{R}C_5\sqrt{t}}{\sqrt{n}}+\frac{4 C_{10}+C_5^2t}{n}\bigg),
\end{align*}
where $\mathcal{R} = \sup_{f\in\mathcal{F}}\|f\|_{L_2(\Omega)}$. Choosing $t = n^{\eta}$, where $\eta = \left(\frac{(\alpha - 1)(2\nu + d)}{2\nu} + 1\right)/4$, for sufficient large $n$, we have that the right-hand side of \eqref{lemma32relazai} can be lower bounded by
\begin{align}\label{finalapprxlemma32lo}
& \|f - \hat f\|_n^2 + \frac{\hat \mu_n}{n} \|\hat f\|^2_{\mathcal{N}_{\Psi}(\Omega)} - 2C_3 n^\eta n^{-\frac{1}{2}} \|f - \hat f\|_n^{1 - \frac{d}{2\nu}} +\frac{1}{n}\sum_{i=1}^n \epsilon_i^2\nonumber\\
\geq & \|f - \hat f\|_{L_2(\Omega)}^2 + \frac{\hat \mu_n}{n} \|\hat f\|^2_{\mathcal{N}_{\Psi}(\Omega)}\nonumber\\
&- C_{12} n^{(\eta-1)/2}R - C_{12} n^{\eta-1} - C_{13} n^\eta n^{-\frac{1}{2}} \|f - \hat f\|_{L_2(\Omega)}^{1 - \frac{d}{2\nu}} \nonumber\\
& - C_{13} n^\eta n^{-\frac{1}{2}} n^{(\eta - 1)\big(1 - \frac{d}{2\nu}\big)/4}R^{\big(1 - \frac{d}{2\nu}\big)} - C_{13}n^\eta n^{-\frac{1}{2}} n^{(\eta - 1)\big(1 - \frac{d}{2\nu}\big)/2}  +\frac{1}{n}\sum_{i=1}^n \epsilon_i^2,
\end{align}
where we also apply Jensen's inequality.

\noindent\textbf{Case 1:} If $2C_3 n^\eta n^{-\frac{1}{2}} \|f - \hat f\|_{L_2(\Omega)}^{1 - \frac{d}{2\nu}} \leq \|f - \hat f\|_{L_2(\Omega)}^2$, then we have $\|f - \hat f\|_{L_2(\Omega)} \gtrsim n^{(2\eta - 1)\frac{\nu}{2\nu + d}}$.

\noindent\textbf{Case 2:} If $2C_3 n^\eta n^{-\frac{1}{2}} \|f - \hat f\|_{L_2(\Omega)}^{1 - \frac{d}{2\nu}} > \|f - \hat f\|_{L_2(\Omega)}^2$, we have $\|f - \hat f\|_{L_2(\Omega)} < C_{14} n^{(2\eta - 1)\frac{\nu}{2\nu + d}}$. Consider function class $\mathcal{G} = \{g:\|g\|_{L_2(\Omega)} \leq C_{14} n^{(2\eta - 1)\frac{\nu}{2\nu + d}}\}\cap \mathcal{F}$, we have $f - \hat f \in \mathcal{G}$ and $R_G = \sup_{g\in\mathcal{G}}\|g\|_{L_2(\Omega)} \leq C_{14} n^{(2\eta - 1)\frac{\nu}{2\nu + d}}$. By \eqref{finalapprxlemma32lo}, we have for sufficient large $n$,
\begin{align}\label{finalapprxlemma32lo2}
& \|f - \hat f\|_n^2 + \frac{\hat \mu_n}{n} \|\hat f\|^2_{\mathcal{N}_{\Psi}(\Omega)} - 2C_3 n^\eta n^{-\frac{1}{2}} \|f - \hat f\|_n^{1 - \frac{d}{2\nu}}\nonumber\\
\geq & \|f - \hat f\|_{L_2(\Omega)}^2 + \frac{\hat \mu_n}{n} \|\hat f\|^2_{\mathcal{N}_{\Psi}(\Omega)} - C_{12} n^{(\eta-1)/2}R_G - C_{12} n^{\eta-1} - C_{13} n^\eta n^{-\frac{1}{2}} R_G^{1 - \frac{d}{2\nu}} \nonumber\\
& - C_{13} n^\eta n^{-\frac{1}{2}} n^{(\eta - 1)\big(1 - \frac{d}{2\nu}\big)/4}R_G^{\big(1 - \frac{d}{2\nu}\big)} - C_{13}n^\eta n^{-\frac{1}{2}} n^{(\eta - 1)\big(1 - \frac{d}{2\nu}\big)/2} \nonumber\\
\geq & \|f - \hat f\|_{L_2(\Omega)}^2 + \frac{\hat \mu_n}{n} \|\hat f\|^2_{\mathcal{N}_{\Psi}(\Omega)} - C_{15}n^\eta n^{-\frac{1}{2}} n^{(2\eta - 1)\frac{\nu}{2\nu + d}(1 - \frac{d}{2\nu})}.
\end{align}

Let 
\begin{align}\label{lemma32pff1def}
f_1 = \argmin_{g\in \mathcal{N}_{\Psi}(\Omega)} \|f - g\|_{L_2(\Omega)}^2 + \frac{\hat \mu_n}{2n} \|g\|^2_{\mathcal{N}_{\Psi}(\Omega)}.
\end{align}
Therefore, by \eqref{lemma32identity}, we have
\begin{align*}
& \|f - \hat f\|_n^2 + 2\langle \epsilon, f - \hat f \rangle_n + \frac{\hat \mu_n}{n} \|\hat f\|^2_{\mathcal{N}_{\Psi}(\Omega)}\nonumber\\
\leq & \|f - f_1\|_n^2 + 2\langle \epsilon, f - f_1 \rangle_n + \frac{\hat \mu_n}{n} \|f_1\|^2_{\mathcal{N}_{\Psi}(\Omega)},
\end{align*}
which, together with \eqref{lemma32relazai} and Lemma \ref{lemmaefsmall}, implies
\begin{align}\label{lemma32f1andfhat1}
&\|f - \hat f\|_n^2 + \frac{\hat \mu_n}{n} \|\hat f\|^2_{\mathcal{N}_{\Psi}(\Omega)} - 2C_3 n^{\eta} n^{-\frac{1}{2}} \|f - \hat f\|_n^{1 - \frac{d}{2\nu}}\nonumber\\
\leq &\|f - f_1\|_n^2 + 2C_{16} n^\eta n^{-\frac{1}{2}} \|f - f_1\|_n^{1 - \frac{d}{2\nu}} + \frac{\hat \mu_n}{n} \|f_1\|^2_{\mathcal{N}_{\Psi}(\Omega)}.
\end{align}
By \eqref{finalapprxlemma32lo2} and \eqref{lemma32f1andfhat1}, we have 
\begin{align}\label{lemma32f1andfhat11111}
&\|f - \hat f\|_{L_2(\Omega)}^2 + \frac{\hat \mu_n}{n} \|\hat f\|^2_{\mathcal{N}_{\Psi}(\Omega)} - C_{15}n^\eta n^{-\frac{1}{2}} n^{(2\eta - 1)\frac{\nu}{2\nu + d}(1 - \frac{d}{2\nu})}\nonumber\\
\leq &\|f - f_1\|_n^2 + 2C_{16} n^\eta n^{-\frac{1}{2}} \|f - f_1\|_n^{1 - \frac{d}{2\nu}} + \frac{\hat \mu_n}{n} \|f_1\|^2_{\mathcal{N}_{\Psi}(\Omega)},
\end{align}

By \eqref{lemma32pff1def}, we have 
\begin{align}\label{lemma32f1xiao2}
\|f -  f_1\|_{L_2(\Omega)}^2 + \frac{\hat \mu_n}{2n} \|f_1\|^2_{\mathcal{N}_{\Psi}(\Omega)} \leq  \frac{ \hat \mu_n}{2n} \|f\|^2_{\mathcal{N}_{\Psi}(\Omega)} \leq C_{17}\frac{ \hat \mu_n}{n} .
\end{align}
By Lemma \ref{Berforsingleg}, it can be shown that with probability at least $1 - \exp(-C_{18}n^{\eta})$, 
\begin{align}\label{finalapprxlemma32up}
& \|f - f_1\|_n^2 + 2C_{16} n^{\eta_1} n^{-\frac{1}{2}} \|f - f_1\|_n^{1 - \frac{d}{2\nu}} + \frac{\hat \mu_n}{n} \|f_1\|^2_{\mathcal{N}_{\Psi}(\Omega)}\nonumber\\
\leq & \|f - f_1\|_{L_2(\Omega)}^2 + \frac{\hat \mu_n}{n} \|f_1\|^2_{\mathcal{N}_{\Psi}(\Omega)}\nonumber\\
& + C_{17} n^\eta n^{-\frac{1}{2}} \|f - f_1\|_{L_2(\Omega)}^{1 - \frac{d}{2\nu}}  + C_{17} n^\eta n^{-\frac{1}{2}} n^{(\eta - 1)\big(1 - \frac{d}{2\nu}\big)/2} + C_{17} n^{\eta-1}\nonumber\\
\leq & \|f - f_1\|_{L_2(\Omega)}^2 + \frac{\hat \mu_n}{n} \|f_1\|^2_{\mathcal{N}_{\Psi}(\Omega)}\nonumber\\
& + C_{17} n^\eta n^{-\frac{1}{2}} \|f - f_1\|_{L_2(\Omega)}^{1 - \frac{d}{2\nu}}  + 2C_{17} n^\eta n^{-\frac{1}{2}} n^{(\eta - 1)\big(1 - \frac{d}{2\nu}\big)/2}.
\end{align}
Combining \eqref{lemma32f1andfhat11111} and \eqref{finalapprxlemma32up} yields
\begin{align}\label{lemma32f1andfhat2}
&\|f - \hat f\|_{L_2(\Omega)}^2 + \frac{\hat \mu_n}{n} \|\hat f\|^2_{\mathcal{N}_{\Psi}(\Omega)} - C_{15}n^\eta n^{-\frac{1}{2}} n^{(2\eta - 1)\frac{\nu}{2\nu + d}(1 - \frac{d}{2\nu})}\nonumber\\
\leq & \|f - f_1\|_{L_2(\Omega)}^2 + \frac{\hat \mu_n}{n} \|f_1\|^2_{\mathcal{N}_{\Psi}(\Omega)}\nonumber\\
& + C_{17} n^\eta n^{-\frac{1}{2}} \|f - f_1\|_{L_2(\Omega)}^{1 - \frac{d}{2\nu}} + 2C_{17} n^\eta n^{-\frac{1}{2}} n^{(\eta - 1)\big(1 - \frac{d}{2\nu}\big)/2}.
\end{align}
By \eqref{lemma32pff1def}, we have 
\begin{align*}
\|f -  f_1\|_{L_2(\Omega)}^2 + \frac{\hat \mu_n}{2n} \|f_1\|^2_{\mathcal{N}_{\Psi}(\Omega)} \leq  \|f -  \hat f\|_{L_2(\Omega)}^2 + \frac{\hat \mu_n}{2n} \|\hat f\|^2_{\mathcal{N}_{\Psi}(\Omega)},
\end{align*}
which implies
\begin{align}\label{lemma32f1xiao}
    2\|f -  f_1\|_{L_2(\Omega)}^2 + \frac{\hat \mu_n}{n} \|f_1\|^2_{\mathcal{N}_{\Psi}(\Omega)} \leq  2\|f -  \hat f\|_{L_2(\Omega)}^2 + \frac{\hat \mu_n}{n} \|\hat f\|^2_{\mathcal{N}_{\Psi}(\Omega)}.
\end{align}

Combining \eqref{lemma32f1xiao2}, \eqref{lemma32f1andfhat2} and \eqref{lemma32f1xiao}, we have for sufficient large $n$,
\begin{align}\label{lemma32f1andfhatfi}
\|f - \hat f\|_{L_2(\Omega)}^2 \geq & \|f - f_1\|_{L_2(\Omega)}^2 - C_{15}n^\eta n^{-\frac{1}{2}} n^{(2\eta - 1)\frac{\nu}{2\nu + d}(1 - \frac{d}{2\nu})} \nonumber\\
&  - C_{17} n^\eta n^{-\frac{1}{2}} \|f - f_1\|_{L_2(\Omega)}^{1 - \frac{d}{2\nu}}- 2C_{17} n^\eta n^{-\frac{1}{2}} n^{(\eta - 1)\big(1 - \frac{d}{2\nu}\big)/2}\nonumber\\
\geq &  \|f - f_1\|_{L_2(\Omega)}^2 - C_{15}n^\eta n^{-\frac{1}{2}} n^{(2\eta - 1)\frac{\nu}{2\nu + d}(1 - \frac{d}{2\nu})}\nonumber\\
&  - C_{17} n^\eta n^{-\frac{1}{2}} n^{\big(1 - \frac{d}{2\nu}\big)(\alpha - 1)/2} - 2C_{17} n^\eta n^{-\frac{1}{2}} n^{(\eta - 1)\big(1 - \frac{d}{2\nu}\big)/2}\nonumber\\
\geq &  \|f - f_1\|_{L_2(\Omega)}^2 - 4C_{15}n^\eta n^{-\frac{1}{2}} n^{(2\eta - 1)\frac{\nu}{2\nu + d}(1 - \frac{d}{2\nu})},
\end{align}
where the last inequality is because $(2\eta - 1)\frac{2\nu}{2\nu + d} > \alpha - 1$.
Let 
\begin{align}\label{lemma32pfftildedef}
\tilde f = \argmin_{g\in \mathcal{N}_{\Psi}(\RR^d)} \|f - g\|_{L_2(\RR^d)}^2 + \frac{\hat \mu_n}{C_{18} n} \|g\|^2_{\mathcal{N}_{\Psi}(\RR^d)},
\end{align}
where $C_{18}$ is a constant determined later. 
By \eqref{lemma32pff1def} and the definition of $\|\cdot \|^2_{\mathcal{N}_{\Psi}(\Omega)}$, we have
\begin{align}\label{lemma32pff1andf21}
 \|f - f_1\|_{L_2(\Omega)}^2 + \frac{\hat \mu_n}{2n} \|f_1\|^2_{\mathcal{N}_{\Psi}(\Omega)} & \leq  \|f - \tilde f\|_{L_2(\Omega)}^2 + \frac{\hat \mu_n}{2n} \|\tilde f\|^2_{\mathcal{N}_{\Psi}(\Omega)}\nonumber\\
 & \leq \|f - \tilde f\|_{L_2(\RR^d)}^2 + \frac{\hat \mu_n}{2n} \|\tilde f\|^2_{\mathcal{N}_{\Psi}(\RR^d)}.
\end{align}
By \eqref{lemma32pfftildedef} and the extension theorem (we still use $f_1$ to denote the extension of $f_1$ for notational simplicity), 
\begin{align}\label{lemma32pff1andf22}
& \|f - \tilde f\|_{L_2(\RR^d)}^2 + \frac{\hat \mu_n}{C_{18} n} \|\tilde f\|^2_{\mathcal{N}_{\Psi}(\RR^d)}\nonumber\\
\leq & \|f - f_1\|_{L_2(\RR^d)}^2 + \frac{\hat \mu_n}{C_{18} n} \|f_1\|^2_{\mathcal{N}_{\Psi}(\RR^d)}\nonumber\\
\leq & \frac{C_{19}}{C_{18}}(C_{18}\|f - f_1\|_{L_2(\Omega)}^2 + \frac{\hat \mu_n}{ n} \|f_1\|^2_{\mathcal{N}_{\Psi}(\Omega)})\nonumber\\
\leq & \frac{C_{19}}{C_{18}}\bigg((C_{18} - 1)\|f - f_1\|_{L_2(\Omega)}^2 + \|f - \tilde f\|_{L_2(\Omega)}^2 + \frac{\hat \mu_n}{n} \|\tilde f\|^2_{\mathcal{N}_{\Psi}(\RR^d)}\bigg),
\end{align}
where $C_{19} > 1$ is a constant. Therefore, combining \eqref{lemma32pff1andf21} and \eqref{lemma32pff1andf22} yields 
\begin{align}\label{lemma32f1andf2fi}
\|f - f_1\|_{L_2(\Omega)}^2 \geq \frac{1}{C_{18} - 1}\bigg((1 - \frac{C_{19}}{C_{18}}) \|f - \tilde f\|_{L_2(\RR^d)}^2 -  \frac{\hat \mu_n(C_{19} - 1)}{C_{18} n} \|\tilde f\|^2_{\mathcal{N}_{\Psi}(\RR^d)} \bigg)
\end{align}
Next, we calculate $\|f - \tilde f\|_{L_2(\RR^d)}^2$ and $\|\tilde f\|^2_{\mathcal{N}_{\Psi}(\RR^d)}$ with respect to $f$.

By Fourier transform and \eqref{lemma32pfftildedef},
\begin{align*}
& \|f-\tilde f\|_{L_2(\RR^d)}^2 + \frac{\hat \mu_n}{C_{18} n} \|\tilde f\|^2_{\mathcal{N}_{\Psi}(\RR^d)}\nonumber\\
= &  \int_{\RR^d}|\mathcal{F}(f)(\omega)-\mathcal{F}(\tilde f)(\omega)|^2 d\omega + \frac{\hat \mu_n}{C_{18} n} \int_{\RR^d}|\mathcal{F}(\tilde f)(\omega)|^2(1+|\omega|^2)^{\nu}d\omega\nonumber\\
= &  \int_{\RR^d}|\mathcal{F}(f)(\omega)-\mathcal{F}(\tilde f)(\omega)|^2 + \frac{\hat \mu_n}{C_{18} n} |\mathcal{F}(\tilde f)(\omega)|^2(1+|\omega|^2)^{\nu}d\omega\nonumber\\
= &  \int_{\RR^d}\frac{\frac{\hat \mu_n}{C_{18} n} (1+|\omega|^2)^{\nu}}{1+ \frac{\hat \mu_n}{C_{18} n} (1+|\omega|^2)^{\nu}}|\mathcal{F}(f)(\omega)|^2d\omega,
\end{align*}
where 
\begin{align*}
\mathcal{F}(\tilde f)(\omega) = \frac{\mathcal{F}(f)(\omega)}{1 + \frac{ \hat \mu_n}{C_{18} n} (1+|\omega|^2)^{\nu}}.
\end{align*}
Therefore,
\begin{align}\label{lemma32pfFtildeF2norm}
\int_{\RR^d}|\mathcal{F}(f)(\omega)-\mathcal{F}(\tilde f)(\omega)|^2 d\omega = \int_{\RR^d}|\mathcal{F}(f)(\omega)|^2 \frac{(\frac{ \hat \mu_n}{C_{18} n} (1+|\omega|^2)^{\nu})^2}{(1 + \frac{\hat \mu_n}{C_{18} n} (1+|\omega|^2)^{\nu})^2} d\omega, 
\end{align}
and
\begin{align}\label{lemma32pfFtildeFnanorm}
\frac{\hat \mu_n}{C_{18} n} \int_{\RR^d}|\mathcal{F}(\tilde f)(\omega)|^2(1+|\omega|^2)^{\nu}d\omega = \int_{\RR^d}|\mathcal{F}(f)(\omega)|^2 \frac{\frac{\hat \mu_n}{C_{18}n} (1+|\omega|^2)^{\nu}}{(1 + \frac{\hat \mu_n}{C_{18} n} (1+|\omega|^2)^{\nu})^2} d\omega.
\end{align}
Let $h(|\omega|) = \frac{\hat \mu_n}{C_{18}n} (1+|\omega|^2)^{\nu}$ and $C_{18} = C_{19}^2$. Plugging \eqref{lemma32pfFtildeF2norm} and \eqref{lemma32pfFtildeFnanorm} into \eqref{lemma32f1andf2fi}, we have
\begin{align}\label{lemma32f1lowerbd}
& \|f - f_1\|_{L_2(\Omega)}^2 \nonumber\\
\geq & \frac{1}{C_{18} - 1}\int_{\RR^d}\frac{|\mathcal{F}(f)(\omega)|^2h(|\omega|)}{(1 + h(|\omega|))^2}\bigg( \frac{C_{18} - C_{19}}{C_{18}} h(|\omega|) - (C_{19} - 1) \bigg) d\omega\nonumber\\
\geq &  \frac{1}{C_{19} + 1}\int_{\RR^d}\frac{|\mathcal{F}(f)(\omega)|^2h(|\omega|)}{(1 + h(|\omega|))^2}( h(|\omega|)/C_{19} - 1 ) d\omega\nonumber\\
= &   \frac{1}{C_{19} + 1}\int_{\{\omega:h(|\omega|) \leq 2C_{19}\}} \frac{|\mathcal{F}(f)(\omega)|^2h(|\omega|)}{(1 + h(|\omega|))^2}( h(|\omega|)/C_{19} - 1 ) d\omega \nonumber\\
& + \frac{1}{C_{19} + 1}\int_{\{\omega:h(|\omega|) > 2C_{19}\}} \frac{|\mathcal{F}(f)(\omega)|^2h(|\omega|)}{(1 + h(|\omega|))^2}( h(|\omega|)/C_{19} - 1 ) d\omega \nonumber\\
\geq &   \frac{1}{C_{19} + 1}\int_{\{\omega:h(|\omega|) \leq 2C_{19}\}} \frac{|\mathcal{F}(f)(\omega)|^2h(|\omega|)}{(1 + h(|\omega|))^2}( h(|\omega|)/C_{19} - 1 ) d\omega \nonumber\\
& + \frac{1}{4(C_{19} + 1)}\int_{\{\omega:h(|\omega|) > 2C_{19}\}} \frac{|\mathcal{F}(f)(\omega)|^2}{2C_{19}} d\omega\nonumber\\
\geq &   - \frac{1}{C_{19} + 1}\int_{\{\omega:h(|\omega|) \leq 2C_{19}\}} |\mathcal{F}(f)(\omega)|^2 d\omega\nonumber\\
& + \frac{1}{4(C_{19} + 1)}\int_{\{\omega:h(|\omega|) > 2C_{19}\}} \frac{|\mathcal{F}(f)(\omega)|^2}{2C_{19}} d\omega.
\end{align}
Now we can build our function $f$. Let 
\begin{align*}
\mathcal{F}(f)(\omega) = \left\{\begin{array}{cc}
    0 & h(|\omega|) \leq 2C_{19}-\delta_n, \\
    g_1(|\omega|) & 2C_{19}-\delta_n \leq h(|\omega|) \leq 2C_{19},\\
    (1 + |\omega|^2)^{-\nu/2 - d/2 - 1} & h(|\omega|) > 2C_{19},
\end{array}\right.
\end{align*}
where $g_1$ and $\delta_n$ are chosen such that $f$ is continuous and 
\begin{align*}
 & \frac{1}{C_{19} + 1}\int_{\{\omega:2C_{19}-\delta_n \leq h(|\omega|) \leq 2C_{19}\}} |\mathcal{F}(f)(|\omega|)|^2 d\omega\nonumber\\
 < & \frac{1}{8(C_{19} + 1)}\int_{\{\omega:h(|\omega|) > 2C_{19}\}} \frac{|\mathcal{F}(f)(|\omega|)|^2}{2C_{19}}.
\end{align*}
Then we normalize $f$ such that $\|f\|_{\mathcal{N}_{\Psi}(\RR^d)} = 1$. Therefore, by \eqref{lemma32f1lowerbd}, direct calculation shows that 
\begin{align*}
\|f - f_1\|_{L_2(\Omega)}^2 \geq C_{20}\frac{\hat \mu_n}{n}.
\end{align*}
By \eqref{lemma32f1andfhatfi} and $(2\eta - 1)\frac{\nu}{2\nu + d}(1 - \frac{d}{2\nu}) + \eta - 1/2 < \alpha -1$, we have 
\begin{align*}
\|f - \hat f\|_{L_2(\Omega)}^2 \geq C_{21}\frac{\hat \mu_n}{n}.
\end{align*}
Note that $(2\eta - 1)\frac{2\nu}{2\nu + d} > \alpha - 1$, which leads to a contradiction of Case 2 as $n$ increases. 

Using this constructed $f$, it can be seen that 
\begin{align*}
    \mathbb{E}_\epsilon\|f-\hat f\|_{L_{2}(\Omega)}^2 \geq C_6 n^{-(1 - 2\eta)\frac{2\nu}{2\nu + d}}.
\end{align*}
Therefore, we finish the proof of the case $\alpha > \frac{d}{2\nu + d}$.

Next, we prove the case $\alpha<\frac{d}{2\nu + d}$. We consider the case $0\leq \alpha<\frac{d}{2\nu + d}$.

By Fubini's theorem,
\begin{align}\label{exerrorsum}
& \mathbb{E}\|f-\hat f_n\|_{L_{2}(\Omega)}^2\nonumber\\
= & \mathbb{E}\int_\Omega (f(x) - r(x)^T(R+ \hat \mu_n I_n)^{-1}f(X) - r(x)^T(R+ \hat \mu_n I_n)^{-1}\epsilon)^2 dx\nonumber\\
= & \sigma_\epsilon^2 \int_\Omega r(x)^T(R+ \hat \mu_n I_n)^{-2}r(x)dx + \int_\Omega (f(x) - r(x)^T(R+ \hat \mu_n I_n)^{-1}f(X))^2 dx\nonumber\\
\geq & \sigma_\epsilon^2 \int_\Omega r(x)^T(R+ \hat \mu_n I_n)^{-2}r(x)dx := \sigma_\epsilon^2 I.
\end{align}
We consider a discrete version of $I$. Let $I_n = \text{tr} (R^2 (R+ \hat \mu_n I_n)^{-2})$. Let $p = \lfloor(n/\hat \mu_n)^{d/(2\nu)}\rfloor$, where $\lfloor \cdot \rfloor$ is the floor function, and $p_1 = \min\{p, C_1 n^{1/2}\}$. Let $\Psi_1=\frac{1}{\sqrt{n}}(\varphi_1(X),...,\varphi_{p_1}(X))$, and $\Psi_2=\frac{1}{\sqrt{n}}(\varphi_{p_1+1}(X),\varphi_{p_1+2}(X),...)$, where $\varphi_k(X)=(\varphi_k(x_1),...,\varphi_k(x_n))^T$ for $k=1,2,...$, and $\varphi_{k}$'s are as in \eqref{eq:AppJ1eq1}. Let $\Lambda_1=\mbox{diag}(n\lambda_1,...,n\lambda_p)$ and $\Lambda_2=\mbox{diag}(n\lambda_{p+1},...)$, where $\lambda_k$'s are as in \eqref{eq:AppJ1eq1}. Therefore, $R=\sum_{k=1}^\infty \lambda_i\varphi_k(X)\varphi_k(X)^T=\Psi_1\Lambda_1\Psi_1^T+\Psi_2\Lambda_2\Psi_2^T$. 

By Lemma \ref{tracelemma},
\begin{align}\label{thm44lbfangcha}
I_n \geq & \text{tr}((\Psi_1\Lambda_1\Psi_1^T)^2(\Psi_1\Lambda_1\Psi_1^T + \hat \mu_n I)^{-2}) \nonumber\\
= & \sum_{i = 1}^{p_1} \bigg(\frac{\lambda_i(\Psi_1\Lambda_1\Psi_1^T)}{\lambda_i(\Psi_1\Lambda_1\Psi_1^T) + \hat \mu_n}\bigg)^2,
\end{align}
where $\lambda_i(\Psi_1\Lambda_1\Psi_1^T)$ denote the $i$-th eigenvalue of $\Psi_1\Lambda_1\Psi_1^T$. Note that the $i$-th eigenvalue $\lambda_i(\Psi_1\Lambda_1\Psi_1^T) = \lambda_i(\Psi_1^T\Psi_1\Lambda_1)$ for $i=1,...,p$,  because if $u_i$ is eigenvector corresponding to $i$-th eigenvalue of $\Psi_1\Lambda_1\Psi_1^T$, then 
\begin{align*}
\Psi_1\Lambda_1\Psi_1^Tu_i = \lambda_i u_i \Rightarrow  \Psi_1^T\Psi_1\Lambda_1\Psi_1^Tu_i = \lambda_i \Psi_1^Tu_i.
\end{align*}
Therefore, \eqref{thm44lbfangcha} implies
\begin{align*}
I_n \geq & \text{tr}((\Lambda_1\Psi_1^T\Psi_1)^2(\Lambda_1\Psi_1^T\Psi_1 + \hat \mu_n I)^{-2}) \nonumber\\
= & \text{tr}(\Lambda_1^2(\Lambda_1 + \hat \mu_n (\Psi_1^T\Psi_1)^{-1})^{-2}).
\end{align*}
By Lemma \ref{boundsofdeterminant}, it can be shown that $\lambda_{\min}(\Psi_1^T\Psi_1) \geq \eta_1$ with probability at least $1 - C_{22} \exp(-C_{23}n^{\eta_2})$, where the constant in the expression of $p_1$ is chosen such that the condition of Lemma \ref{lemmaratio1} is satisfied. Combining this with Lemma \ref{lemDecayEig}, we have $I_n \geq C_{24} p_1$.

Notice that for any $u = (u_1,...,u_n)^T \in \RR^n$,
\begin{align*}
1 - 2\sum_{i=1}^n u_i\Psi(x - x_i) + \sum_{i=1}^n\sum_{j=1}^n u_i u_j\Psi(x_i - x_j) + \hat \mu_n \|u\|_2^2 \geq \hat \mu_n \|u\|_2^2.
\end{align*}
Plugging $u = (R + \hat \mu_n I_n)^{-1}r(x)$, we have 
\begin{align}\label{u2smallthan}
    \hat \mu_nr(x)^T(R + \hat \mu_n I_n)^{-2}r(x) \leq 1 - r(x)^T(R + \hat \mu_n I_n)^{-1}r(x).
\end{align}
By Lemma \ref{lemmaerrorwnug}, with probability at least $1 - C_{25} \exp(-C_{26}n^{\eta_2})$, $ 1 - r(x)^T(R + \hat \mu_n I_n)^{-1}r(x) \leq C_{27} \left(\frac{\hat \mu_n}{n}\right)^{\frac{2\nu - d}{2\nu}}$.

Let $\mathcal{H}_1 = \{h|h(x)^2 =  r(x)^T(R + \hat \mu_n I_n)^{-2}r(x), \text{with } x_i \in \Omega, i=1,...,n\}$. Let $\mathcal{H} = \mathcal{H}_1\bigcap \{\|h\|_{L_\infty(\Omega)}^2 \leq C_{27} \left(\frac{\hat \mu_n}{n}\right)^{\frac{2\nu - d}{2\nu}}/\hat \mu_n\}$. It can be seen from \eqref{u2smallthan} that with probability at least $1 - C_{25} \exp(-C_{26}n^{\eta_3})$, $\mathcal{H}$ is true. It can be also seen that $\|h_1\|_n^2 = \frac{1}{n}\text{tr} (R^2 (R+ \hat \mu_n I_n)^{-2})$ for some $h_1(x) = r(x)^T(R+ \hat \mu_n I_n)^{-2}r(x) \in \mathcal{H}_1$.

By Lemma \ref{thm:thm21inGeer2014}, we have with another probability at least $1-\exp(-n^{\eta_3})$ with $\eta_3 = (1 - \alpha)\frac{2\nu - d}{2d}$,
\begin{align*}
\sup_{h\in\mathcal{H}_1}\bigg|\|h\|^2_n-\|h\|^2_{L_2(\Omega)}\bigg|\leq C_{7}n^{-\frac{1}{2} + \left((\alpha - 1)\frac{2\nu - d}{2\nu}\right)\left(2 - \frac{d}{2\nu}\right)},
\end{align*}
where $J_\infty^2(\sqrt{C_{27}} \left(\frac{\hat \mu_n}{n}\right)^{\frac{2\nu - d}{4\nu}}/\hat \mu_n^{1/2},\mathcal{F})$ can be calculate similarly.
Therefore, we have
\begin{align}\label{thm44pffhatfndiff}
\|h\|^2_{L_2(\Omega)} \geq \|h\|_n^2 -  C_{27}n^{-\frac{1}{2} + \left((\alpha - 1)\frac{2\nu - d}{2\nu}\right)\left(2 - \frac{d}{2\nu}\right)}.
\end{align}
Note 
\begin{align*}
& \|h\|_n^2 = I_n/n \geq C_4 p_1/n \geq C_8\min\{n^{-1 + (1 - \alpha)\frac{d}{2\nu}},n^{-1/2}\}\nonumber\\
> & n^{-\frac{1}{2} + \left((\alpha - 1)\frac{2\nu - d}{2\nu}\right)\left(2 - \frac{d}{2\nu}\right)}
\end{align*}
with probability at least $1 - C_{22} \exp(-C_{23}n^{\eta_2})$. Therefore, combining all probabilities together and by \eqref{thm44pffhatfndiff}, with probability at least $1 - C_{28} \exp(-C_{29}n^{\eta_4})$, we have $\|h\|^2_{L_2(\Omega)} \geq  C_{30} p_1/n$, which finishes the proof of the case $\alpha \in [0,\frac{d}{2\nu +d})$. If $\alpha < 0$, then from \eqref{exerrorsum} it can be seen that the error is larger than choosing $\alpha = 0$. Thus, we complete the proof.

\section{Proof of Theorem \ref{COROCISTO}}

In the proof of Theorem \ref{COROCISTO}, we hide all the probabilities with the form $1-C\exp(-C'n^{-\eta})$ for the conciseness of the proof.

\noindent\textbf{Case 1: $\alpha > \frac{d}{2\nu + d}$.}

By \eqref{exerrorsum} and Lemma \ref{lemmaofnonpara}, we have 
\begin{align*}
& \mathbb{E}_\epsilon\|(f-\hat f_n)/\tilde c_{n,\beta}(\cdot;\hat \mu_n)\|_{L_{2}(\Omega)}^2 \geq C_1\frac{\mathbb{E}_\epsilon\|f-\hat f_n\|_{L_{2}(\Omega)}^2}{\mathbb{E}_\epsilon\|\tilde c_{n,\beta}(\cdot;\hat \mu_n)\|_{L_{2}(\Omega)}^2}\nonumber\\
\geq & C_2 \frac{\mathbb{E}_\epsilon\int_\Omega \hat \mu_n(f(x) - r(x)^T(R+ \hat \mu_n I_n)^{-1}f(X) - r(x)^T(R+ \hat \mu_n I_n)^{-1}\epsilon)^2dx}{\int_\Omega (1 - r(x)^T(R + \hat \mu_n I_n)^{-1}r(x))dx}\nonumber\\
= & C_2 \frac{\int_\Omega\hat \mu_n (r(x)^T(R+ \hat \mu_n I_n)^{-2}r(x) + (f(x) - r(x)^T(R+ \hat \mu_n I_n)^{-1}f(X))^2)dx}{\int_\Omega1 - r(x)^T(R + \hat \mu_n I_n)^{-1}r(x)dx}\nonumber\\
\geq&  C_3 n^{(1 - \alpha)(1 - \frac{d}{2\nu})} \hat\mu_n\int_\Omega (f(x) - r(x)^T(R+ \hat \mu_n I_n)^{-1}f(X))^2dx,
\end{align*}
where $r(x)$ and $R$ are as in \eqref{mean}, and $f(X) = (f(x_1),...,f(x_n))^T$. The first inequality is true because of the Cauchy-Schwarz inequality and that $f-\hat f_n$ is normal. The last inequality follows Lemma \ref{lemmaerrorwnug}. If $\alpha > \frac{d}{2\nu + d}$, then by the proof of Theorem \ref{NONOPTFHAT}, we have $n^{(1 - \alpha)(1 - \frac{d}{2\nu})} \hat\mu_n\int_\Omega (f(x) - r(x)^T(R+ \hat \mu_n I_n)^{-1}f(X))^2 \gtrsim n^{\beta}$ with $\beta > 0$ for $f \in \mathcal{N}_{\Psi}(\Omega)$ constructed in the proof of Theorem \ref{NONOPTFHAT}, which finishes the proof of Case 1.

\noindent\textbf{Case 2: $0 \leq \alpha < \frac{d}{2\nu + d}$.}

First, we prove \eqref{mainpropstoeqexp2}. By the proof of Theorem \ref{NONOPTFHAT} and Lemma \ref{lemmaofnonpara}, it can be shown that $$ \hat \mu_n\mathbb{E}_\epsilon\|\tilde c_{n,\beta}(\cdot;\hat \mu_n)\|_{L_{2}(\Omega)}^2 \gtrsim n^{(\alpha - 1)(1 - \frac{d}{2\nu})}.$$

Noting that $f-\hat f_n$ is normal, we have
\begin{align}\label{upEci1bias}
    & \hat \mu_n\mathbb{E}_\epsilon\|f-\hat f_n\|_{L_{2}(\Omega)}^2\nonumber\\
    \leq & C_1 \hat \mu_n\int_\Omega r(x)^T(R+ \hat \mu_n I_n)^{-2}r(x) + (f(x) - r(x)^T(R+ \hat \mu_n I_n)^{-1}f(X))^2 dx\nonumber\\
    \leq & 2C_1 \hat \mu_n \int_\Omega r(x)^T(R+ \hat \mu_n I_n)^{-2}r(x) + (f(x) - r(x)^TR^{-1}f(X))^2\nonumber\\
     & + ( r(x)^TR^{-1}f(X) - r(x)^T(R+ \hat \mu_n I_n)^{-1}f(X))^2 dx\nonumber\\
     \leq & 2C_1 \hat \mu_n \int_\Omega r(x)^T(R+ \hat \mu_n I_n)^{-2}r(x) + (1- r(x)^TR^{-1}r(X))\log n\nonumber\\
     & + ( r(x)^TR^{-1}f(X) - r(x)^T(R+ \hat \mu_n I_n)^{-1}f(X))^2 dx\nonumber\\
     \leq & 2C_1 \hat \mu_n \int_\Omega r(x)^T(R+ \hat \mu_n I_n)^{-2}r(x) + (1- r(x)^TR^{-1}r(x))\log n\nonumber\\
     & + ( r(x)^TR^{-1}f(X) - r(x)^T(R+ C_2 n^{\frac{d}{2\nu + d}} I_n)^{-1}f(X))^2 dx\nonumber\\
     \leq & 2C_1 \int_\Omega (1-r(x)^T(R+ \hat \mu_n I_n)^{-1}r(x)) + \hat \mu_n(1- r(x)^TR^{-1}r(x))\log n\nonumber\\
     & + \hat \mu_n( r(x)^TR^{-1}f(X) - r(x)^T(R+ C_2 n^{\frac{d}{2\nu + d}} I_n)^{-1}f(X))^2 dx,
\end{align}
where the first inequality is by Fubini's theorem, the second inequality is by the Cauchy-Schwarz inequality, the third inequality is by \eqref{firstestimate} and the fact $\|f\|^2_{\mathcal{N}_{\Psi}(\Omega)} \leq \log n$, the fourth inequality is by $\hat \mu_n\leq C_2 n^{\frac{d}{2\nu + d}}$, and the fifth inequality is by \eqref{u2smallthan}.
By Lemma \ref{Th:Matern} and Lemma \ref{lemmaerrorwnug}, \eqref{upEci1bias} can be further bounded by
\begin{align*}
     & \hat \mu_n\mathbb{E}_\epsilon\|f-\hat f_n\|_{L_{2}(\Omega)}^2\nonumber\\
     \leq & C_2 n^{(\alpha - 1)(1 - \frac{d}{2\nu})} + n^{\alpha-\frac{\nu}{d}+1/2}\log n\nonumber\\
     & + \hat \mu_n\int_\Omega( r(x)^TR^{-1}f(X) - r(x)^T(R+ C_2 n^{\frac{d}{2\nu + d}} I_n)^{-1}f(X))^2 dx.
\end{align*}
Let $f_1(x) =  (r(x)^TR^{-1}f(X) - r(x)^T(R+ C_2 n^{\frac{d}{2\nu + d}} I_n)^{-1}f(X))/\sqrt{\log n}$. Then $\|f_1\|^2_{\mathcal{N}_{\Psi}(\Omega)} \leq 1$, which implies either $\|f_1\|_{L_2(\Omega)}^2 \leq C_3 n^{-\frac{2\nu}{2\nu + d}}$ or the conditions of Lemma \ref{lemmaratio1} are satisfied. If the later happens, by Lemma \ref{lemmaratio1}, it can be shown that 
\begin{align*}
\|f_1\|_{L_2(\Omega)}^2 \leq & \eta_2 \|f_1\|_n^2\nonumber\\
\leq & \eta_2(\|r(x)^TR^{-1}f(X) - r(x)^T(R+ C_2 n^{\frac{d}{2\nu + d}} I_n)^{-1}f(X)\|_n^2\nonumber\\
& +C_2n^{-\frac{2\nu}{2\nu + d}}\| r(x)^T(R+ C_2 n^{\frac{d}{2\nu + d}} I_n)^{-1}f(X)\|^2_{\mathcal{N}_{\Psi}(\Omega)})\nonumber\\
\leq & C_4n^{-\frac{2\nu}{2\nu + d}}.
\end{align*}
By noticing that $n^{\alpha-\frac{\nu}{d}+1/2}\log n + n^{-\frac{2\nu}{2\nu + d}} n^\alpha\log n \leq C_5n^{(\alpha - 1)(1 - \frac{d}{2\nu})}$, we finish the proof of the first part.

For the second part, by the standard minimax theory in nonparametric regression, there exists a function $f$ such that $\mathbb{E}_\epsilon\|f-\hat f_n\|_{L_{2}(\Omega)}^2 \geq C_6 n^{-\frac{2\nu}{2\nu + d}}\|f\|^2_{\mathcal{N}_{\Psi}(\Omega)}$, which implies
\begin{align*}
& \mathbb{E}_\epsilon\|(f-\hat f_n)/\tilde c_{n,\beta}(\cdot;\hat \mu_n)\|_{L_{2}(\Omega)}^2 \geq C_7\frac{\mathbb{E}_\epsilon\|f-\hat f_n\|_{L_{2}(\Omega)}^2}{\mathbb{E}_\epsilon\|\tilde c_{n,\beta}(\cdot;\hat \mu_n)\|_{L_{2}(\Omega)}^2}\nonumber\\
\geq&  C_8 n^{(1 - \alpha)(1 - \frac{d}{2\nu})} \hat \mu_n\mathbb{E}_\epsilon\|f-\hat f_n\|_{L_{2}(\Omega)}^2\nonumber\\
\geq&  C_9 a_n.
\end{align*}

\noindent\textbf{Case 3: $\alpha < 0$.}

By Lemma \ref{lemmaofnonpara}, we have 
\begin{align*}
    & \mathbb{E}_\epsilon\|(f-\hat f_n)/\tilde c_n(\cdot,\beta; \hat \mu_n)\|_{L_{2}(\Omega)}^2\nonumber\\
    \leq & C_{10} \hat \mu_n\int_\Omega \frac{r(x)^T(R+ \hat \mu_n I_n)^{-2}r(x) + (f(x) - r(x)^T(R+ \hat \mu_n I_n)^{-1}f(X))^2}{(1 - r(x)^T(R + \hat \mu_n I_n)^{-1}r(x))} dx\nonumber\\
    \leq&  C_{10} + C_{10} \hat \mu_n\int_\Omega \frac{(f(x) - r(x)^T(R+ \hat \mu_n I_n)^{-1}f(X))^2}{(1 - r(x)^T(R + \hat \mu_n I_n)^{-1}r(x))} dx\nonumber\\
    \leq&  C_{10} + C_{10} \hat \mu_n\log n.
\end{align*}
The second inequality is true because of \eqref{u2smallthan}, and the third inequality is true because of Lemma \ref{lemmaerrorwnug}. Note $\hat \mu_n\log n\rightarrow 0$, which finishes the proof of the case $\alpha < 0$. 

\section{Proof of Lemma \ref{lemmaratio1}}\label{pflemmaratio1}
The idea of the proof is to use the bracket entropy number. By the definition of the bracket entropy number, we can find finite functions $g_s$'s such that the ball with small radius centered on $g_s$'s can cover the function class $\mathcal{G}$. By showing the results hold for these $g_s$'s, we can show the results hold for all function $g\in \mathcal{G}$.

Take $g\in \mathcal{G}$, and suppose that $s\delta_n\leq \|g\|_{L_2(\Omega)}\leq (s+1)\delta_n$, where $s\in\{2,3,...\}$. Let $-K\leq g_L \leq g \leq g_U \leq K$, and $\|g_U-g_L\|_{L_\infty(\Omega)}\leq \delta_n/\text{Vol}(\Omega)$, for functions $g_L$ and $g_U$. For $0 < C \leq \frac{1}{4\text{Vol}(\Omega)}$, by Cauchy-Schwarz inequality, we have
\begin{align*}
g_L^2 & \leq 2g^2/C+2C(g-g_L)^2 \leq 2g^2/C+2C\delta_n^2/\text{Vol}(\Omega)^2,
\end{align*}
which implies
\begin{align*}
2\|g\|_n^2 \geq C\|g_L\|^2_n-2C^2\delta_n^2/\text{Vol}(\Omega)^2.
\end{align*}
The inequality $\|g\|^2_n/\|g\|_{L_2(\Omega)}^2<\eta_1$ implies
\begin{align*}
& \|g_L\|^2_n-\|g_L\|^2_{L_2(\Omega)}/\text{Vol}(\Omega)\nonumber\\
\leq & 2\eta_1\|g\|_{L_2(\Omega)}^2/C-\|g_L\|_2^2/\text{Vol}(\Omega)+2C\delta_n^2/\text{Vol}(\Omega)^2\\
                        \leq & 2\eta_1 (s+1)^2\delta_n^2/C-(s-1)^2\delta_n^2/\text{Vol}(\Omega)+2C\delta_n^2/\text{Vol}(\Omega)^2\\
                        \leq & 2\eta_1 (s+1)^2\delta_n^2/C-(s-1)^2\delta_n^2/\text{Vol}(\Omega)+2C\delta_n^2/\text{Vol}(\Omega)^2.
\end{align*}
By choosing appropriate $C$ and $\eta_1$ (the choice only depends on $\text{Vol}(\Omega)$), we obtain
\begin{align}\label{lemmaA1eq1}
\|g_L\|^2_n-\|g_L\|^2_{L_2(\Omega)}/\text{Vol}(\Omega) & \leq -\frac{1}{2}(s-1)^2\delta_n^2/\text{Vol}(\Omega).
\end{align}
Note that
\begin{align}\label{lemmaA1eq2}
\bigg|\|g_L\|^2_n-\|g_L\|^2_{L_2(\Omega)}/\text{Vol}(\Omega)\bigg|\leq K^2
\end{align}
and
\begin{align}\label{lemmaA1eq3}
\mathbb{E}(g_L^2-\|g_L\|^2_{L_2(\Omega)}/\text{Vol}(\Omega))^2 & \leq 4K^2\|g_L\|^2_{L_2(\Omega)}/\text{Vol}(\Omega)\leq 4K^2(s+2)^2\delta^2_n/\text{Vol}(\Omega).
\end{align}
Combining \eqref{lemmaA1eq1}, \eqref{lemmaA1eq2} and \eqref{lemmaA1eq3} and Lemma \ref{Berforsingleg}, we have
\begin{align}\label{lemmaA1lastforGL}
& \mathbb{P}\bigg(\|g_L\|^2_{L_2(\Omega)}/\text{Vol}(\Omega)-\|g_L\|^2_n\ \geq \frac{1}{2\text{Vol}(\Omega)}(s-1)^2\delta^2_n\bigg) \nonumber\\
& \leq \exp\bigg[-\frac{n\frac{1}{8\text{Vol}(\Omega)^2}(s-1)^4\delta_n^4}{4K^2(s+2)^2\delta^2_n/\text{Vol}(\Omega)+K^2\frac{1}{6\text{Vol}(\Omega)}(s-1)^2\delta_n^2}\bigg]\nonumber\\
& \leq \exp\bigg[-\frac{n\frac{1}{8\text{Vol}(\Omega)}(s-1)^4\delta_n^4}{4K^2(s+2)^2\delta^2_n+K^2\frac{1}{6}(s-1)^2\delta_n^2}\bigg]\nonumber\\
& \leq \exp\bigg[-\frac{1}{8\text{Vol}(\Omega)}\frac{n(s-1)^2\delta_n^2}{36K^2+K^2\frac{1}{6}}\bigg]\nonumber\\
& \leq \exp\bigg[-\frac{1}{8\text{Vol}(\Omega)}\frac{n(s-1)^2\delta_n^2}{37K^2}\bigg]\nonumber\\
& \leq \exp\bigg[-\frac{n(s-1)^2\delta_n^2}{296\text{Vol}(\Omega)K^2}\bigg].
\end{align}
Therefore, taking all $g\in \mathcal{G}$ yields
\begin{align*}
& \mathbb{P}\bigg(\inf_{\|g\|_{L_2(\Omega)} \geq 2\delta_n} \frac{\|g\|^2_n}{\|g\|_{L_2(\Omega)}^2}<\eta_1\bigg) \nonumber\\
\leq & \sum_{s=2}^\infty \exp\bigg[H_B(\delta_n/\text{Vol}(\Omega),\mathcal{G}',\|\cdot\|_{L_\infty(\Omega)})-\frac{n(s-1)^2\delta_n^2}{300\text{Vol}(\Omega)K^2}\bigg].
\end{align*}
Since $H_B(\delta_n/\text{Vol}(\Omega),\mathcal{G}',\|\cdot\|_{L_\infty(\Omega)})\leq \frac{n\delta_n^2}{1200\text{Vol}(\Omega)K^2}$, it can be seen that
\begin{align*}
\mathbb{P}\bigg(\inf_{\|g\|_{L_2(\Omega)} \geq 2\delta_n} \frac{\|g\|^2_n}{\|g\|_{L_2(\Omega)}^2}<\eta_1\bigg)& \leq \sum_{s=2}^\infty \exp\bigg[-\frac{n(s-1)^2\delta_n^2}{1200\text{Vol}(\Omega)K^2}\bigg]\\
 & \leq C_1\exp(-C_2n\delta_n^2/K^2)
\end{align*}
for some constants $C_1$ and $C_2$ only related to $\text{Vol}(\Omega)$, which finishes the proof of the first part.

For $C_0 \leq \frac{1}{4\text{Vol}(\Omega)}$, it can be verified that
\begin{align*}
g^2 & \leq 2g_R^2/C_0+2C_0(g-g_R)^2 \leq 2g_R^2/C_0+2C_0\delta_n^2/\text{Vol}(\Omega)^2,
\end{align*}
which yields
\begin{align*}
\|g\|_n^2 \leq 2\|g_R\|^2_n/C_0+2C_0\delta_n^2/\text{Vol}(\Omega)^2.
\end{align*}
The inequality $\|g\|^2_n/\|g\|_{L_2(\Omega)}^2>\eta_2$ implies
\begin{align*}
\|g_R\|^2_n-\|g_R\|_{L_2(\Omega)}^2/\text{Vol}(\Omega) & \geq \frac{1}{2}\eta_2 C_0 s^2\delta_n^2-\|g_R\|_{L_2(\Omega)}^2/\text{Vol}(\Omega)-C^2_4\delta_n^2/\text{Vol}(\Omega)^2\\
& \geq \frac{1}{2}\eta_2 C_0 s^2\delta_n^2-(s-1)^2\delta^2_n/\text{Vol}(\Omega)-C^2_0\delta_n^2/\text{Vol}(\Omega)^2.
\end{align*}
By choosing appropriate $C_0$ and $\eta_2$, we have
\begin{align}\label{lemmaA1gUeq1}
\|g_R\|^2_n-\|g_R\|_{L_2(\Omega)}^2/\text{Vol}(\Omega) & \geq \frac{1}{4}(s-1)^2\delta_n^2/\text{Vol}(\Omega).
\end{align}
Note that
\begin{align}\label{lemmaA1gUeq2}
\bigg|\|g_R\|^2_n-\|g_R\|^2_{L_2(\Omega)}/\text{Vol}(\Omega)\bigg|\leq K^2
\end{align}
and
\begin{align}\label{lemmaA1gUeq3}
\mathbb{E}(g_R^2-\|g_R\|_{L_2(\Omega)}^2/\text{Vol}(\Omega))^2 & \leq 4K^2\|g_R\|_{L_2(\Omega)}^2/\text{Vol}(\Omega)\leq 4K^2(s+2)^2\delta^2_n/\text{Vol}(\Omega).
\end{align}
By combining \eqref{lemmaA1gUeq1}, \eqref{lemmaA1gUeq2} and \eqref{lemmaA1gUeq3} and Lemma \ref{Berforsingleg}, similar to \eqref{lemmaA1lastforGL}, we obtain
\begin{align*}
\mathbb{P}\bigg(\frac{\|g_R\|_{L_2(\Omega)}^2}{\text{Vol}(\Omega)}-\|g_R\|^2_n \geq \frac{1}{4\text{Vol}(\Omega)}(s-1)^2\delta_n^2\bigg)
& \leq \exp\bigg[-\frac{n(s-1)^2\delta_n^2}{600\text{Vol}(\Omega)K^2}\bigg].
\end{align*}
Taking all $g\in \mathcal{G}$ leads to
\begin{align*}
& \mathbb{P}\bigg(\inf_{\|g\|_{L_2(\Omega)} \geq 2\delta_n} \frac{\|g\|^2_n}{\|g\|_{L_2(\Omega)}^2}<\eta_2\bigg)\nonumber\\
\leq & \sum_{s=2}^\infty \exp\bigg[H_B(\delta_n/\text{Vol}(\Omega),\mathcal{G}',\|\cdot\|_{L_\infty(\Omega)})-\frac{n(s-1)^2\delta_n^2}{600\text{Vol}(\Omega)K^2}\bigg].
\end{align*}
Since $H_B(\delta_n/\text{Vol}(\Omega),\mathcal{G}',\|\cdot \|_{L_\infty(\Omega)})\leq \frac{n\delta_n^2}{1200\text{Vol}(\Omega)K^2}$, we have
\begin{align*}
\mathbb{P}\bigg(\inf_{\|g\|_{L_2(\Omega)} \geq 2\delta_n} \frac{\|g\|^2_n}{\|g\|_{L_2(\Omega)}^2}<\eta_2 \bigg) & \leq \sum_{s=2}^\infty \exp\bigg[-\frac{n(s-1)^2\delta_n^2}{1200\text{Vol}(\Omega)K^2}\bigg]\\
& \leq C_3\exp(-C_4\delta_n^2/K^2)
\end{align*}
for some constants $C_3$ and $C_4$ related to $\text{Vol}(\Omega)$, which finishes the proof of the second part.

\section{Proof of Lemma \ref{lemmaofnonpara}}\label{pflemmaD5}
Notice that
\begin{align}\label{transOf2term}
\frac{\mu}{n} Y^T(R + \mu I)^{-1}Y = \min_{\hat{Y}}\frac{1}{n}\sum_{i=1}^n (y_i-\hat{y}_i)^2+\frac{\mu}{n}\hat{Y}^T R^{-1}\hat{Y},
\end{align}
which can be verified by taking minimization of the objective function inside the right-hand side of (\ref{transOf2term}). Let $\hat{u}=R^{-1}\hat{Y}$. By plugging $\hat u$ into the right-hand side of (\ref{transOf2term}), we have
\begin{align}\label{eq:transu}
\min_{\hat{Y}}\frac{1}{n}\sum_{i=1}^n (y_i-\hat{y}_i)^2+\frac{\mu}{n}\hat{Y}^T R^{-1}\hat{Y} = \min_{\hat{u}} \frac{1}{n}(Y-R \hat{u})^T(Y-R\hat{u})+\frac{\mu}{n}\hat{u}^T R \hat{u}.
\end{align}
Therefore, by the representer theorem and \eqref{recovering}, the right-hand side of (\ref{eq:transu}) is the same as
\begin{align}\label{optGeGoal}
& \min_{\hat f\in \mathcal{N}_{\Psi}(\Omega)}\bigg( \frac{1}{n}\sum_{i=1}^n(y_i-\hat f(x_i))^2 + \frac{\mu}{n} \|\hat f\|^2_{\mathcal{N}_{\Psi}(\Omega)}\bigg).
\end{align}
Notice the objective function in \eqref{optGeGoal} can be written as 
\begin{align}\label{optGeGoalreX}
& \frac{1}{n}\sum_{i=1}^n(y_i-\hat f(x_i))^2 + \frac{\mu}{n} \|\hat f\|^2_{\mathcal{N}_{\Psi}(\Omega)}\nonumber\\
= & \frac{1}{n}\sum_{i=1}^n(f(x_i)-\hat f(x_i))^2 + \frac{\mu}{n} \|\hat f\|^2_{\mathcal{N}_{\Psi}(\Omega)} + \frac{2}{n}\sum_{i=1}^n \epsilon_i(f(x_i)-\hat f(x_i)) + \frac{1}{n}\sum_{i=1}^n \epsilon_i^2.
\end{align}
If $\mu^{-1} = O_P(n^{-\frac{d}{2\nu +d }})$, by Lemma \ref{lemmavandegeerf} and the proof of Lemma \ref{lemmavandegeerf}, it can be shown that \eqref{optGeGoalreX} converges to $\sigma_\epsilon^2$.

If $\mu = O_P(n^{\frac{d}{2\nu +d }})$, then for any $\hat f$, we have
\begin{align*}
\frac{1}{n}\sum_{i=1}^n(y_i-\hat f(x_i))^2 + \frac{\mu}{n} \|\hat f\|^2_{\mathcal{N}_{\Psi}(\Omega)}
\leq \frac{1}{n}\sum_{i=1}^n(y_i-\hat f(x_i))^2 + C n^{-\frac{2\nu}{2\nu +d}} \|\hat f\|^2_{\mathcal{N}_{\Psi}(\Omega)}.
\end{align*}
Applying the results in the case of $\mu^{-1} = O_P(n^{-\frac{d}{2\nu +d }})$, we obtain  
$$\frac{\mu}{n} Y^T(R + \mu I)^{-1}Y \leq 2\sigma_\epsilon^2,$$ with probability at least $1-C_1\exp(-C_2n^\eta)$. The lower bound can be obtained by
\begin{align*}
& \frac{1}{n}\sum_{i=1}^n(y_i-\hat f(x_i))^2 + \frac{\mu}{n} \|\hat f\|^2_{\mathcal{N}_{\Psi}(\Omega)}\nonumber\\
= & \frac{1}{n}\sum_{i=1}^n(f(x_i)-\hat f(x_i))^2 + \frac{\mu}{n} \|\hat f\|^2_{\mathcal{N}_{\Psi}(\Omega)} + \frac{2}{n}\sum_{i=1}^n \epsilon_i(f(x_i)-\hat f(x_i)) + \frac{1}{n}\sum_{i=1}^n \epsilon_i^2\nonumber\\
\geq & -\frac{1}{n}\sum_{i=1}^n(f(x_i)-\hat f(x_i))^2 + \frac{\mu}{n} \|\hat f\|^2_{\mathcal{N}_{\Psi}(\Omega)} + \frac{1}{2n}\sum_{i=1}^n \epsilon_i^2\nonumber\\
\geq & -\frac{1}{n}\sum_{i=1}^n(f(x_i)-\hat f(x_i))^2 - \frac{\mu}{n} \|\hat f\|^2_{\mathcal{N}_{\Psi}(\Omega)} + \frac{1}{2n}\sum_{i=1}^n \epsilon_i^2 \geq \sigma_\epsilon^2/4,
\end{align*}
with probability tending to one, where the first inequality is because of the Cauchy-Schwarz inequality. The last inequality is true because $\|f-\hat f\|_n^2$ and $\frac{\mu}{n} \|\hat f\|^2_{\mathcal{N}_{\Psi}(\Omega)}$ converge to zero \citep{geer2000empirical,gu2013smoothing}. This completes the proof.

\section{Proof of Lemma \ref{tracelemma}}\label{pflemmad6}
We only present the proof of the first inequality. The second inequality can be proved similarly. By direct calculation, it can be shown that
\begin{align*}
& \text{tr} ((A + B)(A + B + C)^{-1}) \geq \text{tr} (A(A + C)^{-1})\\
\Leftrightarrow &  \text{tr} (C(A + B + C)^{-1}) \leq \text{tr} (C(A + C)^{-1}) \\
\Leftrightarrow & \text{tr} (C (A + B + C)^{-1} B (A + C)^{-1}) \geq 0,
\end{align*}
which is true since  $A,B$ and $C$ are positive definite.

\section{Proof of Lemma \ref{lemmaerrorwnug}}\label{pflemma41}
Notice that at point $x$, for any $u = (u_1,...,u_n)^T \in \RR^n$,
\begin{align*}
    & (f(x) - \sum_{j=1}^n u_j f(x_j))^2\\
    = & \bigg|\frac{1}{(2\pi)^d} \int_{\RR^d} \bigg(\sum_{j=1}^n u_j e^{-ix_j^T\omega} - e^{-ix^T\omega}\bigg)\mathcal{F}(f)(\omega) d\omega  \bigg| ^2\\
    \leq & \frac{1}{(2\pi)^d}\int_{\RR^d} \bigg| \sum_{j=1}^n u_j e^{-ix_j^T\omega} - e^{-ix^T\omega}\bigg|^2 \mathcal{F}(\Psi)(\omega) d\omega  \frac{1}{(2\pi)^d}\int_{\RR^d} \frac{|\mathcal{F}(f)(\omega)|^2}{\mathcal{F}(\Psi)(\omega)} d\omega \\
    \leq & \bigg(\frac{1}{(2\pi)^d}\int_{\RR^d} \bigg| \sum_{j=1}^n u_j e^{-ix_j^T\omega} - e^{-ix^T\omega}\bigg|^2 \mathcal{F}(\Psi)(\omega) d\omega + \hat \mu_n \|u\|_2^2\bigg) \|f\|_{\mathcal{N}(\Omega)}^2\\
    = & \bigg(1 - 2\sum_{i=1}^n u_i\Psi(x - x_i) + \sum_{i=1}^n\sum_{j=1}^n u_i u_j\Psi(x_i - x_j) + \hat \mu_n \|u\|_2^2\bigg)\|f\|_{\mathcal{N}(\Omega)}^2,
\end{align*}
where the first inequality is because of the Cauchy-Schwarz inequality. Plugging $u = (R + \hat \mu_n I_n)^{-1}r(x)$ finishes the proof of the first part.

Now we prove the second part of this lemma.

Consider function $g(t) = \Psi(x - t)$. By the interpolation inequality, we have
\begin{align}\label{corub1}
& 1 - r(x)^T(R + \hat \mu_n I_n)^{-1}r(x) \leq  \|g(t) - r(x)^T(R + \hat \mu_n I_n)^{-1}r(t)\|_{L_\infty(\Omega)}\nonumber\\
\leq & C_1\|g(t) - r(x)^T(R + \hat \mu_n I_n)^{-1}r(t)\|_{L_2(\Omega)}^{1 - \frac{d}{2\nu}}\nonumber\\
& \times \|g(t) - r(x)^T(R + \hat \mu_n I_n)^{-1}r(t)\|_{\mathcal{N}_{\Psi}(\Omega)}^{ \frac{d}{2\nu}}.
\end{align}
By direct calculation, it can be seen that
\begin{align}\label{corub2}
\|g(t) - r(x)^T(R + \hat \mu_n I_n)^{-1}r(t)\|_{\mathcal{N}_{\Psi}(\Omega)}^2 \leq 1 - r(x)^T(R + \hat \mu_n I_n)^{-1}r(x).
\end{align}
Combining \eqref{corub1} and \eqref{corub2} leads to
\begin{align}\label{corub3}
1 - r(x)^T(R + \hat \mu_n I_n)^{-1}r(x) \leq & C_2 \|g(t) - r(x)^T(R + \hat \mu_n I_n)^{-1}r(t)\|_{L_2(\Omega)}^{\frac{4\nu - 2d}{4\nu - d}}.
\end{align}
Let $f_1(t) = r(x)^T(R + \hat \mu_n I_n)^{-1}r(t)$. It can be seen by the representer theorem that
\begin{align*}
f_1 = \argmin_{h \in \mathcal{N}_{\Psi}(\Omega)}\|g - h\|_n^2+\frac{\hat \mu_n}{n}\|h\|^2_{\mathcal{N}_{\Psi}(\Omega)}.
\end{align*}
Take $\delta_n = n^{-\frac{4\nu - d}{8\nu}}.$ Direct calculation shows that if $\alpha < \frac{1}{2}$, either $1 - r(x)^T(R + \hat \mu_n I_n)^{-1}r(x) \lesssim n^{(\alpha - 1)(1 - \frac{d}{2\nu})}$, or the conditions of Lemma \ref{lemmaratio1} hold. If $\alpha \geq \frac{1}{2}$, then either 
$\|g - f_1\|_2^2 \geq C_3 n^{2(\alpha-1)}$ for some constant $C_3$, which implies the conditions of Lemma \ref{lemmaratio1} hold, or 
\begin{align*}
    1 - r(x)^T(R + \hat \mu_n I_n)^{-1}r(x) \leq & C_2 \|g(t) - r(x)^T(R + \hat \mu_n I_n)^{-1}r(t)\|_2^{\frac{4\nu - 2d}{4\nu - d}}\\
    \leq & C_4 n^{(\alpha - 1) \frac{4\nu - 2d}{4\nu - d}}\lesssim n^{(\alpha - 1)(1 - \frac{d}{2\nu})}.
\end{align*}
Thus, combining these two cases, either the conditions of Lemma \ref{lemmaratio1} hold, or  $1 - r(x)^T(R + \hat \mu_n I_n)^{-1}r(x) \lesssim n^{(\alpha - 1)(1 - \frac{d}{2\nu})}$. If the conditions of Lemma \ref{lemmaratio1} holds, then with probability at least $1 - C_7\exp(-C_8n^{\eta_1})$,
\begin{align}\label{coroub51}
& \|g - f_1\|_{L_2(\Omega)}^2\nonumber\\
\leq & \eta_2 \|g - f_1\|_n^2\nonumber\\
= & \eta_2(\|g - f_1\|_n^2+\frac{\hat \mu_n}{n}\| f_1\|^2_{\mathcal{N}_{\Psi}(\Omega)}- \frac{\hat \mu_n}{n}\| f_1\|^2_{\mathcal{N}_{\Psi}(\Omega)})\nonumber\\
\leq & \eta_2\frac{\hat \mu_n}{n}(1 - r(x)^T(R + \hat \mu_n I_n)^{-1}R(R + \hat \mu_n I_n)^{-1}r(x))\nonumber\\
= & \eta_2\frac{\hat \mu_n}{n}(1- r(x)^T(R + \hat \mu_n I_n)^{-1}r(x)\nonumber\\
& + r(x)^T(R + \hat \mu_n I_n)^{-1}r(x)- r(x)^T(R + \hat \mu_n I_n)^{-1}R(R + \hat \mu_n I_n)^{-1}r(x))\nonumber\\
= & \eta_2\frac{\hat \mu_n}{n}(1- r(x)^T(R + \hat \mu_n I_n)^{-1}r(x) + \hat \mu_nr(x)^T(R + \hat \mu_n I_n)^{-2}r(x))\nonumber\\
\leq & 2\eta_2\frac{\hat \mu_n}{n}(1- r(x)^T(R + \hat \mu_n I_n)^{-1}r(x)),
\end{align}
where the last inequality is by \eqref{u2smallthan}. Plugging \eqref{coroub51} into \eqref{corub3} yields
\begin{align*}
    1- r(x)^T(R + \hat \mu_n I_n)^{-1}r(x) \lesssim n^{(\alpha - 1)(1 - \frac{d}{2\nu})},
\end{align*}
which finishes the proof. 

\section{Properties of eigenvalues}\label{propertiesofEigens}
Lemma \ref{lemDecayEig} states the asymptotic rate of the eigenvalues of $\Psi(\cdot-\cdot)$. Lemma \ref{boundsofdeterminant} states the minimum eigenvalue of $\Psi_1^T\Psi_1$, where $\Psi_1$ will be defined later.

Since $\Psi(\cdot-\cdot)$ is a positive definite function, by Mercer's theorem, there exists a countable set of positive eigenvalues $\lambda_1\geq \lambda_2\geq...>0$ and an orthonormal basis for $L_2(\Omega)$ $\{\varphi_k\}_{k\in\mathbb{N}}$ such that
\begin{align}\label{eq:AppJ1eq1}
\Psi(x - y) = \sum_{k=1}^\infty \lambda_k \varphi_k(x)\varphi_k(y),
\end{align}
where the summation is uniformly and absolutely convergent.

\begin{lemma}\label{lemDecayEig}
Let $\lambda_k$ be as in (\ref{eq:AppJ1eq1}). Then, $\lambda_k\asymp k^{-2\nu/d}$.
\end{lemma}
\begin{proof}
Let $T$ be the embedding operator of $\mathcal{N}_{\Psi}(\Omega)$ into $L_2(\Omega)$, and $T^*$ be the adjoint of $T$. By Proposition 10.28 in \cite{wendland2004scattered},
\begin{align*}
T^*v(x) = \int_\Omega \Psi(x - y)v(y)dy,\qquad v\in L_2(\Omega), \qquad x\in \Omega.
\end{align*}
By Lemma \ref{coro1013}, $H^{\nu}(\Omega)$ coincide with $\mathcal{N}_{\Psi}(\Omega)$. By Theorem 5.7 in \cite{edmunds1987spectral}, $T$ and $T^*$ have the same singular values. By Theorem 5.10 in \cite{edmunds1987spectral}, for all $k\in \mathbb{N}$, $a_k(T)=\mu_k(T)$, where $a_k(T)$ denotes the approximation number for the embedding operator (as well as the integral operator), and $\mu_k$ denotes the singular value of $T$. By Theorem in Section 3.3.4 in \cite{edmunds2008function}, the embedding operator $T$ has approximation numbers satisfying
\begin{align}\label{ineqnAppNum}
C_3k^{-\nu/d}\leq a_k\leq C_4k^{-\nu/d}, \forall k\in \mathbb{N},
\end{align}
where $C_3$ and $C_4$ are two positive numbers. By Theorem 5.7 in \cite{edmunds1987spectral}, $T^*T\varphi_k=\mu_k^2\varphi_k$, and $T^*T\varphi_k=T^*\varphi_k=\lambda_k\varphi_k$, we have $\lambda_k=\mu_k^2$. By (\ref{ineqnAppNum}), $\lambda_k\asymp k^{-2\nu/d}$ holds.
\end{proof}

\begin{lemma}\label{boundsofdeterminant}
Define two matrices $\Psi_1=\frac{1}{\sqrt{n}}(\varphi_1(X),...,\varphi_{p_1}(X))$, and $\Psi_2=\frac{1}{\sqrt{n}}(\varphi_{p_1+1}(X),\varphi_{p_1+2}(X),...)$, where $\varphi_{k}$'s are eigenfunctions as in (\ref{eq:AppJ1eq1}) and $\varphi_k(X)=(\varphi_k(x_1),...,\varphi_k(x_n))^T$ for $k=1,2,...$. With probability at least $1-C_1 e^{-C_2n^{\eta_1}}$,
\begin{align*}
\lambda_{\min}(\Psi_1^T\Psi_1)\geq \eta,
\end{align*}
for any $\mu = O(\sqrt{n})$, where $p = \lfloor(n/\mu)^{d/(2\nu)}\rfloor$, $p_1 =  \min\{p, C_3n^{1/2} \}$, $\eta,\eta_1,C_1,C_2$ are positive constants.
\end{lemma}

\begin{proof}
Consider $u^T\Psi_1^T\Psi_1u$, where $u=(u_1,...,u_p)^T\in \mathbb{R}^p$ with $\|u\|_2=1$. Let $\mathcal{Q}=\{g:g=\sum_{i=1}^p u_i\varphi_i\}$. Since $\varphi_i$'s are orthonormal, $\|g\|_{L_2(\Omega)}=1$. For any $g\in \mathcal{Q}$, by Lemma \ref{lemDecayEig}, $\|g\|^2_{H^{\nu}(\Omega)}\leq \frac{C_1}{\lambda_{p_1}}\leq C_2 p_1^{2\nu/d}$.
By the interpolation inequality,
\begin{align*}
\|g\|_{L_\infty(\Omega)}  \leq C_3\|g\|^{\frac{d}{2\nu}}_{H^{\nu}(\Omega)}= C_4 p_1^{1/2}.
\end{align*}
We shall use Lemma \ref{lemmaratio1} to link $\|g\|_n$ to $\|g\|_{L_2(\Omega)}$. First we need to check the conditions of Lemma \ref{lemmaratio1} hold. Since $\|g\|_{L_2(\Omega)} = 1$, it suffices to check the entropy condition. Let $\rho = C_2^{1/2} p_1^{\nu/d}$. Consider class $\mathcal{Q}'=\{g:g=\frac{f}{\rho}, f\in\mathcal{Q}\}$. Since $\mathcal{Q}'\subset H^\nu(\Omega)$, there exists a constant $C_5$ such that
\begin{align*}
H_B(\delta_n/\text{Vol}(\Omega),\mathcal{F}',\|\cdot\|_{L_\infty(\Omega)})\leq C_5\bigg(\frac{1}{\delta_n}\bigg)^{d/\nu}.
\end{align*}
The entropy condition is satisfied if
\begin{align}\label{ineq:HBcond1}
n\delta_n^{2+d/\nu}/K^2 > C_6,
\end{align}
where $\delta_n=1/\rho$, $K\leq C_4 p_1^{1/2}/\rho$, and $C_6$ is some constant depending on $C_1$--$C_5$ and $\Omega$. By direct calculations, if $p_1 \leq C_7\sqrt{n}$ for some constant $C_7$, (\ref{ineq:HBcond1}) is satisfied. 

By Lemma \ref{lemmaratio1}, 
\begin{align}\label{72eq1}
u^T\Psi_1^T\Psi_1u & = \frac{1}{n}\sum_{i=1}^n \bigg(\sum_{j=1}^p u_j\varphi_j(X_i)\bigg)^2= \|g\|_n^2\geq \eta,
\end{align}
with probability at least $1-C_8\exp(-C_9n^{\eta_1})$ for some constant $\eta$ and $\eta_1$. This finishes the proof.
\end{proof}

\end{appendix}
\bibliographystyle{apalike}
\bibliography{paperref}

\end{document}